\documentclass[12pt]{amsart}
\usepackage{graphicx}
\usepackage{amssymb}
\usepackage{amsfonts}
\usepackage{hyperref}
\usepackage{mathrsfs}
\usepackage[margin=3cm]{geometry}

  [1995/08/01 v0.1 Double stroke roman fonts]

\def\ds@whichfont{dsrom}
\relax

\DeclareMathAlphabet{\mathds}{U}{\ds@whichfont}{m}{n}

\swapnumbers
\newtheorem{theorem}{Theorem}%[section]
\newtheorem{lemma}[theorem]{Lemma}
\newtheorem{corollary}[theorem]{Corollary}

\theoremstyle{definition}
\newtheorem{definition}[theorem]{Definition}
\newtheorem{assumption}[theorem]{Assumption}
\newtheorem{remark}[theorem]{Remark}
\newtheorem{example}[theorem]{Example}

\numberwithin{equation}{section}
\theoremstyle{plain}

\numberwithin{equation}{section} %% Comment out for sequentially-numbered
\numberwithin{figure}{section} %% Comment out for sequentially-numbered
\theoremstyle{plain}
\theoremstyle{plain}
\theoremstyle{remark}
\newtheorem*{acknowledgement*}{Acknowledgement}
\newcommand{\cA}{{\mathcal A}}
\newcommand{\cB}{{\mathcal B}}
\newcommand{\cC}{{\mathcal C}}

\newcommand{\cE}{{\mathcal E}}
\newcommand{\cF}{{\mathcal F}}

\newcommand{\cH}{{\mathcal H}}
\newcommand{\cI}{{\mathcal I}}

\newcommand{\cK}{{\mathcal K}}
\newcommand{\cL}{{\mathcal L}}
\newcommand{\cM}{{\mathcal M}}

\newcommand{\cR}{{\mathcal R}}

\newcommand{\cT}{{\mathcal T}}

\newcommand{\cX}{{\mathcal X}}
\newcommand{\cY}{{\mathcal Y}}

\newcommand{\te}{{\theta}}

\newcommand{\Om}{{\Omega}}
\newcommand{\om}{{\omega}}
\newcommand{\ve}{{\varepsilon}}
\newcommand{\del}{{\delta}}
\newcommand{\Del}{{\Delta}}
\newcommand{\gam}{{\gamma}}

\newcommand{\sig}{{\sigma}}
\newcommand{\al}{{\alpha}}
\newcommand{\be}{{\beta}}
\newcommand{\ka}{{\kappa}}
\newcommand{\la}{{\lambda}}

\newcommand{\bbC}{{\mathbb C}}
\newcommand{\bbE}{{\mathbb E}}

\newcommand{\bbN}{{\mathbb N}}
\newcommand{\bbP}{{\mathbb P}}
\newcommand{\bbR}{{\mathbb R}}

\newcommand{\bbZ}{{\mathbb Z}}
\newcommand{\bbI}{{\mathbb I}}

\begin{document}
\title[]{Explicit conditions for the CLT and related results for non-uniformly partially expanding random dynamical systems via effective RPF rates}
 %\vskip 0.1cm
 \author{Yeor Hafouta \\
\vskip 0.1cm
Department of Mathematics\\
The Ohio State University and the University of Maryland}
%Jerusalem, Israel}%
%\address{
%Institute of Mathematics, The Hebrew University, Jerusalem 91904, Israel}
\email{yeor.hafouta@mail.huji.ac.il, yhafouta@umd.edu}
\date{\today}
\maketitle
\markboth{Y. Hafouta}{Limit theorems}
\renewcommand{\theequation}{\arabic{section}.\arabic{equation}}
\pagenumbering{arabic}

\begin{abstract}
The purpose of this paper is to provide a first class of explicit sufficient conditions for the central limit theorem and related results in the setup of non-uniformly (partially) expanding non iid random transformations, considered as stochastic processes together with some random Gibbs measure. More precisely, we prove a central limit theorem (CLT), an almost sure invariance principle, a moderate deviations principle, Berry-Esseen type estimates and a moderate local central limit theorem  for  random Birkhoff sums generated by a
non-uniformly partially expanding  dynamical systems $T_\om$ and a random Gibbs measure $\mu_\om$ corresponding to a random potential $\phi_\om$ with a
sufficiently regular variation.  In the  partially expanding case
the maps we consider  are similar to the ones in \cite{Varandas}, with the exception that  the amount of expansion dominates the amount of contraction fiberwise and not only on the average and with an additional regularity condition on a certain type of local variation of $\phi_\om$ along inverse branches of $T_\om$. A notable example when the maps are truly partially-expanding is the case  when $\phi_\om\equiv0$
 which corresponds to random measures of maximal entropy $\mu_\om$, but any potential with a sufficiently small (fiberwise) variation can be considered. Our results in the partially expanding case are new even in the uniformly random case, where all the random variables describing the maps are uniformly controlled. 
For properly expanding maps (as in \cite{MSU,HK}), the above local regularity condition allows applications also in the smooth case where the Gibbs measure is absolutely continuous with respect to the underlying  volume measure and $\phi_\om=-\ln J_{T_\om}$.
For instance, we can consider  certain fiberwise piecewise $C^2$-perturbations of piecewise linear or affine maps.
All of the above is achieved  by first proving random, real and complex, Ruelle-Perron-Frobenius (RPF) theorems with rates that can be expressed analytically by means of certain random parameters that describe the maps (such rates will be referred to as ``effective").   Using these effective rates, 
our conditions for the limit theorems involve some weak type of upper mixing conditions on the driving system (base map) and some integrability conditions on the norm of the random function generating the Birkhoff sums. A big part of the proof of the moderate deviations, the Berry-Esseen type estimates and the local CLT is to show how Rugh's  theory \cite{Rugh} of complex cones contractions applies to the cones considered in \cite{castro} (and their random versions in \cite{Varandas}), which is new even for deterministic dynamical systems $T$ and that case it yields explicit estimates on the spectral gap of appropriate deterministic  complex perturbations of the transfer operator of $T$, as well as explicit constants in the corresponding Berry-Esseen theorem for deterministic partially expanding dynamical systems.
\end{abstract}

\section{Introduction and an overview of the main results}\label{Section 1}
\subsection{Statistical properties of random Birkhoff sums and a preview}\label{Sec1.1}
Probabilistic limit theorems for expanding random dynamical systems have been studied extensively in the past decades. This setup includes a probability space $(\Om,\cF,\bbP)$ and a family locally expanding  maps $T_\om, \om\in\Om$ which are composed along an orbit of a probability preserving  invertible and ergodic map $\te:\Om\to\Om$ 
together with a family of equivariant\footnote{As will be discussed below, in applications $\{\mu_\om\}$ is not just any  equivariant family, but it is generated by an appropriate random potential (i.e. random Gibbs measures). In other situations $\mu_\om$ can be the appropriate volume measure.} probability measures $\mu_\om$ (i.e. $(T_\om)_*\mu_\om=\mu_{\te\om}$ for $\bbP$-a.a. $\om$) on the domain $\cE_\om$ of $T_\om$. When considering a random point $x_0$ distributed according to $\mu_\om$ we get
random orbits
$$
T_{\om}^n x_0=T_{\te^{n-1}\om}\circ\dots T_{\te\om}\circ T_\om x_0
$$
 and the question is whether for $\bbP$-almost all $\om\in\Om$ random Birkhoff sums of the form $S_n^\om=S_n^\om u(x_0)=\sum_{j=0}^{n-1}u_{\te^j\om}\circ T_\om^j(x_0)$ obey limit theorems like the quenched central limit theorem\footnote{Let us recall that the quenched CLT means that for $\bbP$-a.a. $\om$ the sequence of random variables $n^{-1/2}(S_n^\om-\bbE[S_n^\om])$ converges in distribution to a centered normal random variable with variance $\sig^2=\lim_{n\to\infty}\frac 1n\text{Var}(S_n^\om)$.} (CLT) and its variety of stronger versions.
  Here $u_\om$ is  random function on the domain of $T_\om$ satisfying some regularity conditions like H\"older continuity (not necessarily uniformly in $\om$).

\subsubsection{An illustrating example}\label{Eg Into}
In this section we will give an example of a random map $T_\om:[0,1)\to [0,1)$ which already captures the essence of the problem addressed in this manuscript. 
Let $a_\om$ be a random variable taking values in $[\frac 12,1)$. Let us consider the piecewise linear map $T_\om$ on $[0,1)$ so that on $[0,a_\om)$ the map $T_\om$ coicides with the linear function connecting the points $(0,0)$ and $(a_\om,1)$, while on $[a_\om,1)$ it coincides with the  linear function connecting the points $(a_\om,0)$ and $(1,1)$. Then $T_\om([0,1))= [0,1)$ and  $\gamma_\om:=a_\om^{-1}$ is the minimal amount of expansion of the map $T_\om$. 
When $a_\om$ is bounded away from $1$ then the maps $T_\om$ are uniformly expanding, and statistical properties (i.e. limit theorems) for random Birkhoff sums were extensively studied\footnote{See the next section for references.}  for random functions $u_\om$ with uniformly bounded H\"older norms $\|u_\om\|$ (here $\mu_\om=\text{Lebesgue}$). However, when $\text{esssup}_{\om\in\Om}(a_\om)=1$ the maps $T_\om$ are not uniformly expanding.
Focusing for the moment on this example, in this paper we will prove limit theorems in the case when $a_\om$ can take arbitrarily close to $1$ values (i.e. we can have $\text{esssup}_{\om\in\Om}(a_\om)=1$), and the random variable $\om\to \|u_\om\|$ belongs to $L^p$ for some $p>2$ (how small can $p$ be depends on the result, for the CLT we have sufficient conditions for every $p>2$).  Our results will be obtained under some (upper) mixing related assumptions on the sequence of random variables $\{a_{\te^j\om}, j\geq 0\}$. For instance, the quenched CLT holds when $(\Om,\cF,\bbP,\te)$ is the Markov shift corresponding to a sufficiently fast mixing Markov chain $X_n$ (e.g. geometrically ergdoic)  and $\beta(r)=:\|a_{\om}-\bbE[a_\om|X_{-r},...,X_{r}]\|_{L^1}$, $\om=(X_j)_{j\in\bbZ}$  decays polynomially  fast as $r\to \infty$. Note that when  $a_\om$ depends only on finitely many coordinates then $\beta(r)=0$ for $r$ large enough, while when it is a H\"older continuous function of $\om$ (in the case when the chain takes values on some metric space) then $\beta(r)$ decays exponentially fast (in this case we can take any $p>2$ above). We refer to Examples \ref{1}, \ref{2} and \ref{3} for more details.

\subsubsection{Back to the general setup}
In general, the  maps $T_\om$ can very often be described by means of random parameters $a_{i,\om}, i\leq d$ such as minimal amount of local expansion, degree, ``ratio" between contraction and expansion, local variation of the logarithm of the Jacobian, etc (in the example in Section \ref{Eg Into} we can take $d=1$ and $a_{1,\om}=a_\om)$.
We  call the maps uniformly random when the random variables $a_{i,\om}$ take values on appropriate domains (e.g.  minimal local expansion is bounded away from $1$, bounded degree, bounded variation, etc.). In the example in Section \ref{Eg Into} the uniform case corresponds to the case when $a_\om\leq 1-\epsilon$ for some $\epsilon\in(0,\frac12)$ and $\bbP$-a.a. $\om$.
Then most of the  statistical properties  in literature were obtained in the uniformly random case\footnote{See, for instance, the recent results \cite{Davor ASIP, HK, Davor TAMS, Davor CMP, HafYT, DH} and references therein.}, with the exception\footnote{See \cite{ANV} and references therein for results for iid maps when $u_\om$ does not depend on $\om$  and \cite{Su} for results for iid maps which admit a random tower extensions with sufficiently fast decaying tails.} of certain types of maps so that $\{T_{\te^j\om}: 0\leq j<\infty\}$ and $\{u_{\te^j\om}: 0\leq j<\infty\}$ are  iid processes on the probability space $(\Om,\cF,\bbP)$.
Beyond the iid case we are not aware of even a single explicit example with a true non-uniform behavior for which the quenched CLT holds true.

The purpose of this manuscript is to provide explicit sufficient conditions for several limit theorems like the CLT for non-uniformly random (partially) expanding maps (which will provide a variety of examples beyond the iid case).  Let us note that in \cite[Theorem 2.3]{Kifer 1998} an inducing strategy was developed in order to prove the CLT and related results in the non-uniformly random case. The conditions in \cite[Theorem 2.3]{Kifer 1998}  require certain type of regularity of the behavior  of the first visiting time to a 
measurable set $A\subset\Om$ with positive probability so that $\{T_{\om}, \om\in A\}$ are uniformly expanding in an appropriate sense.  While the results in \cite{Kifer 1998} were new even in the uniformly random case,  to the best of our knowledge, there are no examples in literature showing how to apply this method beyond the uniformly random case (where we can take $A=\Om$). Some of the proofs in this paper  will be based on applying the inducting strategy in \cite{Kifer 1998} in the non-uniformly random case (see Section \ref{Int Limit} for a more detailed discussion).

As will be explained in detail in Section \ref{Int Limit},
our conditions for the CLT and the functional law of iterated logarithm (LIL) will involve some mixing (weak-dependence) related conditions\footnote{These conditions will always hold true under appropriate restrictions on some upper mixing coefficients related to  $(\Om,\cF,\bbP,\te)$.} on  sequences $(f_n)$ of random variables of the form $f_n(\om)=f(a_{1,\te^n\om},...,a_{d,\te^n\om})$, where  $f$ has an explicit form, together with the integrability assumption $\|u_\om-\mu_\om(u_\om)\|_{\al}\in L^p(\Om,\cF,\bbP)$, $p>2$, where $\|\cdot\|_
\al$ is the standard H\"older norm corresponding to some exponent $\al$. 
For instance, when $T_\om$ is a piecewise monotone map on the unit interval with full images on each monotonicity interval, whose Jacobian has sufficiently regular\footnote{See more details in the paragraph below.} variation on each monotonicity interval\footnote{e.g. fiberwise sufficiently small $C^2$-perturbations of piecewise linear maps, see  Section \ref{Sec per lin}.}  and its minimal amount of expansion on  the monotonicity intervals is denoted by $\gamma_\om>1$,
 then our general conditions yield that the CLT holds true when $\|u_\om-\mu_\om(u_\om)\|_{\al}\in L^p(\Om,\cF,\bbP)$ and the sequence of random variables $(\gamma_{\te^n\om})_{n=0}^{\infty}$ on $(\Om,\cF,\bbP)$  satisfies some weak type of upper mixing condition. 
 Our condition for  the other limit theorems are similar (some  also require certain integrability conditions, see Section \ref{Int Limit}). We stress that in certain circumstances our mixing assumptions and the other conditions for the limit theorems are essentially independent. For instance, in the above examples the CLT will hold when the random variables $\gamma_\om$ satisfies conditions similar to the ones imposed on $a_\om$ in the end of Section \ref{Eg Into}.

\subsection{On the types of Gibbs measures and the smooth case} 
The equivariant measures $\mu_\om$  considered in this manuscript correspond to a piecewise H\"older continuous random potential $\phi_\om$ and have the property that they are absolutely continuous with respect to a random conformal measure $\nu_\om$, so that $(\cL_\om)^*\nu_{\te\om}=\la_\om\nu_\om$, where $\cL_\om g(x)=\sum_{y: T_\om y=x}e^{\phi_\om(y)}g(y)$ is the transfer operator of $T_\om$ and $\la_\om>0$ (namely $\mu_\om$ is the random Gibbs measure corresponding to $\phi_\om$, see \cite{MSU,Varandas}). 
We will call the case ``smooth"  if the domain of $T_\om$ is a smooth manifold and $e^{-\phi_\om}$ is the Jacobian of $T_\om$ with respect to the volume measure  on the domain (i.e. $\phi_\om=-\ln J_{T_\om}$). In this case we have $\la_\om=1$ and $\nu_\om$ is the volume measure $m_\om$, and so $\mu_\om$ is absolutely continuous with respect to $m_\om$.

Like in \cite{Varandas}, for partially expanding maps we impose a certain restriction on the oscillation of the underlying potential $\phi_\om$, and, in addition, we will impose a certain restriction on the H\"older constant of $\phi_\om$ along inverse branches of $T_\om$, which will be crucial for obtaining the effective RPF rates that will be discussed in the next section.
Like in \cite{Varandas}, because of the restriction on the oscillation, the results are less applicable in the smooth case, but includes applications for  random  measures $\mu_\om$ of maximal entropy (when $\phi_\om=0$) and in the, so-called, high temperature regime when $\phi_\om=\frac{1}{\be}\psi_\om$ for a sufficiently large $\be$ and a given random potential $\phi_\om$ satisfying some regularity conditions.

For properly expanding maps it is unnecessary to directly impose restrictions on the 
 oscillation of the potential $\phi_\om$, and we will only impose restrictions   on the H\"older constant of the compositions of $\phi_\om$ with the inverse branches of $T_\om$. In the smooth case discussed above, this conditions immediately allows applications to piecewise affine maps\footnote{We will also require that on each injectivity domain has a full image.} (where $\mu_\om$ is the Lebesgue measure), since then $\phi_\om$ is constant on each inverse branch. We will also show that the type of control needed over the local variation in this case is satisfied for fiberwise piecewise $C^2$-perturbations of such piecewise affine maps (see Section \ref{Sec per lin}), and so we provide several examples in the smooth case, as well.

\subsection{On our approach: effective RPF rates}
The first part of our approach is based on effective rates in the random version of the (normalized) Ruelle-Perron-Frobenius theorem (RPF), a notion which for the sake of convenience is presented here as a definition.
\begin{definition}[Effective random rates]
 Let $\phi_\om$ be a random potential whose supremum norm is bounded by some random variable $b_\om$ and its ``variation" (e.g. local H\"older constant) is bounded by a random variable $v_\om$. Let $\mu_\om$ be the random equivariant (i.e. $(T_\om)_*\mu_\om=\mu_{\te\om}$) Gibbs measure  corresponding to the potential.
  We say that the transfer operators\footnote{$L_\om$ is the dual of the Koopman operator $g\to g\circ T_\om$ with respect to the probability measure $\mu_\om$.} $L_\om$ of $T_\om$ (with respect to $\mu_\om$) have random effective rates when acting on a space of functions with ``bounded variation" if there are random variables $0<\rho(\om)<1$ and $B_\om\geq0$  which depend analytically only on the random parameters  $a_{1,\om},...,a_{d,\om}$ and on $b_\om$ and $v_\om$ so that $\bbP$-a.s. for every function $g$ on the domain of $T_\om$ with bounded variation (i.e. $\|g\|_{var}<\infty$) we have
\begin{equation}\label{Effective R}
\left\|L_\om^ng-\int gd\mu_\om\right\|_{var}\leq B_{\te^n\om}\rho_{\om,n}\|g\|_{var}
\end{equation}
where $L_\om^n=L_{\te^{n-1}\om}\circ\cdots\circ L_{\te\om}\circ L_\om$ and $\rho_{\om,n}=\prod_{j=0}^{n-1}\rho(\te^j\om)$.
\end{definition}
In this paper $\|g\|_{var}$ will always be the H\"older norm corresponding  to some exponent $\al$.
\begin{example}\label{Eff Eg}
In the example in Section \ref{Eg Into}, the measures $\mu_\om={Lebesgue}$ are equivariant, the corresponding potential is constant on the monotonicity interval of $T_\om$ and the operator $L_\om$ is defined by
$$
L_\om g(x)=a_\om g(a_\om x)+(1-a_\om)g(a_\om+(1-a_\om)x).
$$ 
For this example, if $\|\cdot\|_{var}$ denotes the H\"older norm corresponding to some exponent $\al\in(0,1]$ we obtain \eqref{Effective R} with 
$$
B_\om=24e^{4a_\om^{-\al}}(1+a_\om^{-\al})^2
$$ 
and 
$$
\rho(\om)=\frac{e^{\frac12a_{\te\om}^{-\al}}(1+a_{\te\om}^\al)-(1-a_{\te\om}^\al)}{e^{\frac12a_{\te\om}^{-\al}}(1+a_{\te\om}^\al)+(1-a_{\te\om}^\al)}.
$$
Note that as $a_\om\to 1$ the amount of contraction $\rho(\om)$ converges to $1$, which is expected since when  $a_\om=1$ we have $T_\om x=x$ and the maps $T_\om$ are no longer expanding. Observe also that in this example $B_\om$ is actually bounded. This will be the case also for more general classes of  random piecewise linear maps (and some of their perturbations), see Section \ref{Dis}.
\end{example}
We refer to Theorem \ref{RPF} for a more precise formulation of the effective RPF rates \eqref{Effective R} obtained in this paper. We also refer to Theorem \ref{Complex RPF} for effective rates for appropriate complex perturbations of the operators $L_\om$, which will be crucial for obtaining some of our results (see Section \ref{Complex Intro}).
As noted before, for partially expanding maps in the sense of \cite{Varandas} we obtain effective (real and complex) rates for potentials $\phi_\om$ with a sufficiently regular oscillation (which has applications for measures of maximal entropy and in the high temperature regime).
As we have already mentioned this condition limits the applications to the smooth case, but for properly expanding maps (in the sense of \cite{MSU, HK}) we will only require  that for each inverse branch $y_{i,\om}$ of $T_{\om}$ the H\"older constant of $\phi_\om\circ y_{i,\om}$ does not exceed $(\gamma_{\te\om}^\al-1)$ where $\gamma_\om$ is the minimal amount of local expansion of $T_\om$. This condition means that $\phi_\om\circ y_{i,\om}$ is close to being a constant when the map $T_{\te\om}$ has a small amount of expansion. As mentioned in the end of Section \ref{Sec1.1}, the latter condition about the H\"older constants is satisfied for appropriate types of perturbations of piecewise linear or affine maps (see Section \ref{Sec per lin}).

The proof of Theorem \ref{RPF} is  based on showing that the non-normalized transfer operator $\cL_\om$ of $T_\om$ contracts (w.r.t. the real Hilbert metric) an appropriate family of random cones $\cC_\om$ which are defined by means of the parameters $a_{i,\om}$ (for instance, for properly expanding maps $\cC_\om$ is defined only by means of $\gamma_\om$, where $\gamma_\om$ is the minimal amount of local expansion, which in the circumstances of Section \ref{Eg Into} satisfies $\gamma_\om=a_{\om}^{-1}$). This is the main difference here in comparison to many other applications of the contraction properties of real Hilbert metric for random operators (see \cite{Kif-RPF, Kifer thermo, MSU, HK, Varandas} and references therein), where the cones are usually defined by means of a random variable which can be expressed as a series of known random variables (but with  unclear integrability  or other regularity properties). As mentioned above, in the setup of \cite{MSU, HK} the price to pay for being able to use more explicit cones is an additional limitation on the variation of the potential $\phi_\om$ along inverse branches, while in the setup of \cite{Varandas} we will also require that the amount of expansion dominates the amount of contraction fiberwise and not only on the average.

We would like to think about $\rho(\om)$ in \eqref{Effective R} as the amount of contraction we have on the fiber $\om$.
We refer the readers to Remark \ref{Int cond U } for a discussion about situations where $B_\om$ is actually bounded.
For instance, for the aforementioned example of perturbations of piecewise affine maps (described in Section \ref{Sec per lin}), $B_\om$ is bounded if the H\"older constant of the logarithm of the Jacobian of $T_\om$ (on each inverse branch) is bounded\footnote{We refer to  Remarks \ref{M rem} and \ref{M rem 2} for a discussion about this matter and its relation to artificiality limiting the minimal amount of contraction by forcing it to be bounded above. This can always be done, but then stronger conditions on the potential are needed, which essentially reduce to the boundedness of the variation of the potential along inverse branches, which in the smooth case is the negative logarithm of the Jacobian.}, while when this H\"older constant  is not bounded, we have $B_\om\leq C \gamma_\om^\al e^{\gamma_\om^\al}$ where $\gamma_\om$ is the minimal amount of contraction. For more general expanding maps\footnote{More precisely, for the maps described in Remark \ref{Rem SFT}} we have $B_\om\leq C\gamma_\om^\al e^{4\sup|\phi_\om|+4\gamma_\om^\al}\left(\deg(T_\om)\deg(T_{\te^{-1}\om})\right)^2$, where $\deg(T_\om)$ is the maximal number of preimages that a point $x$ can have under $T_\om$.
Note that under certain types of mixing assumptions on $(\Om,\cF,\bbP,\te)$
 the precise form of $\rho(\om)$ does not make much difference, and only the fact that it is a function of the parameters $a_i(\om)$ plays a significant role (still, we refer to Sections \ref{Aux1} and \ref{Aux2} for the precise form), since we will work under assumptions guaranteeing that the sequences $(a_{i,\te^n\om})_{n=0}^\infty$ are sufficiently fast mixing. 

 \subsubsection{\textbf{A comparison with existing less explicit random RPF rates}}\label{Int RPF}
 Let us compare \eqref{Effective R} with a few other  random RPF rates in literature. For the maps considered in \cite{MSU, Kifer thermo} (see also references therein) and   \cite{Varandas} we have \eqref{Effective R} with
  a constant $\rho$ but with a random variable $B_\om$ which is defined by means of first hitting times to certain sets and a certain  random variable  $Q_\om$ which can be expressed as a series of known random variables (see, for instance, \cite[Lemma 3.18]{MSU}).  
   In fact, the proof of these results relies\footnote{When \eqref{Effective R} holds true with any kind of random variables $B$ and $\rho$, we can replace $\rho(\om)$ by a constant smaller than $1$ by considering the number of visits to a  set of the form $A_\ve=\{\rho(\om)<1-\ve\}$ for $\ve\in(0,1)$ small enough. However, this will make the ``new" $B_\om$ less explicit and with unclear regularity properties.} on obtaining rates of the form \eqref{Effective R} with random $\rho(\om)$ which depends on $Q_\om$ (see, for instance, \cite[Proposition 3.17]{MSU}). Note that even though $Q_\om$ has a closed ``formula" it is unclear which type of regularity conditions (e.g. integrability) it satisfies.
A similar phenomena happens also in the random RPF rates obtained in \cite{ABRV, Atnip 1, Atnip 2, Buzzi} (note that the third includes results for piecewise monotone interval maps without a random covering assumption), namely one can take $\rho(\om)=\rho$ to be a constant but with the price of making $B_\om$ less explicit. 
In any case,  since it is not clear which regularity properties $B_\om$ has in the above setups, it is less likely that these rates will be effective for proving limit theorems under explicit conditions, and not conditions involving some restrictions on the random variable $B_\om$ (which are hard to verify).
 
 Another less direct approach is based on an appropriate version of Oseledets theorem for the cocycle of transfer operators $\{L_\om\,:\,\om\in\Om\}$, and under certain logarithmic integrability conditions (see \cite[Proposition 26]{DS1} and references therein), it yields that
\begin{equation}\label{Var RPF}
\left\|L_\om^n g-\int gd\mu_\om\right\|_{var}\leq K(\om)e^{-n\la }\|g\|_{var}
\end{equation}
for some $\la>0$ and a tempered\footnote{Namely, almost surely we have $\lim_{n\to\infty}\frac1{n}K(\te^n\om)=0$.} random variable $K(\om)$. Notice that once \eqref{Var RPF} is established with some $\la$ then the minimal choice for $K(\om)$ is
$$
K(\om)=\sup_{n}\|L_\om^n-\mu_\om\|_{var}e^{n\la }.
$$
Remark also that 
$$
\la\leq \la_0(\om):=-\limsup_{n\to\infty}\frac 1n\ln\|L_\om^n-\mu_\om\|_{var},\,\,\bbP\text{-a.s.}
$$
and so, in a sense, $\la=\la^{(0)}:=\text{ess-inf }\la_0(\om)$
is the smallest possible choice for $\la$.
We note that when $\ln U$ is integrable then \eqref{Effective R} yields that 
$
\la^{(0)}\leq \bar\la=-\int \ln\rho(\om)d\bbP(\om)
$
which is a limitation on the contraction rate in the exponential convergence towards $\mu_\om$.
Even though we have the above  explicit form for $K(\om)$, it is unclear which type of regularity (beyond being tempered) the random variable $K(\om)$ possesses\footnote{e.g., whether it is in $L^p(\Om,\cF,\bbP)$ for some $p$ or if $(K(\te^n\om))$ satisfies some weak-dependence conditions.} or if it  has  a finite upper bound which depends (in a reasonable way) only the parameters $a_i(\om)$ describing the maps $T_\om$.
Under the conditions of \cite[Proposition 26]{DS1} in \cite{DS1, DHS} limit theorems were obtained in the smooth case for expanding on the average maps $T_\om$ and  random potentials $u_\om$ satisfying (roughly speaking) that $K(\om)\|u_\om\|_{var}\leq C$ for\footnote{Let us also note that in \cite[Appendix A]{DHS} it is demonstrated that, in general, scaling conditions of this form are necessary for the validity of certain limit theorems.} some constant $C$. 
In comparison with the smooth case considered in \cite{DS1, DHS},  we restrict ourselves to maps which have some fiberwise expansion (maybe not on the entire space) and not only expansion on the average. Moreover, we will have an additional assumption on the Jacobian (which will be satisfied for certain perturbations of piecewise affine maps) and certain type of upper mixing conditions on the system $(\Om,\cF,\bbP,\te)$ as well. On the other hand, as noted above, in general $K(\om)$ does not seem to be ``computable", and we also consider more general families of equivariant measures $\mu_\om$ corresponding to potentials with sufficiently regular variation (e.g. the maximal entropy and the high-temperature regime cases discussed above).

Let us also mention related results for (partially hyperbolic) iid maps $\{T_{\te^j\om}:\,j\geq0\}$  which admit a random (Young) tower extension (see \cite{BBM, DU, BBR, ABR}). In this setup  estimates of the form 
\begin{equation}\label{Rppf}
\left\|L_\om^n g-\int gd\mu_\om\right\|_{L^1(\mu_\om)}\leq K(\om)a_n\|g\|_{var}
\end{equation}
were obtained for some sequences $a_n\to0$  (the decay rate of $a_n$ is determined by the decay rates of the tails of the random tower) and a random variable $K(\om)$ which satisfies certain regularity conditions like $K(\om)\in L^{p}(\Om,\cF,\bbP)$ for some $p>1$. While also here $K(\om)$ does not seem to depend only on some parameters describing the original maps (or something similar), integrability conditions on  $K(\om)$ together with polynomial decay of $a_n$ are sufficient to get appropriate control over the non-uniform decay of correlations, which is enough to prove limit theorems  like an almost sure invariance principle (see \cite{Su}). However, this is obtained only for iid maps which admit a sufficiently regular random tower extension (and iid functions $\{u_{\te^j\om}:\,0\leq k<\infty\}$). Moreover, even for iid maps several other limits theorems like the ones described in Section \ref{Complex Intro} seem to require more than \eqref{Rppf}.

\subsection{A more detailed discussion on the proofs and conditions of the limit theorems}\label{Int Limit}
A major difficulty in proving limit theorems in the non-uniformly random case (beyond the iid case) is that the iterates of the annealed transfer operator (see \cite{ANV}) do not  describe the statistical behavior of the random Birkhoff sums, and due to strong dependence between $T_{\om},u_\om$ and $T_{\te\om},u_{\te\om}$ it seems less likely that a random tower extension  with sufficiently fast decaying tails exists (see again \cite{BBM, DU, BBR, ABR} and \cite{Su}). Instead, our results will rely on the effective rates \eqref{Effective R} described in the previous sections, as described in the following paragraphs.

We present two proofs of the central limit theorem (CLT) and the functional law of iterated logarithm (LIL). The first one (i.e. the proof of Theorem \ref{CLT}) is based on inducing, and more precisely we use the inducing strategy in \cite[Theorem 2.3]{Kifer 1998}. To the best of our knowledge, this is the first time that a result based on inducing in the $\Om$ direction is applied effectively for expanding maps like the ones considered in this paper (namely, that the required control over the system between two visiting times to the  set on the base $\Om$ is achieved). The idea in our proof is  that, using \eqref{Effective R}, the conditions of \cite[Theorem 2.3]{Kifer 1998} reduce to certain almost sure growth conditions which involve the random variables $\rho(\om), U_\om$ and $c_\om=\|u_\om-\mu_\om(u_\om)\|_{var}=\|u_\om-\mu_\om(u_\om)\|_{\al}$, which in turn can be verified under certain types of mild upper weak-dependence (mixing) assumptions on  the sequences\footnote{Recall that $\rho(\om)$ and $B\om$ are functions of $a_{i,\om}$ and so it is enough to impose upper weak-dependence conditions on $(a_{i,\te^n\om})_{n=0}^\infty$ for $i=1,2,...,d$.} $(\rho(\te^n\om))_{n=0}^\infty$  and $(B_{\te^n\om})_{n=0}^\infty$  and the integrability condition $c_\om\in L^p(\Om,\cF,\bbP), p>2$. We note that when $(\Om,\cF,\bbP,\te)$ is the shift system generated by a sufficiently fast mixing sequence (e.g. a geometrically ergodic Markov chain or some other exponentially fast mixing sequence) these mixing conditions will always hold true when we can approximate $\rho(\om)$ and $B_\om$ sufficiently fast by functions of $(X_{j})_{|j|\leq r}$ as $r\to\infty$ (in particular, when $\rho(\om)$ and $B_\om$  depend only on finitely many of the $X_j$'s). 
We stress that integrability conditions on $B_\om$ are not required and all that is needed is some type of upper mixing conditions and integrability assumptions on $c_\om$. In order to illustrate the above scheme, for the example in Section \ref{Eg Into} it will enough to induce on a set of the form $A=\{\om: \max(a_\om,a_{\te\om})\leq 1-\del\}$ for a sufficiently small $\del$ such that $\bbP(A)>0$ (note that some sufficient mixing conditions where already discussed in Section \ref{Eg Into}).

Our second proof of the CLT and LIL (namely, the proof of Theorem \ref{CLT2}) is not based on inducing, and instead it exploits \eqref{Effective R} directly and also requires that $B_\om\in L^p$ (as noted above, $B_\om$ is even bounded for a wide class of maps, see Remark \ref{Int cond U }). While in general integrability assumptions on $B_\om$ are  true additional requirements, the second type of sufficient conditions for the CLT has two advantages over the first set. First, it requires much weaker restrictions on certain upper mixing coefficients related to the system $(\Om,\cF,\bbP,\te)$. Second, it allows weaker approximation rates of $(\rho(\te^n\om))_{n=0}^\infty$ and $(B_{\te^n\om})_{n=0}^\infty$ than the ones required in the first set of conditions (in the case when the latter sequences can only be approximated by sequences satisfying some type of upper weak-dependence conditions). 
   
We also obtain an  almost sure invariance principle (ASIP), see Theorem \ref{ASIP}, which concerns strong approximation of the random Birkhoff sums by sums of independent Gaussian random variables (and is a stronger form of the CLT). Under some  upper weak-dependence assumptions  on  $(\rho(\te^n\om))_{n=0}^\infty$ and $(U(\te^n\om))_{n=0}^\infty$  and  some integrability conditions we obtain an ASIP with rates $o(n^{1/4+7p/2+\ve})$, where $p$ is the largest number so that the random variable $Y(\om)$ described in the last paragraph of Section \ref{Sec1.1} belongs to $L^p$.
For instance (see Remark \ref{asip Rem}), under certain regularity assumptions on the potential $\phi_\om$ our intergability conditions are $\|u_\om-\mu_\om(u_\om)\|_\al\in L^p(\Om,\cF,\bbP)$ and $N(\om)\in L^p(\Om,\cF,\bbP)$ where 
$N(\om)=\sup\{v_\al(g\circ T_\om): v_\al(g)\leq 1\}$. In the smooth case these conditions hold true for the aforementioned $C^2$-perturbations of the piecewise affine  maps (where here $N(\om)$ essentially coincides with the maximal amount of expansion of the map). In more general circumstances we also require that $B_\om\in L^p(\Om,\cF,\bbP)$ (in the situations discussed before it is bounded).

For non-uniformly random iid maps which admit a random tower extension an ASIP was obtained\footnote{The ASIP for uniformly random maps was treated in several papers in different setups, see \cite{Davor ASIP} and the references in \cite{Su, CIRM paper}.} in \cite{Su}, while for non uniformly random expanding maps driven by a general ergodic system $(\Om,\cF,\bbP,\te)$ it was obtained in \cite{CIRM paper} by inducing on an appropriate set $A$. The conditions in \cite{CIRM paper} reduce to certain assumptions on the behavior of the random Birkhoff sums $S_n^\om u$ when $n$ is smaller than the first visiting time $n_A(\om)$ of the orbit of $\om$ to $A$.  The proof of the ASIP in this paper is not based on inducing, and instead we apply \eqref{Effective R} directly, but we still think it could be interesting to check how \eqref{Effective R} can be combined with an inducing strategy in order to yield some ASIP rates. Finally, we would also like to refer to \cite{DHS}, where an ASIP was obtained under the scaling conditions described in the penultimate paragraph of Section \ref{Int RPF}.

\subsubsection{\textbf{Results which also require random complex effective rates, and the deterministic case}}\label{Complex Intro}
\,

In Theorems \ref{MDP1} and \ref{MDP2} we also derive a moderate deviations principle (MDP) which deals with the asymptotic behavior of probabilities of the form $\mu_\om\{(S_n^\om u-\mu_\om (S_n^\om u))/a_n\in\Gamma\}$ where $(a_n)$ is a certain type  of normalizing sequence  and $\Gamma\subset\bbR$ is an arbitrary Borel set (we refer to these results as moderate deviations\footnote{As opposed to large deviations.} because of the quadratic rate function involved in the formulation).
 These results are obtained under an additional condition on the potential $u_\om$ which, roughly speaking, means that either the H\"older norm $\|u_\om\|_\al$ is small when $T_\om$ has some inverse branch with a small amount of contraction, or that it is small when the ratio between the amount of expansion and contraction is close to $1$. Such a condition is close in spirit to the scaling conditions in \cite{DS1, DHS} discussed in the penultimate paragraph of Section \ref{Int RPF}, but the scaling is done according to the amount of expansion of $T_\om$.

Under the same additional requirement on the random functions $u_\om$, we will also obtain self-normalized CLT rates and a moderate version of the local CLT (see Theorems \ref{BE} and \ref{LLT}, respectively).  Our CLT rates are of order $n^{-(1/2-6/p)}$ when appropriate random variables (like the ones discussed in previous paragraphs) belong to $L^p(\Om,\cF,\bbP)$. When these random variables are bounded (i.e. in the uniformly random case) we have $p=\infty$ and this result recovers the Berry-Esseen theorem \cite[Theorem 7.1.1]{HK} (where the optimal $n^{-1/2}$ rates were obtained, see also \cite{DH, HafYT}). In \cite[Theorem 7.15]{HK}, in the uniformly random case, a local CLT was derived for the type of expanding maps considered in this paper (see also \cite{Davor CMP, Davor TAMS, DH, HafYT}), but the moderate type of local CLT considered here is in a difference scale,  it holds true without any additional aperiodicity conditions as in \cite{HK, Davor CMP, Davor TAMS, DH, HafYT} and it is new even in the uniformly random case. On the other hand, it provides local CLT estimates on a weaker scale.

The proofs of the MDP, the CLT rates and the moderate local CLT  require effective rates for appropriate complex perturbations of the transfer operators $L_\om$, which is established Theorem \ref{Complex RPF}. In fact, for partially expanding maps (as in Section \ref{Maps2}), Theorem \ref{Complex RPF} is new even in the uniformly random case (in that case $B_\om$ and $\rho(\om)$ are constants in the appropriate complex version). The proof of Theorem \ref{Complex RPF} uses  Rugh's theory \cite{Rugh} of the contraction properties of complex Hilbert metrics associated with  complex cones (see also \cite{Dub1, Dub2}). For uniformly random properly expanding maps $T_\om$ this method was applied successfully for random complex transfer operators  for the first time in \cite[Ch. 4-6]{HK}, and here we show how to apply it when the amount of contraction at the ``jump" from $\om$ to $\te\om$ depends on $\om$ (roughly speaking, the amount of contraction is $\rho(\om)$ appearing on \eqref{Effective R}), as well as for partially expanding maps (for such maps our results are new even in the uniformly random case).

Finally, let us note that for the partially expanding maps considered in this paper, the application of \cite{Rugh} is new even in the deterministic case (i.e. in the setup of \cite{castro}). This results in explicit bounds on the spectral gap of appropriate complex perturbations on the dual operator of the Koopman operator corresponding to the deterministic map $T$ and the underlying Gibbs measure. Using such estimates\footnote{The idea is that,  contrary to the classical perturbative approach  based on an appropriate implicit function theorem, we can control the size of  perturbation as well as obtain explicit bounds on the corresponding complex RPF triplets.} we can obtain, for instance, explicit constants in the Berry-Esseen theorem. That is, the methods used in this paper also make it possible to extend \cite[Theorem 1.1]{Dub2} from the properly expanding case to the partially expanding case, and we expect other similar quantitative results to follow. 

\section{Random expanding maps}
As mentioned in Section \ref{Section 1}, similarly to \cite{Varandas} we will consider partially random expanding maps  and random Gibbs measures corresponding to random potentials with sufficiently small random oscillation and a small H\"older constant along inverse branches.  However, when the maps are properly expanding (i.e. all local inverse branches are strongly contracting) we only need the condition about the H\"older constants, which will allow applications in the smooth case. For that reason we  begin the presentation in the setup of \cite[Ch. 6]{HK} (which is similar to \cite{MSU}) and only after that we will present the setup of partially expanding maps.

\subsection{Properly expanding maps with a local pairing property}\label{Maps1}

%present $\mu_\om$ before the limit theorems, most likely we need to add a description of known results or something similar...

\subsubsection{Random spaces and maps}
Our setup consists of a  probability space 
$(\Om,\cF,\bbP)$ together with an invertible ergodic $\bbP$-preserving transformation $\te:\Om\to\Om$,
of a compact metric space $(\cX,\rho)$ normalized in size so that 
$\text{diam}\cX\leq 1$ together with the Borel $\sig$-algebra $\cB$, and
of a set $\cE\subset\Om\times \cX$ measurable with respect to the product $\sig$-algebra $\cF\times\cB$ such that the fibers $\cE_\om=\{x\in \cX:\,(\om,x)\in\cE\},\,\om\in\Om$ are compact. The latter yields (see \cite{CV} Chapter III) that the mapping $\om\to\cE_\om$ is measurable with respect to the Borel $\sig$-algebra induced by the Hausdorff topology on the space $\cK(\cX)$ of compact subspaces of $\cX$ and the distance function $\rho(x,\cE_\om)$ is measurable in $\om$ for each $x\in \cX$. 
Furthermore, the projection map $\pi_\Om(\om,x)=\om$ on $\cE$ is
measurable and it maps any $\cF\times\cB$-measurable set to a
$\cF$-measurable set (see ``measurable projection" Theorem III.23 in \cite{CV}).
\begin{remark}
Compactness of either $\cX$ or $\cE_\om$ will  only be needed to insure the measurability  of $\om\to \cE_\om$ in the above sense. Thus, when $\cE_\om=\cX$ for every $\om$ then our results will remain valid for bounded metric spaces $\cX$ which are not necessarily  compact.
\end{remark}
Next, let
\[
\{T_\om: \cE_\om\to \cE_{\te\om},\, \om\in\Om\}
\]
be a collection of  maps between the metric spaces 
$\cE_\om$
and $\cE_{\te\om}$ so that
the map $(\om,x)\to T_\om x$ on $\cE$ is measurable with respect to the restriction of $\cF\times\cB$ to $\cE$.
%Consider the 
%skew product transformation $T:\cE\to\cE$ given by 
%\begin{equation}\label{Skew product}
%T(\om,x)=(\te\om,T_\om x).
%\end{equation}
For every $\om\in\Om$ and $n\in\bbN$  consider the $n$-th step iterates $T_\om^n$ given by
\begin{equation}\label{T om n}
T_\om^n=T_{\te^{n-1}\om}\circ\cdots\circ T_{\te\om}\circ T_\om: \cE_\om\to\cE_{\te^n\om}.
\end{equation}

Our first additional requirement from the maps $T_\om$ is that there is a random variable $\gamma_\om>1$ so that for every $x,x'\in\cE_{\te\om}$ we can write 
\begin{equation}\label{Pair1.0}
T_\om^{-1}\{x\}=\{y_i=y_{i,\om}(x): i<k\}\,\,\text{ and }\,\,T_\om^{-1}\{x'\}=\{y_i'=y_{i,\om}(x'): i<k\}
\end{equation}
and
\begin{equation}\label{Pair2.0}
\rho(y_i,y_i')\leq (\gam_\om)^{-1}\rho(x,x')
\end{equation}
for all  $1\leq i<k$ (where either $k\in\bbN$ or $k=\infty$).

\begin{remark}\label{Rem SFT}
We can also consider a somehow different setup where \eqref{Pair1.0} and \eqref{Pair2.0} hold true only for two points $x,x'$ such that $\rho(x,x')\leq\xi$ for some fixed constant $\xi<1$. In this case 
we would have to assume that $\deg(T_\om)<\infty$ (so $k$ above is finite) and that $\deg(T_\om)$  is measurable\footnote{we can assume that $\deg(T_\om)\leq d_\om$ for some random variable $d_\om$ instead.}. Here $\deg(T_\om)=\max\{|T_\om^{-1}\{x\}|:\,x\in\cE_{\te\om}\}$, where $|A|$ denotes the cardinality of a finite set $A$.
Moreover, we will need to require the following two additional covering assumptions:
 \vskip0.1cm
(i) there exists $n_0\in \bbN$ and  an $\xi$-cover of $\cE_\om$ by points $x_i=x_{i,\om}$ such that 
for all $i$ we have $T_\om^{n_0}\big(B_\om(x_i,\xi)\big)=\cE_{\te\om}$.
\vskip0.1cm
 \vskip0.1cm
(ii) for all $x,x'\in\cE_{\te\om}$, for every $y\in T_\om^{-1}\{x\}$ there exists $y\in T_\om^{-1}\{x'\}$  so that $\rho(y,y')\leq\xi$. 
 \vskip0.1cm

The main example we have in mind for such types of maps are certain classes of random sub-shifts of finite type (see, for instance \cite{Kifer thermo} but also \cite[Section 2.1]{MSU}). In this case we take $\xi$ small enough so that $\rho(x,x')\leq \xi$ means that the $0$-th coordinate of $x$ and $x'$ coincide. Note that for the maps presented before Remark \ref{Rem SFT} we can just take $\xi=1$, and in this case there is no need for the additional requirements (i) and (ii). In order to avoid a more complicated presentation of our main results we decided to focus on the case $\xi=1$, which already includes most of the (non-symbolic) examples we have in mind (e.g. the one in Section \ref{Eg Into}). 
\end{remark}

Next, for every $\om\in\Om$ and all $g:\cE_\om\to\bbC$ set
\begin{eqnarray*}
v(g)=v_{\al,\xi,\om}(g)=\inf\{R: |g(x)-g_\om(x')|\leq R\rho^\al(x,x')\,\text{ if }\,
\rho(x,x')<\xi\}\\
\text{and }\,\,\,\|g\|=\|g\|_{\al,\xi}=\|g\|_\infty+v_{\al,\xi}(g)\hskip1cm
\end{eqnarray*}
where $\|\cdot\|_\infty$ is the supremum norm and $\rho^\al(x,x')=\big(\rho(x,x')\big)^\al$ (and $\al$ is the same as in \eqref{phi cond}).
\begin{remark}\label{RR}
If $g:\cE\to\bbC$ is measurable and $g_\om:\cE_\om\to\bbC$ is given by $g_\om(x)=g(\om,x)$ then 
the function $\om\to\|g_\om\|$ is measurable by \cite[Lemma 5.1.3]{HK}. 
 \end{remark}
 Next, consider the Banach spaces $(\cH_\om,\|\cdot\|)=(\cH_\om^{\al,\xi}, \|\cdot\|_{\al,\xi})$
of all functions $h:\cE_\om\to\bbR$ such that $\|h\|_{\al,\xi}<\infty$ and
denote by $\cH_{\om,\bbC}=\cH_{\om,\bbC}^{\al,\xi}$ the space of all complex-valued functions with $\|h\|_{\al,\xi}<\infty$.

\subsubsection{The random potential}
Let $\phi:\cE\to\bbR$ be a measurable function so that $\|\phi(\om,\cdot)\|_{\infty}<\infty$. Fix some $\al\in(0,1]$, and suppose that for $\bbP$-a.e. $\om$ and let $H_\om$ be a random variable such that for all  $x$ and $x'$ in $\cE_{\te\om}$ for all $i$ we have
\begin{equation}\label{phi cond}
|\phi_\om(y_{i,\om}(x))-\phi_\om(y_{i,\om}(x'))|\leq H_\om \rho^\al(x,x')
\end{equation}
where $\phi_\om(x)=\phi(\om,x)$. In this paper we assume that
\begin{equation}\label{H cond}
H_\om\leq \gamma_{\te\om}^\al-1.
\end{equation}
This condition means that the H\"older constant of each composition $\phi_\om\circ y_{i,\om}$ is small when the ``next map" $T_{\te\om}$ has a small amount of contraction on some piece of the space. 
\begin{remark}
Condition \eqref{phi cond} holds true when $v(\phi_\om)\leq \gamma_\om^{\al}(\gamma_{\te\om}^\al-1)$. However, this
 condition is more general since it only imposes restrictions on the H\"older constant along the inverse branches of $T_\om$, and, as will be demonstrated in Section \ref{Sec per lin}, 
it allows applications in the smooth case. The idea is that for piecewise affine maps \eqref{phi cond} holds true with $H_\om=0$, and so \eqref{phi cond} and \eqref{H cond} will hold true for appropriate perturbations of such maps (see Section \ref{Sec per lin}). Condition \eqref{H cond} is also in force when 
$\phi_\om$ has the form $\frac{\psi_\om}{\gamma_{\om}^\al(\gamma_{\te\om}^\al-1)}$ where $\psi_\om$ has H\"older constant (corresponding to the exponent $\al$)  smaller than $1$ (see also Remark \ref{RR}).
\end{remark}
\begin{remark}
In principle, we can define 
\begin{equation}\label{H def 1}
H_\om=\sup_{i}v(\phi_\om\circ y_{i,\om})
\end{equation}
and assume that $H_\om\leq \gamma_{\te\om}^\al-1$. However, the function $\om\to H_\om$ might not be measurable because we did not assume that $\om\to \deg(T_\om)$ is measurable. 
\end{remark}

Next, we need following summability condition:
\begin{assumption}\label{SumAss}
There is a random variable $D_\om<\infty$  such that
$$
\sup_{x\in\cE_{\te\om}}\sum_{y\in T_\om^{-1}\{x\}}e^{\phi_\om(y)}\leq D_\om.
$$
\end{assumption}
Note that when $\deg(T_\om)$ is finite and measurable then this assumption trivially holds true $D_\om=\deg(T_\om)e^{\|\phi_\om\|_\infty}$ ($\|\phi_\om\|_\infty=\sup|\phi_\om|$ is measurable by \cite[Lemma 5.1.3]{HK}).

\begin{remark}\label{M rem}
The non-uniform expansion comes from the possibility that $\gamma_\om$ will be arbitrary close to $1$ 
(when $\gamma_\om$ is large then $T_\om$ is strongly expanding). Notice that conditions \eqref{Pair1.0} and \eqref{Pair2.0} remain valid if we replace $\gamma_\om$ with $\gam_{\om,M}=\min(M,\gamma_\om)$ for some constant $M>1$. While this limits the amount of expansion, some (but not all) of the conditions of the limit theorems that will be proven in this paper require intergability assumptions which are weaker when $\gamma_\om$ is bounded. On the other hand, forcing $\gamma_\om$ to be bounded by replacing it with $\gamma_{\om,M}$ essentially means that instead of \eqref{H cond} we require that $H_\om\leq (\gamma_{\te\om,M}^\al-1)$ for some $M>0$, and so there is a trade-off between the aforementioned integrability conditions and the latter stronger version of \eqref{H cond}.
\end{remark}

\subsubsection{An example and a comparison with \cite{MSU}}
One example of maps which satisfy our conditions  are piecewise injective maps.
In this case let $\cE_\om=\cX=[0,1)^d$ for some $d\in\bbN$,  and take a random partition of $\cX$ into rectangles of the form $[a_1,b_2)\times[a_2,b_2)\times\cdots\times[a_d,b_d)$. Now, on each rectangle we can take a distance expanding map which maps it onto $[0,1)^d$. Note that since $\cE_\om$ does not depend on $\om$ there is no need in compactness to insure its measurability. 

%When $\xi<1$ our main conditions are satisfied by the maps $T_\om$ considered in \cite{MSU,HK}, and we refer to \cite[Section 2.1]{MSU} for examples (see also \cite{Kifer thermo}).
%In comparison with \cite{MSU} we have two main additional conditions. The first  is Assumption \ref{Ass cov} (ii), which corresponds to taking $n_\xi(\om)=1$ in the topological exactness assumption \cite[(2.3)]{MSU}.  
The condition \eqref{H cond} is a restriction on the potential $\phi_\om$.  While we can always choose a potential which satisfies this condition, it is interesting to see when this setup applies to the smooth case when $\phi_\om=-\ln(\text{J}_{T_\om})$ and $\mu_\om$ is the unique absolutely continuous invariant measure w.r.t. the volume measure, and we refer to Section \ref{Sec per lin} for examples in the smooth case (fiberwise piecewise $C^2$ perturbations of certain piecewise linear or affine maps).

\subsection{Random maps with dominating expansion}\label{Maps2}
\subsubsection{Random  spaces  and maps}
Let $(\Om,\cF,\bbP,\te), (\cX,\rho), \{\cE_\om\}$ and $\{T_\om:\cE_\om\to\cE_{\te\om}\}$ satisfy the same properties described in the first paragraph  of Section \ref{Maps1}. In this section, our additional assumptions on the maps $T_\om$ are as follows:
we suppose that there exist random variables $l_\om\geq 1$, $\sig_\om>1$, $q_\om\in\bbN$ and $d_\om\in\bbN$ so that 
$q_\om<d_\om$ and for every $x\in\cE_{\te\om}$ we can write 
\begin{equation}\label{Pair1}
T_\om^{-1}\{x\}=\{y_{1,\om}(x),...,y_{d_\om,\om}(x)\}
\end{equation}
where for every $x,x'\in\cE_{\te\om}$ and for  $i=1,2,...,q_\om$ we have
\begin{equation}\label{Pair2}
\rho(y_{i,\om}(x),y_{i,\om}(x'))\leq l_\om\rho(x,x')
\end{equation}
while for  $i=q_{\om}+1,...,d_\om$,
\begin{equation}\label{Pair3}
\rho(x_i,x'_i)\leq \sig_\om^{-1}\rho(x,x').
\end{equation}
The above conditions are satisfied in the setup of \cite{Varandas} (see Section \ref{Dis} for a discussion).
We assume here that 
\begin{equation}\label{a om}
a_\om:=\frac{q_\om l_\om^\al+(d_\om-q_\om)\sig_\om^{-\al}}{d_\om}<1
\end{equation}
which is a quantitative estimate on the amount of allowed contraction, given the amount of expansion $T_\om$ has. 

Next, denote by $\cH_\om$ the space of  functions on $\cE_\om$ equipped with the norm 
\[
\|g\|=\|g\|_\infty+v(g)
\]
where $\|g\|_\infty=\sup|g|$ and $v(g)=v_\al(g)$ is the smallest number so that $|g(x)-g(y)|\leq v(g)\big(\rho(x,y)\big)^\al$ for all $x$ and $y$ in $\cE_\om$ (namely $\cH_\om=\cH_\om^{\al,1}$ in the notations of the previous section). In the case when $\al=1$ and each $\cE_\om$ is a Riemannian manifold we will also consider the norms $\|g\|=\|g\|_{C^1}=\sup|g|+\sup\|Dg\|$ on the space of $C^1$-functions, namely $v(g)$ above is replaced by the supremum norm of the deferential of $g$ (so in this case $v(g)$ could either be the Lipschitz constant or $\sup\|Dg\|$).

\subsubsection{The random potential}
Next, let $\phi:\cE\to\bbE$ be a measurable function and let $\phi_\om:\cE_\om\to\bbR$ be given by $\phi_\om(x)=\phi(\om,x)$. Set 
$$
\ve_\om=\text{osc}(\phi_\om)=\sup\phi_\om-\inf\phi_\om
$$
be the oscillation of $\phi_\om$
and for some fixed $\al\in[0,1)$ let
\begin{equation}\label{H def}
H_\om=\max\{v(\phi_\om\circ y_{i,\om}):\,1\leq i\leq d_\om\}
\end{equation}
be the maximal H\"older constant along inverse branches. We assume here that both $\ve_\om$ and $H_\om$ are finite. 
 Note that if $\phi_\om$ was H\"older continuous on the entire space $\cE_\om$ then $H_\om\leq l_\om v(\phi_\om)$.
Our additional requirements  from the function $\phi_\om$ is that
\begin{equation}\label{Bound ve}
s_\om:=e^{\ve_\om}a_\om<1\,\,\text{ and }\,\,e^{\ve_\om}H_\om\leq\frac{s_{\te\om}^{-1}-1}{1+s_\om^{-1}}
\end{equation}
\begin{remark}
The assumption about $H_\om$ is a  version of the combination of conditions \eqref{phi cond} and \eqref{H cond}.
Let $\del_\om$ be so that $(1+\del_\om)a_\om<1$ and suppose that $e^{\ve_\om}(1+\del_\om)a_\om<1$. Then the condition about $H_\om$ is satisfied when $H_\om\leq \frac{\del_{\te\om}a_\om^2}{1+a_\om}$. Note that we can always assume that $a_\om$ is bounded below by some positive constant (by replacing $a_\om$ with $\tilde a_\om=\max(a_\om,1-\ve)$ if needed). This will make no difference in our proofs, and in that case the second condition reads $H_\om\leq C\del_{\te\om}$ for some $C$ which can be arbitrarily close to $1$.
\end{remark}

\subsubsection{\textbf{A comparison with \cite{Varandas} and (additional) examples}}\label{Dis}
\subsubsection*{On the assumptions}Our assumptions  \eqref{Pair1}, \eqref{Pair2} and \eqref{Pair3} on the maps correspond to \cite[Assumption (H1)]{Varandas} (see also the proof of \cite[Proposition 5.4]{Varandas}). 
Our condition $s_\om<1$ is a stronger version of \cite[(2.2)]{Varandas} in \cite[Assumption (H3)]{Varandas}, which requires that $\int \ln s_\om d\bbP(\om)<0$ instead. In addition to this difference we also have the additional assumption about $H_\om$, which is an additional fiberwise restriction on the local H\"older constant of the potential on inverse branches. This condition always holds true when $\phi_\om$ is constant on each inverse branch.
On the other hand, in \cite{Varandas} there are several other  assumptions on the maps $T_\om$ like  \cite[Assumptions (H4) and (H5)]{Varandas} or \cite[Assumptions (H4) and (H5')]{Varandas} whose purpose is to prove uniqueness of the RPF triplets described in \cite[Theorem A]{Varandas} (see also Theorem \ref{RPF}). While it is natural to work under assumptions that guarantee uniqueness of RPF triplets (and equilibrium states), our result will not require such assumptions and all the limit theorems will hold for a certain type of random equilibrium state (Gibbs measure), which coincides with the unique one under the additional assumptions in \cite{Varandas}. 

\subsubsection*{Special choices of random potential}
Let us discuss two special types of potentials considered in \cite[Theorem D]{Varandas}. First, let us consider the case when  when $\phi_\om\equiv 0$. This case corresponds to equivariant measure $\mu_\om$ of maximal entropy (see \cite[Theorem D]{Varandas}), and in our case \eqref{Bound ve} holds true for that choice as long as $a_\om<1$. Again, the main difference in this case in comparison with \cite{Varandas} is that  the weaker assumption $\int \ln a_\om d\bbP(\om)<0$ was assumed instead.
Another special  choice for $\phi_\om$ is the case when $\phi_\om=\psi_\om/T$ for some other random potential and a sufficiently large constant $T$ (this  is  usually referred to as the high-temperature regime). In the high-temperature regime our results for general potentials are mostly effective in the uniformly random case\footnote{In the non-uniformly case we can consider $\psi_\om$ which satisfies \eqref{Bound ve} with $\ve_\om/T$ instead of $T$ and take $\phi_\om=\psi_\om/T$.}, but we note that in the setup of this section most of the results will be new even in then. 

\subsubsection*{Some additional examples of maps}
In \cite[Section 3]{Varandas} several examples were given, and in our setup we can consider the same examples replacing the assumptions about the integral of $\ln a_\om$ by almost sure assumptions on $a_\om$. For instance, let us consider a random finite partition of $[0,1)$ into intervals $I_{\om,i}=[a_{\om,i},b_{\om,i}), i\leq d_\om$. On each $i$ let us take a monotone H\"older continuous map $T_{\om,i}:I_{\om,i}\to[0,1)$ which IS onto $[0,1)$. Let us assume that the absolute value of the derivatives of $T_{\om,1},...,T_{\om,p_\om}$ is not less than $\sig_\om>1$, while the derivatives of the other $q_\om=d_\om-p_\om$ maps $T_{\om,i}$ does not exceed $l_\om^{-1}$ for some $l_\om\geq1$. Then all the conditions described before are valid if $a_\om<1$.
A particular case are the, so called, random Manneville–Pomeau maps. Let $\be(\om)\in(0,1)$ be a  random variable and let us take $I_{\om,1}=[0,\frac12)$ and $I_{\om,2}=[\frac12,1)$. On the first interval, let $T_{\om,1}(x)=x(1+(2x)^{\beta(\om)})$ while on the second we set $T_{\om,2}(x)=2x-1$. Then $q_\om=p_\om=1$, $\sig_\om=2$ and $l_\om=1$. In this case 
$$
s_\om=e^{\ve_\om}\frac{1+2^{-\al}}{2},\, H_\om=\max\left(v_\al(\phi_\om\circ T_{\om,1}^{-1}), v_\al(\phi_\om\circ T_{\om,2}^{-1})\right).
$$
Similar multidimensional examples can be given, for instance $\cI_{\om,i}$ can be a partition of $[0,1)^d=[0,1)\times\cdots\times [0,1)$  into rectangles with disjoint interiors of the form $I_{\om,i}=[a_{1,\om}^{(i)},b_{1,\om}^{(i)})\times\cdots\times[a_{d,\om}^{(i)},b_{d,\om}^{(i)})$ and on each rectangle we can take an injective  map whose image is $[0,1)^d$, and assume that some of the maps are distance expanding, while others might contract distance on some regions of the rectangle (e.g. we can start with affine maps and perturb).
We would  also like to refer to \cite[Section 3.4]{Varandas} for other multidimensional maps, which are included in our setup when the condition 
$$
\sup_{k}\left(\left(1-\frac{\ell_k}k\right)+\frac{\ell_k}{k}L_k\right)<1
$$
mentioned after \cite[(3.2)]{Varandas} is satisfied (and $\sig_k>1$ for all $k$). 
Note that there are additional requirements in \cite[Section 3.4]{Varandas}  but as mentioned above their purpose is to insure the uniqueness of the RPF triplets, which is not a requirement in this paper.

 \subsection{Frequently used random variables}
In this section we will introduce several random variables and the random cones that will be involved in the formulation of the effective Perron-Frobenius rates (Theorems \ref{RPF} and \ref{Complex RPF}), as well as in the conditions of the limit theorems.
\begin{remark}
From now on variables $x$ which involves $\te\om$ directly will be written in the form $x(\om)$, while we will use the notation $x_\om$ for a variable $x$ whose definition does not directly involve $\te\om$.
\end{remark}

\subsubsection{Properly expanding maps}\label{Aux1}
In the circumstances of Section \ref{Maps1}, 
let $\cC_\om$ be the real cone defined by 
$$
\cC_{\om}=\{g\in \cH_{\om}: g\geq0,\, g(x)\leq e^{\gamma_\om^\al \rho^\al(x,x')}g(x'),\,\,\forall\,\,x,x'\in\cE_\om\}.
$$

The following random variables are used in the formulation of Theorems \ref{CLT}, \ref{CLT2}, \ref{ASIP} (CLT and ASIP) and Theorems \ref{RPF} and \ref{Complex RPF} (Perron-Frobenius rates). 
\vskip0.2cm

%\begin{center}
\begin{tabular}{||c ||c ||} 
 \hline
$\text{\textbf{Random variable}}$ &\textbf{Role/comments}\\ [3ex] 
 \hline
$B_\om=24e^{4\gamma_\om^\al}(1+\gamma_\om^\al)^2$ & {\small appears on the RHS of \eqref{Effective R}}\\ [3ex] 
 \hline
 $q(\om)=\frac{H_\om+1}{\gamma_{\te\om}^\al}$ &   {\small $q(\om)\in(0,1)$ because of \eqref{H cond}} \\[2ex]
 \hline
 $D(\om)=\gamma_{\te\om}^\al+2\ln\left(\frac{1+q(\om)}{1-q(\om)}\right)$ &{\small bounds  projective  diameter, see Corollary \ref{Cor diam}}\\[3ex]
 \hline
$\rho(\om)=\tanh (D(\om)/4)$
 & {\small contraction rate as in the RHS of \eqref{Effective R}; $\rho(\om)\in(0,1)$} \\[3ex]
 \hline
  $B_{\om,1}=e^{\gamma_\om^\al}$
 & {\small a lower bound on the random equivariant density $h_\om$, see Theorem \ref{RPF}} \\[3ex]
 \hline
  $K_\om=(1+\gamma_\om^\al)e^{\gamma_\om^\al}$
 & {\small bounds aperture of $\cC_\om$, see Theorem \ref{Complex cones Thm0}; see also Theorem \ref{RPF}} \\[3ex]
  \hline
  $M_\om=8(1-e^{-\gamma_\om^\al})^{-2}$
 &  {\small bounds aperture of the dual of $\cC_\om$, see Theorem \ref{Complex cones Thm0}; $M_\om\leq 16$} \\[3ex]
%Let us also set $\tilde{\rho}(\om)=\tanh (7D(\om)/4)\in(0,1)$ which will serve as the one step contraction coefficient in the effective rates for the complex perturbations.
  \hline
  \hline
\end{tabular}
%\end{center}
\vskip 0.2cm

Note that the ``the projective diameter" refers to the diameter of the image $\cL_\om \cC_\om$ inside $\cC_{\te\om}$ with respect to the  (projective) Hilbert metric.

\begin{remark}\label{M rem 2}
As a continuation of Remark \ref{M rem}, when $H_\om\leq \min(M,\gamma_{\te\om})^\al-1$ for some $M>1$ (i.e. $H_\om\leq\gamma_{\te\om}^\al-1$ and it is bounded)   then 
we can replace $\gamma_{\om}$ with $\gamma_{\om,M}=\min(M,\gamma_\om)$ (namely assume that $\gamma_\om$ is bounded above). In this case we have 
$$
D(\om)\leq M+2\ln\left(\frac{1+q_M(\om)}{1-q_M(\om)}\right),\, B_\om\leq C(M)
$$
for some constant $C(M)$,
where $q_M(\om)$ is defined like $q(\om)$ but with $\gamma_{\om,M}$ instead of $\gamma_\om$.
\end{remark}
%\vskip0.2cm

%Then $\rho(\om)$ will serve as a ``one step contraction coefficient", and in our applications it will be the variable appearing on the right hand side of \eqref{Effective R}. 

%Next, let
%set $B_1(\om)=e^{\gamma_\om^\al}$, 
%$$
%K_\om=(1+\gamma_\om^\al)e^{\gamma_\om^\al}
%$$
%and define  $C_\om=4K_\om$, $U_\om=6B_1^2(\om)K_\om$ and $B_\om=C_\om U_\om$. 
%\begin{equation}\label{UUUU}
%B_\om=24e^{4\gamma_\om^\al}(1+\gamma_\om^\al)^2.
%\end{equation}
%In our applications this $B_\om$ will be the variable appearing on the right hand side of \eqref{Effective R}.
%When $\xi=1$ we 

Next, 
let  $u_\om\in\cH_\om$ be a random function so that 
\begin{equation}\label{tilde H cond}
  \tilde H_\om:=\gamma_\om^{-\al}v_\al(u_\om)+H_\om\leq\gamma_{\te\om}^\al-1  
\end{equation}
and an equivariant family of probability measures $\mu_\om$ (i.e. $(T_\om)_*\mu_\om=\mu_{\te\om}$). Note that \eqref{tilde H cond} is a stronger version of \eqref{H cond}. Let 
$$
\tilde u_\om=u_\om-\mu_\om(u_\om).
$$
In the circumstances of Section \ref{Maps1}, 
the following random variables are used in the formulation of Theorems \ref{MDP1} and \ref{MDP2} (moderate deviations principles), Theorem \ref{BE} (CLT rates) and Theorem \ref{LLT} (moderate local CLT).
\vskip0.2cm

%Z is now \tilde H
\vskip0.2cm

%\begin{center}
\begin{tabular}{||c||c ||} 
 \hline
$\text{\textbf{Random variable}}$ &\textbf{Role/comments}\\ [3ex] 
\hline
  $\tilde\rho(\om)=\tanh (7D(\om)/4)$
 & {\small complex contraction rate as in Theorem \ref{Complex RPF}$; \tilde\rho(\om)\in(0,1)$} \\[3ex]
 \hline
% $\tilde q(\om)=\frac{\tilde H_\om+1}{\gamma_{\te\om}^\al}$ & {\small  $q(\om)\in(0,1)$ because of \eqref{tilde H cond}} \\[2ex]
 %\hline
 %$\tilde D(\om)=\gamma_{\te\om}^\al+2\ln\left(\frac{1+\tilde q(\om)}{1-\tilde q(\om)}\right)$ & {\small see $E(\om)$ below}\\[3ex]
 %\hline
$c_0(\om)=3\|\tilde u_\om\|_\infty+\frac{\tilde H_\om}{\gamma_{\te\om}^{\al}-(1+\tilde H_\om)}$
 &  {\small $c_0(\om)>0$ by \eqref{tilde H cond}}\\[3ex]
  \hline
  $E(\om)=c_0(\om)\left(1+\cosh(7D(\om)/2)\right)$
 & {\small determines when transfer operators preserve complex cones} \\[3ex]
 \hline
 $\bar D_\om=16e^{\|\tilde u_\om\|_\infty}(1+v(u_\om))(1+\tilde H_\om)D_\om$ & {\small $D_\om$ comes from Assumption \ref{SumAss}}\\[3ex]
  \hline 
  \hline 
\end{tabular}
%\end{center}
\vskip0.5cm

\subsubsection{Partially expanding maps}\label{Aux2}
In the circumstances of Section \ref{Maps2}, consider the real cone
\[
\cC_{\om}=\cC_{\om,\ka_\om}=\{g\in\cH_\om:\,g>0\,\text{ and }\,v(g)\leq s_\om^{-1}\inf g\}.
\]
The following random variables are used in the formulation of Theorems \ref{CLT}, \ref{CLT2}, \ref{ASIP} (CLT and ASIP) and Theorems \ref{RPF} and \ref{Complex RPF} (Perron-Frobenius rates). 
\vskip0.4cm

%\begin{center}
\begin{tabular}{||c|| c ||} 
 \hline
$\text{\textbf{Random variable}}$ &\textbf{Role/comments}\\ [3ex] 
 \hline
$B_\om=12(1+2/s_\om)^4$ & {\small appears on the RHS of \eqref{Effective R}}\\ [3ex] 
 \hline
 $\zeta_\om=s_{\te\om}\left(1+(1+s_\om^{-1})e^{\ve_\om}H_\om\right)$  &   {\small $\zeta_\om<1$ by \eqref{Bound ve}} \\[2ex]
 \hline
 $D(\om)=2\ln\left(\frac{1+\zeta_\om}{1-\zeta_\om}\right)+2\ln\left(1+\zeta_\om s_\om^{-1}\right)$ &{\small bounds  projective  diameter, see Corollary \ref{Cor diam}}\\[3ex]
 \hline
$\rho(\om)=\tanh (D(\om)/4)$
 & {\small contraction rate as in the RHS of \eqref{Effective R}; $\rho(\om)\in(0,1)$} \\[3ex]
  \hline
   $B_{\om,1}=1+s_\om^{-1}$
 & {\small lower bound on random equivariant density, see Theorem \ref{RPF}} \\[3ex]
 \hline
  $K_\om=1+2s_\om^{-1}$
 & {\small bounds aperture of $\cC_\om$, see Theorems \ref{Complex cones Thm0}} and \ref{RPF} \\[3ex]
  \hline
  $M_\om=6s_\om^{-1}$
 &  {\small bounds aperture of the dual of $\cC_\om$, see Theorem \ref{Complex cones Thm0}} \\[3ex]
\hline
  \hline
\end{tabular}
%\end{center}

%  $C_\om=2K_\om$ and $U_\om=6B_1^2(\om)K_\om$. Set also 
 % \begin{equation}\label{UUUUU}
  %  U(\om)=C_\om U_\om
  %\end{equation}
  % and 
 % $M_\om=6s_\om^{-1}$.  
  \vskip0.2cm

  Finally, let $u_\om:\cE_\om\to\bbR$ be a random function and $\mu_\om$ be an equivariant family of probability measures $\mu_\om$. Set $\tilde u_\om=u_\om-\mu_\om(u_\om)$.
In the circumstances of Section \ref{Maps1}, 
the following random variables are used in the formulation of Theorems \ref{MDP1} and \ref{MDP2} (moderate deviations principles), Theorem \ref{BE} (CLT rates) and Theorem \ref{LLT} (local moderate CLT).
\vskip0.4cm
  
  %set 
  %$$
%E_\om=c_0(\om)\left(1+\cosh(D(\om)/2)\right)
%$$  
%where with $\tilde u_\om=u_\om-\mu_\om(u_\om)$,
%$$
%c_0(\om)=32s_{\te\om}(1+2s_\om^{-1})e^{\|\tilde u_\om\|_\infty+2\|\phi_\om\|_\infty} \|\tilde u_\om\|(1+H_\om)(1-\zeta_\om)^{-1}.
%$$
%Finally, let us also set
%$$
%\bar D_\om=M_{\te\om}e^{\|\tilde u_\om\|_\infty}(1+v(\tilde u_\om))(1+ %H_\om+l_\om^{\al})D_\om
%$$
%where $D_\om=\deg{T_\om}e^{\|\phi_\om\|_\infty}=d_\om e^{\|\phi_\om\|_\infty}$.

\begin{tabular}{||c||c ||} 
 \hline
$\text{\textbf{Random variable}}$ &\textbf{Role/comments}\\ [3ex] 
\hline
  $\tilde\rho(\om)=\tanh (7D(\om)/4)$
 & {\small complex contraction rate, see Theorem \ref{Complex RPF}$; \tilde\rho(\om)\in(0,1)$} \\[3ex]
 \hline
 $c_0(\om)=\frac{32s_{\te\om}(1+2s_\om^{-1})e^{\|\tilde u_\om\|_\infty+2\|\phi_\om\|_\infty} \|\tilde u_\om\|(1+H_\om)}{1-\zeta_\om}$ & {\small  $c_0(\om)\in(0,\infty)$ } \\[3ex]
 \hline
  \hline
  $E(\om)=c_0(\om)\left(1+\cosh(7D(\om)/2)\right)$
 & {\smaller determines when transfer operators preserve complex cones} \\[3ex]
 \hline
 $\bar D_\om=16e^{\|\tilde u_\om\|_\infty}(1+v(u_\om))(1+\tilde H_\om)D_\om$ &
 {\small $D_\om=\deg{T_\om}e^{\|\phi_\om\|_\infty}=d_\om e^{\|\phi_\om\|_\infty}$}\\[3ex]
  \hline 
  \hline 
\end{tabular}
%\end{center}
\vskip0.5cm

\subsection{More general (than Section \ref{Eg Into}) examples in the smooth case: fiberwise piecewise perturbations of piecewise linear expanding maps}\label{Sec per lin}
As discussed before, the maps described in Sections \ref{Maps1} and \ref{Maps2} were essentially considered in \cite{MSU,HK} and \cite{Varandas}, respectively, with the exception that in \cite{MSU,HK} the potential $\phi_\om$ did not satisfy \eqref{phi cond} and \eqref{H cond}, and in \cite{Varandas} the weaker condition $\int \ln s_\om d\bbP(\om)<0$ was considered (instead of $s_\om<1$), and the potential $\phi_\om$ did not  satisfy the second estimate in \eqref{Bound ve}, as well.
In comparison with  \cite{MSU}, the inequality \eqref{H cond} is our main additional assumption on the potential $\phi_\om$  in the setup of Section \ref{Maps1}. While we can always work with a random Gibbs measure $\mu_\om$ corresponding to a potential $\phi_\om$ satisfying \eqref{phi cond} and \eqref{H cond}, it is interesting to see for which maps these conditions hold true in the smooth case when $e^{\phi_\om}$ is the Jacobian of $T_\om$ with respect to the volume measure on $\cE_\om$. In this section  we will show that  \eqref{phi cond} and \eqref{H cond} are valid in the smooth case for certain type of $C^2$ fiberwise perturbations of piecewise linear maps (and similarly we can consider perturbations of piecewise affine maps, but for the sake of simplicity we will describe only the one dimensional case). 

\subsubsection{The piecewise linear case}
The examples in this section are generalizations of the illustrating examples in Section \ref{Eg Into}.
Let $\cI_{\om}=\{I_{\om,i}=[a_i(\om),b_i(\om))\}$ be a (nontrivial) partition of the unit interval $[0,1)$ into intervals, and on each interval let $\ell_{i,\om}$ be a linear map that maps $I_{\om,i}$ to $[0,1)$ (there are two options, either the decreasing one or the increasing one). Then the slope of $\ell_{\om,i}$ is $\pm|I_{\om,i}|^{-1}$, where $|I_{\om,i}|$ is the length of $I_{\om,i}$. Let us assume that $I_{\om,1}$ is the largest interval and set
$$
\gamma_\om(\ell)=|I_{\om,1}|^{-1}>1.
$$ 
Next, for each $i$ let $I_{\om,i_\om(y)}$ be the unique interval $I_{\om,i}$ so that $y\in I_{\om,i}$.
Then the map $\ell_\om$ defined by $\ell_\om(y)=\ell_{\om,i_\om(y)}(y)$ satisfies all the conditions in Section \ref{Maps1} in the case $\xi=1$ with $\gamma_\om=\gamma_\om(\ell)$. Moreover, if we consider the smooth case and take $e^{\phi_\om}$ to be the Jacobian of $\ell_\om$ then, since the map is piecewise linear we have that $H_\om$ in \eqref{phi cond} vanishes, and so \eqref{H cond} trivially holds true, where we can take the H\"older exponent $\al=1$. Moreover, we have that $\mu_\om$ is the Lebesgue measure and $\cL_\om \textbf{1}=\textbf{1}$ (and hence $D_\om=1$ in this case, and so $\bar D_\om=e^{\|u_\om\|_\infty}(1+v_\al(u_\om))$ depends only on $u_\om$). Recall also that $B_\om$ is bounded (see Remark \ref{M rem 2}) and note that $q(\om)=\frac{1}{\gamma_\om(\ell)\gamma_{\te\om}(\ell)}$ in this case. 

\subsubsection{Fiberwise piecewise $C^2$-perturbations}
Let us explain for which type of piecewise $C^2$-perturbations of $\ell_\om$ the conditions of Section \ref{Maps1} with $\xi=1$ remain true. On each interval $I_{\om,i}$, without changing the value of $\ell_{\om,i}$ at the end points, let us take a $C^2$ perturbation $T_{\om,i}$ of $\ell_{\om,i}$ so that 
\begin{equation}\label{omega wise}
\|T_{\om,i}-\ell_{\om,i}\|_{C^2}\leq \ve_\om=\frac12\min\left(\frac14(\gamma_{\te\om}(\ell)-1)\big(\gamma_{\om}(\ell)\big)^2,(\gamma_{\te\om}(\ell)-1),(\gamma_{\om}(\ell)-1)\right).
\end{equation}
Let us define $T_{\om}(y)=T_{\om,i_\om(y)}(y)$ (namely by gluing the maps $T_{\om,i}$).
Consider again the smooth case and take $e^{\phi_\om}$ to be the Jacobian of $T_\om$. Then $\mu_\om$ is equivalent to the Lebesgue measure (since $\nu_\om$ in Theorem \ref{RPF} is the Lebesgue measure).
\begin{lemma}
Under \eqref{omega wise} the maps satisfy the conditions in Section \ref{Maps1} with $\xi=1$ and $\gamma_\om=\gamma_\om(\ell)-\ve_\om\geq\frac{\gamma_\om(\ell)+1}{2}$. Moreover, the potential $\phi_\om=-\ln J_{T_\om}$ satisfies \eqref{phi cond} with $\al=1$ and $H_\om$ so that \eqref{H cond} holds true (with $\al=1$). 
\end{lemma}
\begin{proof}
First, it is clear that we can take $\gamma_\om=\gamma_\om(\ell)-\ve_\om\geq \frac{\gamma_\om(\ell)+1}{2}$.
In order to show that condition \eqref{H cond} is in force it is enough to show that  the derivative of each composition $\phi_\om\circ y_{\om,i}$ is bounded by 
$\gamma_{\te\om}-1$. 
To establish that, note that $y_{\om,i}(x)=T_{\om,i}^{-1}(x)$ and so 
$$
\sup_{x}\left|\left(\phi_\om\circ y_{\om,i}\right)'(x)\right|=\sup_{x}\left(|\phi_\om'(y_{\om,i}(x))|\cdot|y_{\om,i}'(x)|\right)
\leq \sup_{y}\sup_{i}\left|\frac{T_{\om,i}''(y)}{\big(T_{\om,i}'(y)\big)^2}\right|
\leq \frac{4\ve_\om}{\big(\gamma_{\om}(\ell)\big)^2}
$$
where in the last inequality we have used that $|T_{\om,i}''(y)|=|T_{\om,i}''(y)-\ell_{\om,i}''(y)|\leq \ve_\om$ and that 
$$
|T_{\om,i}'(y)|\geq |\ell_{\om,i}'(y)|-\ve_\om\geq |I_{\om,i}|^{-1}-\ve_\om\geq \gamma_{\om}(\ell)-\ve_\om\geq \frac12\gamma_{\om}(\ell).
$$
Finally, using that $\ve_\om\leq \frac18(\gamma_{\te\om}(\ell)-1)\big(\gamma_{\om}(\ell)\big)^2$ and that 
$\ve_\om\leq\frac12(\gamma_\om(\ell)-1)$ we see that
$$
\sup_{x}\left|\left(\phi_\om(y_{\om,i})\right)'(x)\right|\leq\frac{4\ve_\om}{\big(\gamma_{\om}(\ell)\big)^2} \leq \frac12(\gamma_{\te\om}(\ell)-1)\leq \gamma_{\te\om}(\ell)-\ve_\om-1\leq \gamma_{\te\om}-1.
$$
\end{proof}

\begin{remark}\label{C 2 REM}
As mentioned in Remark \ref{M rem 2}, when $H_\om$ is bounded then $B_\om$ from \eqref{Effective R} will be bounded in our applications. Notice that once \eqref{H cond} established $H_\om$ will be bounded if $\gamma_\om$ is. In our case this just means that $\gamma_{\om}(\ell)$ is bounded, namely that the number of intervals in the partition $\cI_\om$ is bounded. 
\end{remark}

\begin{remark}
In certain instances $|I_{\om,1}|$ is bounded away from $1$ (e.g. $T_\om x=(m_\om x)\text{ mod }1$, $m_\om\in \bbN$). 
In this case we can take allow larger $\ve_\om$ so that the resulting perturbation will not be uniformly expanding  (as $\gamma_\om$ could be arbitrary close to $1$).
\end{remark}

%Finally, let us state a CLT which will follow from our general results

\section{Preliminaries and main results}

\subsection{The random probability space: on the choice of measures $\mu_\om$}
For both classes of maps considered in Sections \ref{Maps1} and \ref{Maps2} let $\mu_\om$ be the Gibbs measures corresponding to the potential $\phi_\om$. The the detailed exposition of these measures is postponed to Section \ref{SecRPF} (where our results concerning effective rates are described), and for the meanwhile we refer to \cite{MSU} and \cite{Varandas} for the construction and the main  properties of these measures (see also Theorem \ref{RPF}). For instance they are equilibrium states and they have an exponential decay of correlations for H\"older continuous functions. Let us note that the smooth case discussed in Section \ref{Section 1} corresponds to the choice of $\phi_\om=-\ln(J_{T_\om})$ (see Section \ref{Sec per lin}), while the choice of $\phi_\om=0$ corresponds to random measures of maximal entropy (more generally the case $\phi_\om=\psi_\om/T$ for a sufficiently large $T$ and a sufficiently regular potential $\psi_\om$ corresponds to the high temperature regime, see \cite[Theorem D]{Varandas}).

\subsection{Upper mixing coefficients}
Let $X=\{X_j: j\in\bbZ\}$ be a stationary sequence of random variables (taking values on some measurable space) which generates the system  $(\Om,\cF,\bbP,\te)$, so that $\te$ is the left shift on the paths of $X_j$, namely $\te((X_j))=(X_{j+1})$.

Recall next that the $k$-th upper $\al,\phi$ and $\psi$ mixing coefficients of the sequence $\{X_j\}$ are the smallest numbers $\al_U(k),\phi_U(k)$ and $\psi_U(k)$ so that for every  $n$ and a  set $A$  measurable\footnote{Here $\sig\{X_j:\,j\in \cI\}$ is the $\sig$-algebra generated by $\{X_j: j\in \cI\}$, where $\cI\subset\bbZ$.} with respect to $\sig\{X_j: j\leq n\}$  and a set  $B$  measurable with respect to $\sig\{X_m: \,m\geq n+k\}$ we have
$$
\bbP(A\cap B)\leq \bbP(A)\bbP(B)(1+\psi_U(k)),
$$ 
\vskip0.1cm
$$
\bbP(A\cap B)\leq \bbP(A)\bbP(B)+\phi_U(k)\bbP(A)
$$ 
and 
$$
\bbP(A\cap B)\leq \bbP(A)\bbP(B)+\al_U(k).
$$
Clearly 
$$
\al_U(k)\leq\phi_U(k)\leq\psi_U(k).
$$
Notice next that $\al_U(k), \phi_U(k)$ and $\psi_U(k)$ are decreasing, and so
\begin{equation}\label{lim inf}
\limsup_{k\to\infty}\eta(k)=\lim_{k\to\infty}\eta(k)=\inf_k\eta(k)
\end{equation}
where $\eta$ is either $\al_U, \phi_U$ or $\psi_U$.

\begin{remark}
 Note that due to stationarity we can always consider only $n=0$ in the definitions of the upper mixing coefficients. We prefer to present the upper mixing coefficients in the above more general form in order to avoid repeating that both forms are equivalents in the course of the proofs.
\end{remark}
\begin{remark}[Two sided mixing coefficients]
Recall that the (two sided) mixing coefficients $\al(k),\phi(k),\psi(k)$ are defined similarly through the inequalities 
$$
\left|\bbP(A\cap B)-\bbP(A)\bbP(B)\right|\leq \bbP(A)\bbP(B)\psi(k),
$$
\vskip0.1cm
$$
\left|\bbP(A\cap B)-\bbP(A)\bbP(B)\right|\leq \bbP(A)\phi(k),
$$
and
$$
\left|\bbP(A\cap B)-\bbP(A)\bbP(B)\right|\leq \al(k).
$$
Clearly $\psi_U(k)\leq\psi(k)$, $\phi_U(k)\leq \phi(k)$ and $\al_U(k)\leq \al(k)$. The sequences $\al(k),\phi(k)$ and $\psi(k)$ are classical quantities  measuring the long range weak-dependence of the sequence $\{X_n\}$ (see \cite{BrMix,Douk}). 
\end{remark}

\begin{example}\label{MixEg}[(properly) mixing examples]
\begin{enumerate}

    \item[(i)]  $\phi(n)$ decays exponentially fast as $n\to\infty$ 
when $X_j$ is a geometrically ergodic Markov chain, namely if $R$ is the transition operator\footnote{Namley, $R$ maps a bounded function $g$ to a function $Rg$ given by $Rg(x)=\bbE[g(X_1)|X_0=x]$.} of the chain then 
$$
\|R^n-\mu\|_{\infty}\leq C\del^n
$$
for some $C>0$ and $\del\in(0,1)$ (this is a consequence of \cite[Theorem 3.3]{BrMix}). 
\vskip0.1cm

\item[(ii)] $\phi(n)$ also decays exponentially fast for uniformly contacting Markov chains in the sense of Dobrushin (see \cite[Lemma 3.3]{HafMix}), namely if for some $n_0\in\bbN$, 
 $$
 \sup_{x,y,\Gamma}\left|\bbP(X_{n_0}\in\Gamma|X_0=x)-\bbP(X_{n_0}\in\Gamma|X_0=y)\right|<1
 $$
 where $\Gamma$ ranges over all measurable subsets of the underlying state space and $(x,y)$ ranges over pairs of states. 
\vskip0.1cm

\item[(iii)]  $\psi(n)$ decays exponentially fast  as $n\to\infty$ when $X_j$ is a Markov chain satisfying the two sided Doeblin condition: there exists an $\ell>0$ such that for any measurable subset $\Gamma$ on the state space of $X_j$ and a state $x$ we have 
$$
C_1\eta(\Gamma)\leq \bbP(X_{\ell}\in\Gamma|X_0=x)\leq C_2\eta(\Gamma)
$$
for some constants $C_i>0$ and a probability measure $\eta$ (see \cite{BlHans}). 
\vskip0.1cm

\item[(iv)] $\psi(n)$ decays exponentially fast  when $X_j$ is the $j$-th coordinate of a topologically mixing sub-shift of finite type, see \cite{Bowen}. Therefore, it deacys exponentially fast also when $X_n$ is measurable with respect to $T^{-n}\cM$, where $T$ is an Anosov map and $\cM$ is a  Markov partition with a sufficiently small diameter (see also \cite{Bowen}).

\vskip0.1cm

\item[(v)]  $\psi(n)$ decays exponentially fast also when $X_j$ is the $j$-coordinate of the symbolic representation of a Gibbs Markov map, see \cite{AaDen} (like in (iv) this has an interpretation involving the Gibbs Markov map itself).
\vskip0.1cm

\item[(vi)] $\al(n)=O(n^{-(d-1)})$ when $X_j=T^jX_0$ and $T$ is a Young tower whose tails decay like $O(n^{-d})$ and $X_0$ is measurable with respect to the partition generating the tower (see the proof of \cite[Lemma 4]{HaydenYT} or \cite[Proposition 4.14]{HafYT}).
\vskip0.1cm

\item[(vii)] $\al(n)=O(n^{-(A-1)})$ if $(X_j)$ is a real valued Gaussian sequence such that $\text{Cov}(X_n,X_0)=O(n^{-A})$ for some $A>1$ (see \cite[ Section 2.1, Corollary 2]{Douk} and also \cite[Section 7]{BrMix}).

\end{enumerate}
We refer to  \cite{Douk} for additional examples.
\end{example}

While our results are new also if we work only with the usual mixing coefficients $\al,\phi,\psi$ presented above, the proofs only require assumptions on the upper mixing coefficients. In the following remark we will discuss a  situation  in which the upper mixing coefficients come in handy.

\begin{remark}
The only place where ergodicity of  $\te$ (i.e. of $\{X_j: j\in\bbZ\}$) is used is to insure that there is a number   $\sig\geq 0$ such that (like in Theorems \ref{CLT} and \ref{CLT2}),
$$
\sig^2=\lim_{n\to\infty}\frac1n\text{Var}(S_n^\om u),\,\,\bbP\text{ a.s. }
$$
 However, without ergodicity, by applying an appropriate  ergodic decomposition theorem we will get that the limit\footnote{i.e. the asymptotic variance.} 
$$
\sig^2(\om)=\lim_{n\to\infty}\frac1n\text{Var}(S_n^\om u)
$$
exists, but now it is no longer a constant, and instead it is measurable with respect to the $\sig$-algebra of invariant sets. Moreover, we will get the coboundary chracteriztion (as in Theorems \ref{CLT} and \ref{CLT2}) for the positivity of this limit, but on each ergodic component. Hence our results hold true if we start with a component for which the asymptotic variance $\sig^2(\om)$ is positive (when the asymptotic variance vanishes then the CLT is degenerate, namely $S_n^\om u/\sqrt n$ converges to $0$ in $L^2$, so there is nothing to prove).  

We conclude that in the non-ergodic case we can just assume that each ergodic component is mixing. For instance, the above modification of our results holds true when $X_j$ is a stationary finite state Markov chain whose transition matrix is composed of blocks which have only positive entries (perhaps after several steps). Indeed, each ergodicity class gives raise to an exponentially fast $\psi$ mixing sequence (see \cite{BrMix}), and thus $\psi_U(n)$ decays exponentially fast (using that for sets from different classes we have $\bbP(A\cap B)=0$). 
\end{remark}

\subsection{Quenched limit theorems for random Birkhoff sums}
Let $u_\om:\cE_\om\to\bbR$ be a random function (i.e. $u(\om,x)=u_\om(x)$ is measurable) so that $u_\om\in \cH_\om$ (i.e. $u_\om$ is $\al$-H\"older continuous).
Let us consider the corresponding random Birkhoff sums
$$
S_n^\om u=\sum_{j=0}^{n-1}u_{\te^j\om}\circ T_\om^j.
$$ 
In this paper, under appropriate assumptions on $u_\om$,
when $\om$ is fixed (chosen from a set of probability one),  we will prove limit theorems for the sequence of functions $S_n^\om u(\cdot)$  considered as random variables on the probability space $(\cE_\om,\cB_\om,\mu_\om)$, where $\cB_\om$ is the Borel $\sig$-algebra on $\cE_\om$.

\subsection{The CLT and LIL}
First, let us note that,  in order to avoid repetitions, in all the result formulated in this section the random variables defined in Sections \ref{Aux1} and \ref{Aux2} will be in constant use, sometime without referring to these sections.

Next, in order to formulate our first set of sufficient conditions for the CLT we consider the following assumption.

\begin{assumption}\label{Sets Ass} 
There is a measurable set $A\subset \Om$ with positive probability  so that for all $\om\in A$ we have $\rho(\om)\leq 1-\ve$ and $B_\om\leq M$ for some $\ve,M>0$ and that for all $r\in\bbN$ there is a set $A_r$ which is measurable with respect to $\sig\{X_{j}, |j|\leq r\}$ so that $\beta_r=\bbP(A\setminus A_r)\to 0$ and $\lim_{r\to\infty}\bbP(A_r)=\bbP(A)$. 
\end{assumption}

\begin{remark}\label{Rrim}
When $\rho(\om)$ and $B_\om$ depend only on $X_j, |j|\leq d$ for some $d$ then we can just take $A$ to be a set of the form $A=A_{\ve,M}=\{\om: \rho(\om)\leq 1-\ve,\, B_\om\leq M\}$ for a sufficiently small $\ve$ and a sufficiently large $M$ (or any measurable subset of such a set with positive probability). In this case $\be_r=0$ for all $r>d$. It is important to note that in these circumstances our conditions for the CLT (see Theorems \ref{CLT} and \ref{CLT2}) will only involve some  decay rates for $\al_U,\phi_{R,U}$ or $\psi_U$ and integrability assumptions on the norm $\|u_\om-\mu_\om(u_\om)\|$, which are independent of all the random variables describing the maps $T_\om$ (the general principle is that if $\om\to T_\om$ depends only on finitely many variables then so are the describing parameters, and then we have no additional requirements from these parameters). 
\end{remark}
\begin{example}\label{1}
Let us consider the example in Section \ref{Eg Into}. Then using the formulas for $\rho(\om)$ and $B_\om$ given in Example \ref{Eff Eg} we see that $\rho(\om)$ depends only on $a_{\te\om}$ and $B_\om$  depends only on $a_\om$. Thus, if $a_\om$ depends only on finitely many coordinates then both $\rho(\om)$ and $B_\om$ depend on finitely many coordinates and, as discussed in Remark \ref{Rrim}, Assumption \ref{Sets Ass} is in force.
\end{example}

To provide examples where Assumption \ref{Sets Ass} holds true with non-vanishing $\beta_r$ we will use the following simple result.

\begin{lemma}\label{Egg}
Suppose that the following approximation condition holds true: there are\footnote{Note that  condition \eqref{Appp} can also be written as
$$
\max(\|\rho-\bbE[\rho|X_{-r},...,X_r]\|_{L^1}, \|B-\bbE[B|X_{-r},...,X_r]\|_{L^1})\leq \del_r
$$ 
where $\rho=\rho(\om)$ and $B=B_\om$.
This condition is fulfilled when $\rho(...,X_{-1},X_0,X_1,...)$ and $B_{...,X_{-1},X_0,X_1,...}$ weakly depend (in an $L^1$-sense) on the coordinates $X_j, |j|\geq r$. Limit theorems under conditions similar to \eqref{Appp} (with some decay rate for $\del_r$) have been studied extensively in weak dependence theory, see \cite{Bil, IL}  (where iid $X_j$ are considered).} random variables $\rho_r=\rho_r(\om)$ and $B_r=B_{\om,r}$ measurable with respect to $\sig\{X_j, |j|\leq r\}$ so that
\begin{equation}\label{Appp}
\max(\|\rho(\om)-\rho_r(\om)\|_{L^1}, \|B_\om-B_{\om,r}\|_{L^1})\leq \del_r\to 0.
\end{equation}
Then Assumption \ref{Sets Ass} holds with $A=A_{\ve,M}=\{\om: \rho(\om)\leq 1-\ve,\, B_\om\leq M\}$, 
$A_r=\{\om: \rho_r(\om)\leq 1-\ve+\sqrt{\del_r},\, B_{\om,r}\leq M+\sqrt{\del_r}\}$
and $\beta_r\leq 2\sqrt{\del_r}$ (where $\ve$ is small enough and $M$ is large enough to insure that $\bbP(A)>0$).
\end{lemma}
\begin{proof}
Using the Markov inequality we see that
$$
\bbP(A_r)\leq \bbP(|\rho-\rho_r|\geq \sqrt{\del_r})+\bbP(|B-B_r|\geq \sqrt{\del_r})+\bbP(A_{\ve+2\sqrt{\del_r},M+2\sqrt{\del_r}})\leq  2\sqrt\del_r +\bbP(A_{\ve+2\sqrt{\del_r},M+2\sqrt{\del_r}})
$$
and
$$
\bbP(A\setminus A_r)\leq\bbP(|\rho-\rho_r|\geq \sqrt{\del_r})+\bbP(|B-B_r|\geq \sqrt{\del_r})\leq 2\sqrt\del_r.
$$
Hence $\lim_r\bbP(A_r)=\bbP(A)$ and $\beta_r=\bbP(A\setminus A_r)\leq 2\sqrt{\del_r}$.
\end{proof}

\begin{remark}
When $\om\to B_\om$ is in $L^1(\Om,\cF,\bbP)$ then there is always a sequence $\del_r$ satisfying \eqref{Appp}, but  our main results involve certain decay rates for $\beta_r$.
\end{remark}

\begin{example}\label{2}

  Let us return to the example in Section \ref{Eg Into}, but in addition 
  we suppose that the coordinates $\om_j$ of $\om$ take values on some bounded metric space $(Y,d_Y)$ (not necessarily compact). Let us define a metric on $\Om$ by 
$$
d_\Om(\om,\om')=\sum_{j}2^{-|j|}d_Y(\om_j,\om'_j).
$$
Then the left shift $\te$ is  Lipschitz continuous.
  
Let us assume  that $a_\om$ is a H\"older continuous function  of $\om$ with some exponent $\ka$.  As discussed in Example \ref{Eg Into}, we have
$$
B_\om=24e^{4a_\om^{-\al}}(1+a_\om^{-\al})^2
$$ 
and 
$$\rho(\om)=\frac{e^{\frac12a_{\te\om}^{-\al}}(1+a_{\te\om}^\al)-(1-a_{\te\om}^\al)}{e^{\frac12a_{\te\om}^{-\al}}(1+a_{\te\om}^\al)+(1-a_{\te\om}^\al)}.
$$
Since  $a_\om\in[\frac12,1)$ we see that $B_\om$ is a Lipschitz continuous function of $a_\om$ and $\rho(\om)$ is a Lipschitz continuous function of $a_{\te\om}$. 
Thus both $B_\om$ and $\rho(\om)$ are H\"older continuous continuous functions of $\om$ with the same exponent $\ka$.

We claim that condition \eqref{Appp} holds true with $\beta_r=O(2^{-\ka r/2})$ (in fact we will get (stronger) estimates in $L^\infty$). Note that such exponential rate of decay will be more than enough for the decay rates in Theorem \ref{CLT} (the CLT) to hold.
To prove the claim, observe that   for every  H\"older continuous function  $R:\Om\to\bbR$ with exponent $\ka\in(0,1]$ and all points $\om,\om'$ such that $\om_j=\om'_j$ if $|j|\leq d$ we have 
$$
|R(\om)-R(\om')|\leq  L_R D^\ka 2^{-\ka(d-1)}
$$
where $L_R$ is the H\"older constant of $R$ and $D=\text{Diam}(Y)$. Therefore,
$$
\sup_\om|R(\om)-R_r(\om)|=O(2^{-\ka r})
$$
where $R_r(\om)=\inf \{R(\om'): \om'_j=\om_j\,\text{ if }\,|j|\leq r\}$ (apply this with $R(\om)=B_\om$ and $R(\om)=\rho(\om)$).
Note that $R_r(\om)$ depends only on the coordinates $\om_j$ with $|j|\leq r$. Finally, Assumption \ref{Sets Ass} holds because of Lemma \ref{Egg}, and, moreover, we have $\be_r=O(2^{-\ka r/2})$,

\end{example}

\begin{example}\label{3}
In the context of Section \ref{Eg Into}, let us consider the case when
 $a_\om$ has the form 
$$
a_\om=\sum_{j=0}^\infty f_j(\om_{-j},...,\om_0,,...,\om_j)
$$
for some measurable functions $f_j$ on $\cY^{2j+1}$, where $\cY$ is the space such that $\Om=\cY^\bbZ$. 
Then 
$$
\left\|a_\om-\sum_{j\leq r}f_j(\om_{-j},...,\om_j)\right\|_{L^1(\bbP)}\leq \sum_{j\geq r}\|f_j\|_{L^1(\bbP)}:=A_1(r).
$$
Similarly, 
$$
\left\|a_{\te\om}-\sum_{j\leq r}f_j(\om_{-j+1},...,\om_{j+1})\right\|_{L^1(\bbP)}\leq \sum_{j\geq r}\|f_j\|_{L^1(\bbP)}=A_1(r).
$$
Now, since $B_\om$ is  a Lipschitz continuous function of $a_\om$ and $\rho(\om)$ is a Lipschitz continuous function of $a_{\te\om}$ we conclude that condition \eqref{Appp} holds true with $\beta_r=O(A_1(r))$ which converges to $0$ if $\|f_j\|_{L^1}$ is summable (to obtain decay rates for $\beta_r$ we only need to assume some decay rates for the norms $\|f_j\|_{L^1}$ as $j\to\infty$).

To give a concrete example when $a_\om\in[\frac12,1)$ we can, for instance, consider the case when 
$$
f_j(\om_{-j},...,\om_j)=a_j+b_j\bbI(\om_j\in \cA)
$$
for some measurable set $\cA$ such that $0<\bbP(\om_j\in \cA)<1$ and positive sequences $(a_j)$ and $(b_j)$ such that 
$$
\frac12\leq \sum_{j}b_j\,\,\text{ and }\,\,\sum_{j}(a_j+b_j)=1.
$$
Then $a_\om\in(\frac12,1)$ ($\bbP$-a.s.) and $a_\om$ can take arbitrarily close to $1$ values if $\bbP(\om_k\in \cA; k\leq d)>0$ for every $d>0$ (e.g. in the iid case or for the Markov chains in Example \ref{MixEg} (iv) when $\ell=1$). 

\vskip0.1cm
We can also assume that $\om_j\in[a,b], 0<a<b$ and then take  $f_j(\om_{-j},...,\om_j)=\al_j(\om_0,....,\om_j)\om_j$ for some $\al_j(\cdot)$ so that 
$\frac{1}{2a}\leq \sum_{j}\al_j<\frac{1}{b}$, assuming that $b<2a$.
For instance, if also $b<\frac{4a}3$ then $\al_j$ can have the form $\al_j(\cdot)=v_j(\om_0)\gamma_j$ for a positive series such that $\frac3{4a}\leq \sum_{j}\gamma_j\leq \frac1b$ and a random variable $v(\om_0)$ such that $\frac23<v(\om_0)<1$ and $\|v\|_{L^\infty}=1$.

\end{example}

\begin{theorem}[CLT]\label{CLT}
Let Assumption \ref{Sets Ass}  be in force.
Assume that the random variable $\om\to\|u_\om-\mu_\om(u_\om)\|$ is in $ L^p(\Om,\cF,\bbP)$ for some $p>2$ so that $\sum_{j}(\ln j\beta_{C j/\ln j})^{1-2/p}<\infty$ for all $C>0$. In  addition, assume that one of the following conditions is in force:
\vskip0.1cm
(M1)    $\sum_{j}(\al_U(C j/\ln j))^{1-2/p}<\infty$ for all $C>0$; 
\vskip0.1cm

(M2) $\limsup_{k\to\infty}\phi_U(k)<\bbP(A)$ \,\,(i.e. $\phi_U(k)<\bbP(A)$ for some $k$);
\vskip0.1cm

(M3) $\limsup_{k\to\infty}\psi_U(k)<\frac 1{1-\bbP(A)}-1$ \,\,(i.e. $\psi_U(k)<\frac 1{1-\bbP(A)}-1$ for some $k$).
\vskip0.1cm
Then:
\vskip0.1cm
(i) There is a number $\sig\geq0$ so that for $\bbP$-.a.e $\om$ we have 
$$
\lim_{n\to\infty}\frac1n\text{Var}_{\mu_\om}(S_n^\om u)=\sig^2.
$$
Moreover, $\sig>0$ if and only if the function $U(\om,x)=\sum_{j=0}^{n_A(\om)-1}(u_{\te^j\om}\circ T_\om^j x-\mu_{\te^j\om}(u_{\te^j\om}))$ has the form $U(\om,x)=q(\om,x)-q(\te^{n_A(\om)},T_\om^{n_A(\om)}x)$ for $\bbP_A$ almost every $\om$ and all $x$, where $n_A$ is the first return time to $A$, $\bbP_A(\cdot)=\bbP(\cdot|A)/\bbP(A)$ is the conditional measure on $A$ and $q$ is a measurable function so that $\int_{A}\int_{\cE\om}|q(\om,x)|^2d\mu_\om(x)d\bbP(\om)<\infty$.   
\vskip0.1cm
(ii) The sequence $S_n^\om u$ obeys the CLT: for every real $t$ we have
$$
\lim_{n\to\infty}\mu_\om\{x: n^{-1/2}\left(S_n^\om u(x)-\mu_\om(S_n^\om u)\right)\leq t\}=\frac{1}{\sqrt{2\pi}\sig}\int_{-\infty}^{t}e^{-\frac{s^2}{2\sig^2}}ds
$$
where if $\sig=0$ the above right hand side is interpreted as the distribution function of the constant random variable $0$.
\vskip0.1cm
(iii) Set $\tau(\om,x)=(\te\om, T_\om x)$, $\mu=\int_{\Om} \mu_\om d\bbP(\om)$ and $\tilde u(\om,x)=u_\om(x)-\mu_\om(u_\om)$.
If $\sig>0$ then the following functional version of the law of iterated logarithm (LIL) holds true. 
Let $\zeta(t)=\left(2t\log\log t\right)^{1/2}$ and 
$$
\eta_n(t)=\left(\zeta(\sig^2 n)\right)^{-1}\sum_{j=0}^{k-1}\left(\tilde u\circ\tau^j+(nt-k)\tilde u\circ\tau^k\right)
$$
for $t\in[\frac{k}{n}, \frac{k+1}n), k=0,1,...,n-1$. Then $\mu$-a.s. the sequence of functions $\{\eta_n(\cdot), n\geq 3/\sig^2\}$ is relatively compact in $C[0,1]$ (the space of continuous functions on $[0,1]$ with the supremum norm), and the set of limit points as $n\to\infty$ coincides with the set $K$ of absolutely continuous functions $x\in C[0,1]$ so that $\int_{0}^1(\dot{x}(t))^2dt\leq 1$. 

\end{theorem}

\begin{remark}
In the circumstances of Example \ref{1},\, $\beta_r$ vanishes for all $r$ large enough, and so the summability assumption on $\beta_r$ holds.   In the circumstances of Example \ref{2},\, $\beta_r$ decays exponentially fast as $r\to\infty$ and so the latter condition still holds true. In the circumstances of Example \ref{3} the summability condition $\sum_{j}(\ln j\beta_{C j/\ln j})^{1-2/p}<\infty$  holds if 
$$
\sum_{|j|\geq r}\|f_j\|_{L^1}=O(r^{-q})
$$
for some $q$ such that $q(1-2/p)>1$.
\end{remark}
The proof of Theorem \ref{CLT} starts in Section \ref{CLT 1 pf} and it is completed in Section \ref{Tails}. 
The proof of Theorem \ref{CLT} is based on inducing, and more precisely we  apply \cite[Theorem 2.3]{Kifer 1998}. The role  the assumptions on the upper mixing coefficient play is that, together 
with  the effective random RPF rates \eqref{Effective R},
 they allow us to verify the abstract conditions of  \cite[Theorem 2.3]{Kifer 1998} with the set $Q=A$ (where $Q$ is in the notations of  \cite[Theorem 2.3]{Kifer 1998}).

\begin{remark}
When using condition (M1) we only need that $\sum_{j}(\ln j\beta_{j/(3C\ln j)})^{1-2/p}<\infty$ 
and $\sum_{j}(\al_U(C j/\ln j))^{1-2/p}<\infty$ for some $C$ so that $C|\ln(1-\bbP(A)/2)|(1-2/p)>1$.
\vskip0.1cm
When using (M2) we only need that  $\sum_{j}(\ln j\beta_{j/(3C\ln j)})^{1-2/p}<\infty$ for some $C$ so that $C|\ln \del|(1-2/p)>1$, where $\del=1-\bbP(A)+\limsup_{r\to\infty}\phi_U(r)<1$.
\vskip0.1cm
When using (M3) we only need that  $\sum_{j}(\ln j\beta_{j/(3C\ln j)})^{1-2/p}<\infty$ for some $C$ so that $C|\ln\del|(1-2/p)>1$, where $\del=\left(1+\limsup_{r\to\infty}\psi_U(r)\right)\left(1-\bbP(A)\right)<1$.
\end{remark}

Next, let us provide alternative conditions for the CLT which involve a stronger type of approximation and moment assumptions on $B_\om$, but  do not require any approximation rates.
\begin{assumption}\label{Approx2}
There is a  sequence $\beta_r\to0$ as $r\to\infty$ so that for every $r$ there is a random variable $\rho_r(\om)$ which is measurable with respect to $\sig\{X_j; |j|\leq r\}$ and 
 $$
\|\rho-\rho_r\|_{L^\infty}\leq \beta_r.
 $$
 Namely,
 $$
 \lim_{r\to\infty}\left\|\rho-\bbE[\rho|X_{-r},...,X_r]\right\|_{L^\infty}=0.
 $$
 \end{assumption}
 \begin{example}
  (i)    Assumption \ref{Approx2} holds when $\rho(...,X_{-1},X_0,X_1,...)$ depends on finitely many of the $X_j$'s (in this case we can take $\rho_r=\rho$ for $r$ large enough). We refer to Example \ref{1} for an explicit example (namely the one in Section \ref{Eg Into} with $a_\om$ like in Example \ref{1}).
\vskip0.1cm
(ii)  Assumption \ref{Approx2} also holds true in the context of Example \ref{2} since  the estimates obtained there were actually in $L^\infty$.

  \vskip0.1cm

  (iii) In the context of Example \ref{3} we get that 
$$
\left\|a_{\te\om}-\sum_{0\leq j\leq r}f_j(\om_{-j+1},...,\om_{j+1})\right\|_{L^\infty(\bbP)}\leq \sum_{j\geq r}\|f_j\|_{L^\infty(\bbP)}:= A_\infty(r).
$$
Using that $\rho(\om)$ is a Lipschitz continuous function of $a_{\te\om}$ we conclude that
$$
\left\|\rho-\bbE[\rho|X_{-r},...,X_r]\right\|_{L^\infty}\to 0\,\text{ as }\,r\to\infty
$$
if $\sum_j\|f_j\|_{L^\infty}<\infty$ (and so Assumption \ref{Approx2} holds).
 \end{example}

\begin{theorem}\label{CLT2}
Let Assumption \ref{Approx2} be in force.
Moreover, assume that
$$
\limsup_{s\to\infty}\psi_{U}(s)<\infty\,\,\,(\text{i.e. } \psi_U(s)<\infty \text{ for some s})
$$
and that   $\|u_\om-\mu_\om(u_\om)\|, B_\om\in L^{3+\del}(\Om,\cF,\bbP)$ for some $\del>0$. 
 \vskip0.1cm
(i) There is a number $\sig\geq0$ so that for $\bbP$-a.a. $\om$ we have
$$
\sig^2=\lim\frac{1}n\text{Var}_{\mu_\om}(S_n^\om u).
$$
 Moreover, $\sig=0$ if and only if $\tilde u(\om,x)=q(\om,x)-q(\te\om, T_\om x)$  for some measurable function $q(\om,x)$ so that $\int q^2(\om,x)d\mu_\om(x)d\bbP(\om)<\infty$.  
  \vskip0.1cm
 (ii) The CLT as stated in Theorem \ref{CLT} (ii) is valid.
 \vskip0.1cm
 (iii) The functional LIL as stated in Theorem \ref{CLT} (iii) is valid.
\end{theorem}
Note that the results in Theorem \ref{CLT2} are slightly better than Theorem \ref{CLT} since we obtain a simpler coboundary characterization for the positivity of $\sig$. 

The proof of Theorem \ref{CLT} appears in Section \ref{CLT 2 pf}, and, like the proof of Theorem \ref{CLT}, it is also based on applying \cite[Theorem 2.3]{Kifer 1998}. However, even though  the conditions of \cite[Theorem 2.3]{Kifer 1998}  are  related to an inducing strategy, we will apply it with the set $Q=\Om$, namely we will ``induce" on $\Om$, so the proof will not really be based on inducing. In this case the conditions of  \cite[Theorem 2.3]{Kifer 1998} concern the asymptotic behavior of the system $(\Om,\cF,\bbP,\te)$ itself (that is, of $\te^n\om$ as $n\to\infty$) and not the induced system. Even though this is a stronger requirement, we will verify these conditions using the assumptions on the upper mixing coefficients together with the effective random RPF rates (Theorem \ref{RPF}).

\begin{remark}\label{Int cond U }
Besides the additional integrability assumptions in \ref{CLT2}, the  main difference between Theorems \ref{CLT} and \ref{CLT2} is that in the former we essentially require certain $L^1$-approximation rates (decay rates for $\beta_r$), while in the latter we do not require such rates, but instead we work with the stronger $L^\infty$-approximation coefficients, and only with the upper $\psi$-mixing coefficients. On the other hand, the restrictions on $\limsup_{k\to\infty}\psi_U(k)$ in Theorem \ref{CLT2} are much weaker than the ones in Theorem \ref{CLT}. 
Concerning the additional integrability assumption, as explained in Remark \ref{M rem 2} (see also Remark \ref{C 2 REM}), under the additional condition on $H_\om$ described there $B_\om$ is bounded, and so in this case the additional requirement that $ B_\om\in L^{3+\del}(\Om,\cF,\bbP)$ is always satisfied, and the only true integrability condition in Theroem \ref{CLT2} is  $\|u_\om-\mu_\om(u_\om)\|\in L^{3+\del}(\Om,\cF,\bbP)$ (like in Theorem \ref{CLT}). Finally, recall that $B_\om$ is bounded for the piecewise affine maps and their perturbations consider in Section \ref{Sec per lin}. Thus in this case $\|u_\om-\mu_\om(u_\om)\|\in L^{3+\del}(\Om,\cF,\bbP)$ is the only integrability condition needed.
\end{remark}

\subsection{The ASIP}

In this section we further assume that there is a random variable $N(\om)$ so that 
\begin{equation}\label{N def}
v(g\circ T_\om)\leq N(\om)v(g)
\end{equation}
for all functions $g$ with $v(g)<\infty$. Note that when the maps $T_\om$ are piecewise differntiable with bounded derivatives then we can always take $N(\om)=\sup\|DT_\om\|$.
Our next result is about almost sure approximation of $S_n^\om u$ by sums of independent Gaussians.

\begin{theorem}[ASIP]\label{ASIP}
Let Assumption \ref{Approx2} hold true and suppose that $\sig>0$.
Suppose also that\footnote{This condition always holds when $(X_j)$ is $\psi$-mixing, and we refer to Example \ref{MixEg}.} 
\begin{equation}\label{LS}
\limsup_{s\to\infty}\psi_U(s)<\frac{1}{\bbE_\bbP[\rho]}-1\,\,\,\,(\text{i.e. }\psi_U(s)<\frac{1}{\bbE_\bbP[\rho]}-1\text{ for some }s).
\end{equation}
Further assume that $\om\to B_\om$ belongs to $L^{2p}$ and  $\om\to N(\om)$ and $\om\to \|u_\om-\mu_\om(u_\om)\|$ belong to $L^p(\Om,\cF,\bbP)$ for some  $p>8$. Let $\tilde u_\om(x)=u_\om(x)-\mu_\om(u_\om)$. Then there is a coupling of $u_{\te^j\om}\circ T_\om^j$ (considered as a sequence of random variables on the probability space $(\cE_\om,\mu_\om)$) with a sequence of independent centered Gaussian random variables $Z_j$ so that  for every $\ve>0$,
$$
\max_{1\leq k\leq n}\left|S_k^\om\tilde u-\sum_{j=1}^kZ_j\right|=O(n^{1/4+\frac{9}{2p}+\ve}),\text{ a.s.}
$$
and 
$$
\left\|\sum_{j=1}^n Z_j\right\|_2^2=\text{Var}_{\mu_\om}(S_n^\om u)+O(n^{1/2+3/p+\ve}).
$$
\end{theorem}
The proof of Theorem \ref{ASIP} appears in Section \ref{ASIP pf}.
\begin{example}
For the example described in Section \ref{Eg Into} we have that $B_\om$ is bounded and \eqref{N def} holds with $N(\om)=(1-a_\om)^{-1}$ (i.e. the maximal amount of expansion). We conclude that under \eqref{LS}, the ASIP rates $O(n^{1/4+\frac{9}{2p}+\ve})$ hold if $\om\to \|u_\om\|$ is in $L^p$ and $a_\om=1+1/R_\om$ for some $R_\om\in L^p$. Finally, let us note that this result is already  meaningful when $X_j$ are independent. In this case $\psi_U(n)=0$ and we can just consider any function $\om\to a_\om, \,a_\om\in(\frac12,1)$ of the Bernoulli shift $(\Om,\cF,\bbP,\te)$ such that $(1-a_\om)^{-1}\in L^p$ (still, we can consider such functions of any shift system generated by an arbitrary $\psi$-mixing sequence).
\end{example}
\begin{remark}\label{asip Rem}
 More generally, 
 as discussed in Remark \ref{M rem 2} (see also Remark \ref{C 2 REM}), when, in addition to \eqref{H cond} we have $\|H_\om\|_{L^\infty}<\infty$ then  $B_\om$ is bounded (in particular we ). Thus, for such maps the only true integrability conditions in Theorem \ref{ASIP} are $N(\om),\|\tilde u_\om\|\in L^p$.  Finally, recall that $B_\om$ is bounded for the piecewise affine maps and their perturbations considered in Section \ref{Sec per lin}. Now, for such maps $N(\om)$ is the supremum norm of the (piecewise) gradient and so Theorem \ref{ASIP} holds true when the supremum norm and $\om\to\|\tilde u_\om\|$ are in $L^p$. 
\end{remark}

\subsection{Large deviations principles with a quadratic rate function}
%Also prove some exponential concentration inequalities?
Consider the following additional condition.

\begin{assumption}\label{Add ass}
(i) The random variable $E_\om$ defined in Section \ref{Aux1} (or Section \ref{Aux2}) is bounded. 
\vskip0.1cm

(ii)
In the setup of Section \ref{Maps1}, \eqref{phi cond} is satisfied with some $H_\om$ so that 
 $$
Z_\om:=\gamma_{\om}^{-\al}v(u_\om)+H_\om\leq \gamma_{\te\om}^\al-1.
$$
\end{assumption}
\begin{remark}
In the setup of Section \ref{Maps1},
the condition that $E_\om$ is bounded essentially means that $\|u_\om\|_\infty$ and $Z_\om$ are small when $\gamma_{\te\om}$ is close to $1$.  To demonstrate that let us assume that 
$$
Z_\om\leq r_\om(\gamma_{\te\om}^\al-1)
$$
for some $r_\om<1-\ve$, $\ve\in(0,1)$. 
Then by replacing $Z_\om$ with the above upper bound and then using some elementary estimates we see that
$$
E_\om\leq C\left(\|u_\om-\mu_\om(u_\om)\|_{\infty}+r_\om\right)e^{2\gamma_\om^\al}\gamma_{\te\om}^{\al}\left(\gamma_{\te\om}^\al-1\right)^{-1}.
$$
We thus see that $E_\om$ is bounded if $\|u_\om\|_\infty+r_\om$ is small enough (fiberwise). For instance, when $\gamma_\om$ is bounded above we get  the sufficient condition
$$
\|u_\om-\mu_\om(u_\om)\|_\infty+r_\om\leq C(\gamma_{\te\om}^\al-1)
$$
which means that $\|u_\om-\mu_\om(u_\om)\|_\infty+r_\om$ is small when $T_{\te\om}$ has a local inverse branch with a close to $1$ amount of contraction $\gamma_{\te\om}^{-1}$. 

\vskip0.2cm
In the setup of Section \ref{Maps2}, the random variable $E_\om$ is bounded  
 when $\|\phi_\om\|_\infty, H_\om$ and $\|u_\om\|$ are (fiberwise) small enough when $\zeta_\om$ is close to $1$.
\end{remark}
\begin{theorem}\label{MDP1}
Under Assumption \ref{Add ass} we have the following.  Assume that $\sig>0$ and that for some $p>4$ we have that $\bar D_\om\in L^{2p}(\Om,\cF,\bbP)$ and $K_\om,M_\om,\|\tilde u_\om\|_\infty\in L^p(\Om,\cF,\bbP)$, where $\tilde u_\om=u_\om-\mu_\om(u_\om)$.
Moreover, suppose that for some measurable set $A\subset\Om$ with positive probability  we have:
\vskip0.1cm
(i) $A$ satisfies the approximation properties described in Assumption \ref{Sets Ass};
\vskip0.1cm
(ii) the random variable $\max(M_\om, K_\om, B_\om)$ is bounded on $A$;
\vskip0.1cm
(iii) $A$ satisfies the assumptions of Theorem \ref{CLT} with $\frac{p-1}p=1-1/p$ instead of $1-2/p$
 (let us denote the corresponding conditions by (M1'), (M2') and (M3'), respectively).
\vskip0.1cm
Then the following moderate deviations principle holds true for $\bbP$-a.a. $\om$: for every balanced\footnote{We say that a sequence $(a_n)$ of positive numbers is balanced if $\frac{a_n}{a_{c_n n}}\to 1$ for every sequence $(c_n)$ so that $c_n\to 1$.} sequence $(a_n)$ so that 
$\frac{a_n}{\sqrt n}\to\infty$, and $a_n=o(n^{1-6/p})$  and all Borel measurable sets $\Gamma\subset\bbR$ we have 
\begin{equation}\label{mdp}
-\inf_{x\in\Gamma^o}\frac12x^2\sig^{-2}\leq\liminf_{n\to\infty}\frac{1}{a_n^2/n}\ln\bbP(S_n^\om\tilde u/a_n\in\Gamma)
\leq \limsup_{n\to\infty}\frac{1}{a_n^2/n}\ln\bbP(S_n^\om\tilde u/a_n\in\Gamma)\leq -\inf_{x\in\overline{\Gamma}}\frac12x^2\sig^{-2}
\end{equation}
where $\Gamma^o$ is the interior of $\Gamma$ and $\overline{\Gamma}$ is its closure.
\end{theorem}
Note that as an example of a sequence $a_n$ we can take $a_n=n^{q}(\ln n)^\te$ for $\te\geq0$ and $\frac12<q<\min(1,2-p/6)$. As in  Example \ref{Egg} the approximation conditions and (M1')-(M3') hold true when $B_\om, M_\om$ and $K_\om$ can be approximated sufficiently fast by functions of $X_j, |j|\leq r$ and that the upper mixing coefficients of the sequence $(X_j)$ satisfy (M1')-(M3').

The following result provides alternative conditions for the MPD.
\begin{theorem}\label{MDP2}
Under Assumption \ref{Add ass} we have the following. Let the same integrability conditions in Theorem \ref{MDP1} hold with some $p>8$ and suppose again that $\sig>0$.  Then the MDP \eqref{mdp} holds true with any sequence $(a_n)$ so that $a_n n^{-\max(6/p,1/2)}\to\infty$ and $a_n=o(n^{1-8/p})$.
\end{theorem}

The proof of Theorems \ref{MDP1} and \ref{MDP2} appear in Section \ref{LDP pf}.
\begin{remark}
Recall that $M_\om$ is bounded in the setup of Section \ref{Maps1}, and so the condition $M_\om\in L^p$ is not really a restriction in that setup. Moreover, as explained in Remark \ref{M rem 2} (see also Remark \ref{C 2 REM}) when $H_\om$ is also bounded then the random variables $K_\om$ and $B_\om$ are bounded. In this case also the condition $K_\om,B_\om\in L^p$ is not really a restriction, and the only real integrability condition is $\bar D_\om\in L^{2p}$.
\end{remark}
\begin{remark}
The main difference between Theorems \ref{MDP1} and \ref{MDP2} is that Theorem \ref{MDP1} essentially requires some mixing assumptions on the sequences of random variables $(B_{\te^j\om}), (M_{\te^j\om})$ and $(K_{\te^j\om})$, while Theorem \ref{MDP2} does not require mixing assumptions. On the other hand, the integrability conditions in Theorem \ref{MDP1} are weaker than the ones in Theorem \ref{MDP2} (i.e. $p>4$ versus $p>8$). Since the first integrability conditions are not much  better than the second, Theorem \ref{MDP2} is somehow better than Theorem \ref{MDP1}, and the reason that Theorem \ref{MDP1} is included is that its proof is based on a  certain inducing strategy, and we find it interesting to present exact conditions which make the method of proof by inducing effective for proving an MDP for random Birkhoff sums.
\end{remark}

\subsection{Berry Esseen type estimates and moderate local limit theorem}
Using the arguments in the proof of Theorems \ref{MDP1} and \ref{MDP2} we can also prove the following results.
\begin{theorem}[A Berry-Esseen theorem]\label{BE}
Let $\sig_{\om,n}=\sqrt{\text{Var}_{\mu_\om}(S_n^\om u)}$.
\vskip0.1cm

(i) 
Under the assumptions of Theorem \ref{MDP1}, when $p>12$ then  $\bbP$-a.s. we have 
$$
\sup_{t\in\bbR}\left|\mu_\om(S_n^\om \tilde u\leq t\sig_{\om,n})-\Phi(t)\right|=O(n^{-(1/2-6/p)})
$$
where when $p=\infty$ we use the convention $6/p=0$.

(ii) Under the assumptions of Theorem \ref{MDP2}, when $p>16$ then $\bbP$-a.s. we have 
$$
\sup_{t\in\bbR}\left|\mu_\om(S_n^\om \tilde u\leq t\sig_{\om,n})-\Phi(t)\right|=O(n^{-(1/2-8/p)})
$$
where for $p=\infty$ we have $8/p:=0$.
\end{theorem}

The proof of Theorem \ref{BE} appears in Section 
\ref{LLT pf}.
\begin{remark}
In the setup of Section \ref{Maps1}, the uniformly random case (i.e. $p=\infty$) was covered in
 \cite[Theorem 7.1.1]{HK}, see also  \cite{DH, HafYT} for optimal rates for different types of random maps. However, in the setup of Section \ref{Maps2}, Theorem \ref{BE} is new even in the uniformly random case (so we get the optimal CLT rates $O(n^{-1/2})$ in that case). In the deterministic case when the maps $T_\om$ and the functions $u_\om$ do not depend on $\om$ the arguments in the proof of Theorem \ref{BE} provide  explicit constants in the Berry-Esseen theorem (similarly to \cite[Theorem 1.1]{Dub2}).
\end{remark}

\begin{theorem}[A moderate local central limit theorem]\label{LLT}

Under the assumptions of Theorem \ref{MDP2}, $\bbP$-a.s. we have the following.
Let  $(a_n)$ be a  sequence so that $a_{n}n^{-2/p}\to\infty$ and $a_n n^{-1/2}\to 0$ (where $p$ comes from Theorem \ref{MDP2}). Then for every
continuous function $g:\bbR\to\bbR$ with a compact support or an indicator function of a bounded interval we have
\begin{equation}\label{llt}
\sup_{v\in\bbR}\left|\sqrt{2\pi \ka_{\om,n}}\mu_{\om}(g(S_n^\om \tilde u/a_n-v))-\left(\int g(y)dy\right)e^{\frac{-v^2}{2\ka_{\om,n}^2}}\right|=o(1)
\end{equation}
where $\ka_{\om,n}=\sig_{\om,n}/a_n$.
 In particular, for every bounded interval $I$ we have
 $$
\sup_{v\in\bbR} \left|\sqrt{2\pi \ka_{\om,n}}\mu_{\om}(S_n^\om\tilde u\in a_n (v+I))-|I|e^{\frac{-v^2}{2\ka_{\om,n}^2}}\right|=o(1).
 $$
\end{theorem}

The proof of Theorem \ref{LLT} appears in Section 
\ref{LLT pf}.

\begin{remark}
The classical local central limit theorem (LCLT) corresponds to the case when $a_n=1$, which is excluded in Theorem \ref{LLT} even when $p=\infty$, where the first requirement on $a_n$ becomes $a_n\to\infty$. The case $p=\infty$ corresponds to the uniformly random case, and we refer to \cite[Ch. 6]{HK} for sufficient conditions\footnote{Which involve a period point of $\te$ and  some notion of an aperiodicity of the Birkhoff sums.} for the validity of \eqref{llt} with $a_n=1$ in the uniformly random version of the setup of Section \ref{Maps1}.
Relying on Theorem \ref{Complex RPF} below, the classical LCLT  can be obtained in the uniformly random version of the maps considered in Section \ref{Maps2} when $l_\om\leq1$, since in that case we have $\|\cL_\om^{it,n}\|\leq C(1+|t|)$ for some $C\geq1$ and all $t\in\bbR$, where $\cL_{\om}^{it,n}$ is defined before Theorem \ref{Complex RPF}.
Finally, note that Theorem \ref{LLT} is new even in the uniformly random case, which is important especially when the known sufficient conditions for the classical LCLT fail due to a some type of periodicity exhibited by the random Birkhoff sums. 
\end{remark}

\subsection{Key technical tools: real and complex  random RPF theorems with effective rates}\label{SecRPF}
For all the maps $T_\om$ considered in Sections \ref{Maps1} and \ref{Maps2},  and every complex number $z$ we
consider the random transfer operator $\cL_\om^{(z)}$ which maps functions on $\cE_\om$ to functions on $\cE_{\te\om}$ according to the formula
\begin{equation}\label{Tr op}
\cL_\om^{(z)}g(x)=\sum_{y\in T_\om^{-1}\{x\}}e^{\phi_\om(y)+z\tilde u_\om(y)}g(y)=\sum_i e^{\phi_\om(y_{i,\om}(x))+z\tilde u_\om(y_{i,\om}(x))}g(y_{i,\om}(x)).
\end{equation}
Here $\tilde u_\om=u_\om-\mu_\om(u_\om)$.
We also set $\cL_\om^{(0)}=\cL_\om$. For each $\om,n$ and $z$ write
\[
\cL_\om^{z,n}=\cL_{\te^{n-1}\om}^{(z)}\circ\cdots\circ\cL_{\te\om}^{(z)}\circ\cL_\om^{(z)}.
\]
It is clear that $\cL_\om^{(z)}\cH_\om\subset \cH_{\te\om}$. We will denote by $(\cL_\om^{(z)})^*$ the appropriate dual operator. When $z=0$ we denote $\cL_\om^n=\cL_{\om}^{0,n}$.

In \cite{MSU} it was shown that in the setup of Section \ref{Maps1}  there is a unique triplet $(\la_\om,h_\om,\nu_\om)$ consisting of a random variable $\la_\om>0$, a random positive function $h_\om\in\cH_\om$ and a probability measure on $\cE_\om$ so that $\bbP$-a.s. we have $\nu_\om(h_\om)=1$,
$$
\cL_\om h_\om=\la_\om h_{\te\om}\,\,\text{ and }\,\,(\cL_\om)^*\nu_{\te\om}=\la_\om\nu_\om.
$$
Moreover, with $\la_{\om,n}=\prod_{j=0}^{n-1}\la_{\te^j\om}$, there is a constant $\del\in(0,1)$ and a random variable $C(\cdot)$ so that for every $g\in\cH_\om$ we have
$$
\|(\la_{\om,n})^{-1}\cL_{\om}^{n}g-\nu_\om(g)h_{\te^n\om}g\|\leq C(\te^n\om)\del^n\|g\|.
$$
The above result is often refereed to as a random Ruelle-Perron-Frobenius (RPF) theorem.
The random variable $C(\om)$  can be expressed by means of a first hitting time to a certain set which can be defined by means of some ergodic average and a random variable which can be expressed as a series of known random variables.  A similar result follows from \cite{Varandas} in the setup of Section \ref{Maps2} (the uniqueness is obtained under additional assumptions but the construction of the RPF triplets proceeds without the additional requirements).

One of the main tools in the proof of all  the limit theorems in this paper is the following result, which is an effective version of the above RPF theorem.

\begin{theorem}[An effective RPF theorem]\label{RPF}

The  RPF triplets above $(\la_\om,h_\om,\nu_\om)$ satisfy the following ($\bbP$-a.s.):

(i)  $h_\om\in \cC_\om$ and $\|h_\om\|\leq K_\om$;

(ii) $1\leq \sup h_\om\leq B_{\om,1}\inf h_\om\leq  B_{\om,1}$;

(iii) for every $n$ and $g\in\cH_\om$,
\begin{equation}\label{RPF ExpC}
\left\|\frac{\cL_\om^n g}{\la_{\om,n}}-\nu_\om(g) h_{\te^n\om}\right\|\leq 4K_{\te^n\om}\rho_{\om,n}
\end{equation}
where $\la_{\om,n}=\prod_{j=0}^{n-1}\la_{\te^j\om}$ and $\rho_{\om,n}=\prod_{j=0}^{n-1}\rho({\te^j\om})$;

(iv) let the probability measures $\mu_\om$ on $\cE_\om$ be given by $\mu_\om=h_\om \nu_\om$. Then $\bbP$-a.s. we have $(T_\om)_*\mu_\om=\mu_{\te\om}$. Moreover,
let $L_\om$ be the operator given by $L_\om g=\frac{\cL_\om g}{\la_\om h_{\te\om}}$. Then for every H\"older continuous function $g$ on $\cE_\om$ and all $n\geq 1$ we have
\begin{equation}\label{Exp L}
\|L_\om^n g-\mu_\om(g)\|\leq B_{\te^n\om}\rho_{\om,n}\|g\|.
\end{equation}
 
(v) [exponential decay of correlations] for every natural $n$ and  for all $g\in\cH_\om$ and $f\in\cH_{\te^n\om}$ we have
\begin{equation}\label{DEC}
\left|\mu_\om(g\cdot (f\circ T_\om^n))-\mu_\om(g)\mu_\om(f\circ T_\om^n)\right|\leq 
B_{\te^n\om}\rho_{\om,n}\|g\|\|f\|_{L^1(\mu_{\te^n\om})}.
\end{equation}
\end{theorem}

The proof of Theorem \ref{RPF} appears in Section \ref{RPF pf}.

\begin{remark}
Notice that $\la_\om=\nu_{\te\om}(\cL_\om \textbf{1})$, where $\textbf{1}$ is the function which takes the constant value $1$.
 Hence,
\begin{equation}\label{la bounds}
e^{-\|\phi_\om\|_\infty}\leq \inf\cL_\om \textbf{1}\leq \la_\om\leq \sup\cL_\om \textbf{1}.
\end{equation}
In the finite  degree case $|\cL_\om \textbf{1}|\leq \deg(T_\om)e^{\|\phi_{\om}\|_\infty}$, while if the degree is not bounded then Assumption \ref{SumAss} insures that $\cL_\om \textbf{1}$ is a bounded function (recall that for piecewise affine maps we always have $\cL_\om \textbf{1}=\textbf{1}$). 
\end{remark}

In the proof of Theorems \ref{MDP1}, \ref{MDP2}, \ref{BE} and \ref{LLT} we will also need the following complex version of Theorem \ref{RPF}.
\begin{theorem}\label{Complex RPF}
When  the random variable $E_\om$ is bounded we have the following. There is a positive number $r_0>0$ so that for any complex number $z$ such that $|z|\leq r_0$
there exist  measurable families
$\la_\om(z)$, $h_\om^{(z)}$ and $\nu_\om^{(z)}$ 
which are analytic in $z$,
consisting of a nonzero complex number 
$\la_\om(z)$, a complex function $h_\om^{(z)}\in\cH_\om$ and a 
complex continuous linear functional $\nu_\om^{(z)}\in\cH_\om^*$ such that:
\vskip0.1cm

(i) We have
\begin{equation}\label{RPF deter equations-General}
\cL_\om^{(z)} h_\om^{(z)}=\la_\om(z)h_{\te\om}^{(z)},\,\,
(\cL_\om^{(z)})^*\nu_{\te\om}^{(z)}=\la_\om(z)\nu_{\om}^{(z)}\text{ and }
\nu_\om^{(z)}(h_\om^{(z)})=\nu_\om^{(z)}(\textbf{1})=1.
\end{equation} 
 Moreover, $h_\om^{(0)}=h_\om$, $\la_\om(0)=\la_\om$ and $\nu_\om^{(0)}=\nu_\om$.
%Some explicit bounds in the complex case?
\vskip0.1cm

(ii) We have $\|\nu_\om^{(z)}\|\leq M_\om$ and $h_\om^{(z)}=\frac{\hat h_\om^{(z)}}{\al_\om(z)}$,  where $\al_\om(z):=\nu_\om^{(z)}(\hat h_\om^{(z)})
\not=0$
 for some analytic in $z$ family of  functions $\hat h_\om^{(z)}$ so that $\|\hat h_\om^{(z)}\|\leq 2\sqrt 2K_\om$ (note that $\al_\om(0)=1$ and $\|\al_\om(z)\|\leq 2\sqrt 2M_\om K_\om$).
\vskip0.1cm

%there exist a constant $c>0$, which depends only on the initial parameters, so that $|\la_\om(z)|\geq c$ 
%and $\min_{x\in \cX_j}|h_\om^{(z)}(x)|\geq c$ for any integer $j$ and $z\in U$.
%explicit bounds...what do we 
\vskip0.1cm
%Here we assume the assumption that allows us to get the integrability of the first time that $M_\om^{-1}\rho_n(\te^{-n}\om)\leq 1/2$ or something similar
(iii) Let $n_0(\om)$ be the first time that $|\al_\om(z)|\geq 2\sqrt 2 M_\om K_\om \rho_{\te^{-n}\om,n}$. Then for every $n\geq n_0(\om)$ and all $g\in\cH_{\te^{-n}\om}$ we have
\begin{equation}\label{Exponential convergence}
\Big\|\frac{\cL_{\te^{-n}\om}^{z,n}g}{\la_{\te^{-n}\om,n}(z)}-\nu_{\te^{-n}\om}^{(z)}(g)h^{(z)}_{\om}\Big\|\leq 8 M_{\te^{-n}\om}K_\om\left(|\al_\om(z)|^{-1}+\frac{M_\om K_\om}{|\al_\om(z)|^2}\right)\|g\|\tilde\rho_{\te^{-n}\om,n}:=\cR(\om,n,z)
\end{equation}
where $\la_{\om,n}(z)=\la_{\om}(z)\cdot\la_{\te\om}(z)\cdots\la_{\te^{n-1}\om}(z)$ (recall that $M_\om$ is bounded in the setup of Section \ref{Maps1}).
\end{theorem}
\vskip0.1cm

(iv) Let the operators $L_\om^{(z)}$ be given by
$$
L_{\om}^{(z)}g=L_\om(e^{zu_\om}g)=\frac{\cL_{\om}^{(z)}(gh_\om)}{\la_\om h_{\te\om}}
$$ 
and set $\bar\la_\om(z)=\frac{\la_\om(z)}{\la_\om}$, $\bar h_\om(z)=\frac{h_\om^{(z)}}{h_\om}$ and $\bar\nu_\om^{(z)}=h_\om\cdot\nu_\om^{(z)}$. Then for all $n\geq n_0(\om)$,
\begin{equation}\label{Exponential convergence CMPLX}
\Big\|\frac{L_{\te^{-n}\om}^{z,n}g}{\bar\la_{\te^{-n}\om,n}(z)}-\bar\nu_{\te^{-n}\om}^{(z)}(g)\bar h^{(z)}_{\om}\Big\|\leq 6B_{\om}\cR(\om,n,z).
\end{equation}

The proof of Theorem \ref{Complex RPF} appears in Section \ref{RPF CPMLX pf}.

\begin{remark}
Since $z\to|\al_\om(z)|$ is continuous and positive, we have $\beta_\om=\sup_{|z|\leq r_0}|\al_\om(z)|^{-1}<\infty$ and so, in principle, we can use that to get an upper bound which does not depend on $z$. However, $\beta_\om$ does not have an explict form. Instead, in the proof of the large deviations theorems (Theorems \ref{MDP1} and \ref{MDP2}) we will use that $|\al_\om(z)-1|\leq C|z|K_\om M_\om$ and that, when $K_\om M_\om\in L^p(\Om,\cF,\bbP)$ we have $K_{\te^n\om} M_{\te^n\om}=o(n^{2/p})$, which will produce effective bounds when $|z|=O(n^{-2/p})$.
\end{remark}

\begin{remark}
Theorem \ref{Complex RPF} was proven in \cite[Ch.5]{HK} in the uniformly random version of the setup in Section \ref{Maps1}. However, in the setup of Section \ref{Maps2} Theorem \ref{Complex RPF} is new even in the uniformly random case (i.e. when $s_\om\leq s<1$ for some constant $s$ and all the other random variables are bounded). In fact, it is new even in the deterministic case when and $T_\om=T, \phi_\om=\phi$ and $u_\om=u$ do not depend on $\om$. As mentioned in Section \ref{Section 1}, Theorem \ref{Complex RPF} makes it possible to extend results like \cite[Theorem 1.1]{Dub2} to the partially expanding case.
Note that in the uniformly random case $\sup_{|z|\leq r_0}|\al_\om(z)|\leq C$ for some constant $C>0$, and so, since $\al_\om(0)=1$, by using the mean value theorem and the Cauchy integral formula and decreasing $r_0$ if needed, we have that $\frac12\leq |\al_\om(z)|\leq \frac32$ (and so we can replace the term $|\al_\om(z)|$ with a constant).
\end{remark}

\section{Proofs of the limit theorems based on the random RPF theorems}
\subsection{The CLT and LIL: Proof of Theorem \ref{CLT} by inducing}\label{CLT 1 pf}
The proof of  Theorem \ref{CLT} is based on an application of \cite[Theorem 2.3]{Kifer 1998} with the set $Q=A$, where $A$ comes from  Assumption \ref{Sets Ass}.
Let $c(\om)=\|u_\om-\mu_\om(u_\om)\|_{L^2(\mu_\om)}$ and let $n_A(\om)$ be the first hitting time to the set $A$.
Set $\tilde u_\om=u_\om-\mu_\om(u_\om)$ and
\begin{equation}\label{PPsi}
\Psi_\om=S_{n_A(\om)}^{\om}\tilde u_\om=\sum_{j=0}^{n_A(\om)-1}\tilde u_{\te^j\om}\circ T_\om^j
\end{equation}
and let $\Theta: A\to A$ be given by $\Theta(\om)=\te^{n_A(\om)}(\om)$.  Let us also consider the maps $\cT_\om=T_{\om}^{n_A(\om)}$ and the corresponding transfer operators $\tilde L_\om=L_\om^{n_A(\om)}$.

Then the conditions of \cite[Theorem 2.3]{Kifer 1998} are met if 

\begin{equation}\label{Ver1}
\left\|\sum_{j=0}^{n_A(\om)-1}c(\te^j\om)\right\|_{L^2(\bbP)}<\infty.
\end{equation}

\begin{equation}\label{Ver2}
\left\|\bbI(\om\in A)\sum_{n=0}^\infty|\bbE_{\mu_\om}[\Psi_\om\cdot\Psi_{\Theta^n\om}\circ \cT_\om^{n}]|\right\|_{L^1(\bbP)}<\infty
\end{equation}
and
\begin{equation}\label{Ver3}
\left\|\bbI(\om\in A)\sum_{n=0}^{\infty}\bbE_{\mu_\om}(|\tilde L_{\Theta^{-n}\om}^n\Psi_{\Theta^{-n}\om}|)\right\|_{L^2(\bbP)}<\infty.
\end{equation}

\subsubsection{Reduction to tails estimates of the first hitting times}

\begin{lemma}\label{LV}
All three conditions \eqref{Ver1}, \eqref{Ver2} and \eqref{Ver3} are valid if $c(\om)\in L^p(\Om,\cF,\bbP)$ and
\begin{equation}\label{Series}
\sum_{j=0}^{\infty}(\bbP(n_A>j))^{1-2/p}<\infty
\end{equation}
 for some $p>2$.
\end{lemma}

\begin{proof}
Let us begin with showing that condition \eqref{Ver1} is in force.
Write
$$
\sum_{j=0}^{n_A(\om)-1}c(\te^j\om)=\sum_{j=0}^{\infty}c(\te^j\om)\bbI(n_A>j).
$$
Then, by the H\"older inequality we have
$$
\left\|\sum_{j=0}^{n_A(\om)-1}c(\te^j\om)\right\|_{L^2(\bbP)}\leq \sum_{j=0}^{\infty}\|(c\circ\te^j)\bbI(n_A>j)\|_{L^2(\bbP)}
\leq\|c\|_{L^p(\bbP)}^2\sum_{j=0}^{\infty}(\bbP(n_A>j))^{1-2/p}<\infty.
$$

Next, let us show that condition \eqref{Ver2} is satisfied.
First, let $k_n(\om)$ be so that $\cT_{\om}^{n}=T_\om^{k_n(\om)}$. Then  by using the definition of $\Psi_\om$ and that $\{\mu_\om\}$ is an equivariant family we see that
$$
\bbE_{\mu_\om}[\Psi_\om\cdot\Psi_{\Theta^n\om}\circ \cT_{\om}^{n}]=\sum_{j=0}^{n_A(\om)-1}\bbE_{\mu_{\te^j\om}}[\tilde u_{\te^j\om}\cdot \Psi_{\Theta^n\om}\circ T_{\te^j\om}^{k_n(\om)-j}].
$$
Now, by using \eqref{DEC}, the properties of the set $A$ and that $\Theta^n\om\in A$  we see that 
$$
\left|\bbE_{\mu_{\te^j\om}}[\tilde u_{\te^j\om}\cdot \Psi_{\Theta^n\om}\circ T_{\te^j\om}^{k_n(\om)-j}]\right|\leq M(1-\ve)^{n}\|\Psi_{\Theta^n\om}\|_{{L^1(\mu_{\Theta^n\om})}}\|\tilde u_{\te^j\om}\|
$$
where we have used that there are $n$ visits to $A$ between $\te^j\om$ and $\Theta^n\om$ for 
$j<n_A(\om)$. 
Thus,  
$$
|\bbE_{\mu_\om}[\Psi_\om\cdot\Psi_{\Theta^n\om}\circ \cT_{\om}^{n}]|\leq M \|\Psi_{\Theta^n\om}\|_{L^1(\mu_{\Theta^n\om})}(1-\ve)^n\sum_{j=0}^{n_A(\om)-1}\|\tilde u_{\te^j\om}\|.
$$
Next, with $\om_n=\Theta^n\om$ we have 
$$
\|\Psi_{\Theta^n\om}\|_{L^1(\mu_{\Theta^n\om})}
\leq \sum_{j=0}^{n_A(\om_n)-1}\|\tilde u_{\te^j\om_n}\|_{L^1(\mu_{\te^j\om_n})}\leq 
\sum_{j=0}^{n_A(\om_n)-1}\|\tilde u_{\te^j\om}\|. 
$$
Let 
\begin{equation}\label{I om}
I(\om)=\sum_{j=0}^{n_A(\om)-1}\|\tilde u_{\te^j\om}\|=\sum_{j=0}^{n_A(\om)-1}c(\te^j\om).
\end{equation}
Then we conclude from the above estimates that 
$$
\left\|\bbI(\om\in A)\sum_{n=0}^\infty|\bbE_{\mu_\om}[\Psi_\om\cdot\Psi_{\Theta^n\om}\circ\cT_\om^n]|\right\|_{L^1(\bbP)}\leq M\bbE[\bbI(\om\in A)I(\om)I(\Theta^n\om)]\sum_{n=0}^\infty(1-\ve)^n.
$$
To complete the proof of \eqref{Ver2} we notice that in the proof of \eqref{Ver1} we showed that $I(\om)\in L^2(\Om,\cF,\bbP)$, which together with Cauchy-Schwarz inequality yields that 
$$
\bbE[\bbI(\om\in A)I(\om)I(\Theta^n\om)]=\bbE[(I(\om))(\bbI(\om\in A)I(\Theta^n\om))]\leq (P(A))^{-1}\bbE[I^2(\om)]
$$
where we have used that $\Theta$ preserves $\bbP_A=\bbP(\cdot|A)$.

Finally, let us verify condition \eqref{Ver3}. First, we have 
$$
\left\|\bbI(\om\in A)\sum_{n=0}^{\infty}\bbE_{\mu_\om}(|\tilde L_{\Theta^{-n}\om}^n\Psi_{\Theta^{-n}\om}|)\right\|_{L^2(\bbP)}\leq
\sum_{n=0}^{\infty}\left\|\bbI(\om\in A)\bbE_{\mu_\om}(|\tilde L_{\Theta^{-n}\om}^n\Psi_{\Theta^{-n}\om}|)\right\|_{L^2(\bbP)}.
$$
Second, since $\te$ preserves $\bbP$ and $\{\mu_\om\}$ is an equivariant family for each $n$ we have 
$$
\left\|\bbI(\om\in A)\bbE_{\mu_\om}(|\tilde L_{\Theta^{-n}\om}^n\Psi_{\Theta^{-n}\om}|)\right\|_{L^2(\bbP)}=
\left\|\bbI(\Theta^n\om\in A)\bbE_{\mu_{\Theta^n\om}}(|\tilde L_{\om}^n\Psi_{\om}|)\right\|_{L^2(\bbP)}
$$
$$
\leq \sum_{j=0}^{n_A(\om)-1}\left\|\bbE_{\mu_{\Theta^n\om}}(|L_{\te^j\om}^{u_n(\om)-j}\tilde u_{\te^j\om}|)\right\|_{L^2(\bbP)}
$$
where $\Theta^n\om=\te^{u_n(\om)}\om$, and in the last inequality we have used \eqref{PPsi}.
 Now, since $\te^{u_n}\om\in A$ and there are exactly $n$ returns to $A$ between ``times" $j$  and $u_n$ (since $j<n_A(\om)$) we get from \eqref{Exp L} that 
$$
\bbE_{\mu_{\Theta^n\om}}(|L_{\te^j\om}^{u_n(\om)-j}\tilde u_{\te^j\om}|)\leq 
M(1-\ve)^n\|\tilde u_{\te^j\om}\|.
$$
Thus, 
$$
\left\|\bbI(\om\in A)\bbE_{\mu_\om}(|\tilde L_{\Theta^{-n}\om}^n\Psi_{\Theta^{-n}\om}|)\right\|_{L^2(\bbP)}\leq M(1-\ve)^nI(\om),
$$
where $I(\om)$ was defined in \eqref{I om}.
Combining the above estimates we conclude that
$$
\left\|\bbI(\om\in A)\sum_{n=0}^{\infty}\bbE_{\mu_\om}(|\tilde L_{\Theta^{-n}\om}^n\Psi_{\Theta^{-n}\om}|)\right\|_{L^2(\bbP)}\leq\|I(\cdot)\|_{L^2(\bbP)}M\sum_{n=0}^\infty(1-\ve)^n<\infty
$$
and the proof of \eqref{Ver3} is completed.
\end{proof}

\subsubsection{Tails estimates using upper mixing coefficient: proof of Theorem \ref{CLT}}\label{Tails}
In this section we will show that  condition \eqref{Series} in Lemma \ref{LV} is valid under the assumptions of Theorem \ref{CLT}. This together with Lemma \ref{LV} and 
\cite[Theorem 2.3]{Kifer 1998} will complete the proof of Theorem \ref{CLT}.

Before we begin with obtaining upper bounds on the tail probabilities $\bbP(n_A>j)$, let us note that 
\begin{equation}\label{GenForm}
\bbP(n_A>j)=\bbP\left(\bigcap_{k=1}^{j}\te^{-k}(\Om\setminus A)\right).
\end{equation}

\subsubsection{Proof of Theorem \ref{CLT} under Assumption (M1)}
We first need the following result.
\begin{lemma}\label{L alpha}
Let $I_1,I_2,...,I_m$, $m\geq 2$ be finite subsets of $\bbN$ so that $I_i$ is to the left of $I_{i+1}$ and the gap between them is at least $L$ for some $L>0$.
Let $A_1,A_2,...,A_m$ be sets of the same probability $p=\bbP(A_i)$ so that $A_i$ is measurable with respect to $\sig\{X_j:\,j\in I_i$\}. 
Then 
$$
\bbP\left(\bigcap_{i=1}^{m} A_i\right)\leq p^{m}+\al_P(L)\sum_{j=0}^{m-2}p^j\leq
 p^{m}+\al_P(L)\frac{1}{1-p}.
$$
\end{lemma}
\begin{proof}
We will prove the lemma by induction on $m$. For $m=2$ by the definition of $\al_U(\cdot)$ we have 
$$
\bbP(A_1\cap A_2)\leq \bbP(A_1)\bbP(A_2)+\al_P(L)
$$
which coincides with the desired upper bound for $m=2$. Next, suppose that the lemma is true for some $m\in\bbN$ and
 let $I_1,...,I_{m+1}$ be sets with minimal gap greater or equal to some $L$, and measurable sets $A_1,...,A_{m+1}$ with the same probability $p$ so that $A_i$ is measurable with respect to $\sig\{X_j:\,j\in I_i\}$.
   Then by the definition of $\al_U(\cdot)$ we have 
 $$
\bbP\left(\bigcap_{i=1}^{m+1} A_i\right)\leq
\bbP\left(\bigcap_{i=1}^{m} A_i\right)\bbP(A_{m+1})+\al_P(L)
$$
$$
\leq \left(p^{m}+\al_P(L)\sum_{j=0}^{m-2}p^j\right)p+\al_P(L)=
p^{m+1}+\al_P(L)\sum_{j=0}^{m-1}p^j=p^{m+1}+\al_P(L)\sum_{j=0}^{(m+1)-2}p^j
$$
where in the last inequality we have used the induction hypothesis with the sets $A_1,...,A_m$ and that $\bbP(A_{m+1})=p$.
\end{proof}

\begin{corollary}\label{Cor3}
Under Assumption \ref{Sets Ass},
condition \eqref{Series} holds true under the assumption
that 
\begin{equation}\label{Conv1}
\sum_{j}(\ln j\beta_{C j/(3\ln j)})^{1-2/p}<\infty\,\,  \text{ and }\,\,  \sum_{j}(\al_U(j/(3C\ln) j))^{1-2/p}<\infty
\end{equation}
for some constant $C$ so that $C|\ln\big(1-\bbP(A)/2\big)|(1-2/p)>1$.
\end{corollary}
\begin{proof}
First, for all integers $s\geq 1$ we have
$$
\bbP(n_A>j)=\bbP\left(\bigcap_{k=1}^{j}\te^{-k}(\Om\setminus A)\right)\leq 
\bbP\left(\bigcap_{k=1}^{[j/s]}\te^{-ks}(\Om\setminus A)\right).
$$
Now, let us take $s$ of the form $s=3r$ for $r\in\bbN$. Then 
$$
\bbP\left(\bigcap_{k=1}^{[j/s]}\te^{-ks}(\Om\setminus A)\right)\leq 
\bbP\left(\bigcap_{k=1}^{[j/s]}\te^{-ks}(\Om\setminus A_r)\right)+[j/s]\beta_r
$$
where $A_r$ and $\beta_r$ come from Assumption \ref{Sets Ass}. Thus,
\begin{equation}\label{UPB}
\bbP(n_A>j)=\bbP\left(\bigcap_{k=1}^{j}\te^{-k}(\Om\setminus A)\right)\leq 
\bbP\left(\bigcap_{k=1}^{[j/s]}\te^{-ks}(\Om\setminus A_r)\right)+[j/s]\beta_r.
\end{equation}
Next, by Lemma \ref{L alpha} we have
$$
\bbP\left(\bigcap_{k=1}^{[j/s]}\te^{-ks}(\Om\setminus A_r)\right)\leq \left(1-\bbP(A_r)\right)^{[j/s]}+\frac{\al_U(r)}{1-\bbP(A_r)}.
$$
Next, let us take $s$ of the form  $s=s_j=C^{-1}[j/\ln j]$ for some $C>0$. Using that $\lim_{r\to\infty}\bbP(A_r)=\bbP(A)>0$ we get that for all $j$ large enough we have $\frac12(1-\bbP(A))\leq 1-\bbP(A_r)\leq1-\frac12\bbP(A)$. We thus see that for $j$ large enough we have
$$
\bbP(n_A>j)\leq \left(1-\frac12\bbP(A)\right)^{C\ln j}+\frac2{1-\bbP(A)}\al_U(C[j/\ln j])+[j/s]\beta_r.
$$
Now let us take $C$ so that $C|\ln\big(1-\frac12\bbP(A)\big)|(1-2/p)>1$. Then the series $\sum_j\left(1-\frac12\bbP(A)\right)^{C(1-2/p)\ln j}$ 
converges and now the convergence of the series in \eqref{Series} follows from \eqref{Conv1}.
\end{proof}

\begin{proof}[Proof of Theorem \ref{CLT} under Assumption (M1)]
By Corollary \ref{Cor3} condition \eqref{Series} in Lemma \ref{LV} is valid. The proof of Theorem \ref{CLT} in this case follows now by combining Lemma \ref{LV} and 
\cite[Theorem 2.3]{Kifer 1998}.
\end{proof}

\subsubsection{Proof Theorem \ref{CLT} under Assumption (M2)}

\begin{lemma}\label{L phi}
Let $I_1,I_2,...,I_m$, $m\geq 2$ be finite subsets of $\bbN$ so that $I_i$ is to the left of $I_{i+1}$ and the gap between them is at least $L$ for some $L>0$.
Let $A_1,A_2,...,A_m$ be sets of the same probability $p=\bbP(A_i)$ so that $A_i$ is measurable with respect to $\sig\{X_j:\,j\in I_i\}$.
Then 
$$
\bbP\left(\bigcap_{i=1}^{m} A_i\right)\leq \left(p+\phi_U(L)\right)^{m-1}.
$$
\end{lemma}
\begin{proof}
We will prove the lemma by induction on $m$. For $m=2$ the lemma follows from the definition of $\phi_U$. Now, suppose that the lemma is true for some $m$. Let $I_1,...,I_{m+1}$ be sets with minimal gap greater or equal to some $L$, and measurable sets $A_1,...,A_{m+1}$ with the same probability $p$ so that $A_i$ is measurable with respect to $\sig\{X_j:\,j\in A_i\}$. Then by the definition of $\phi_U$ we have 
$$
\bbP\left(\bigcap_{i=1}^{m} A_i\cap A_{m+1}\right)\leq \bbP\left(\bigcap_{i=1}^{m} A_i\right)\left(\bbP(A_{m+1})+\phi_U(L)\right)
$$
and not the proof of the induction step is completed by using the induction hypothesis with the sets $A_1,...,A_m$.
\end{proof}
\begin{corollary}\label{Cor2}
Suppose that 
$$
\limsup_{r\to\infty}\phi_U(r)<\bbP(A)
$$
and  that $\sum_{j}(\ln j\beta_{j/(3C\ln j)})^{1-2/p}<\infty$ for some $C$ so that $C|\ln \del|(1-p/2)>1$, where $\del=1-\bbP(A)+\limsup_{r\to\infty}\phi_U(r)$. Then the series on the left hand side of \eqref{Series} converges. 
\end{corollary}
\begin{proof}
As in the beginning of the proof of Corollary \ref{Cor3},
for every $s\in\bbN$ of the form $s=3r$ we have
$$
\bbP(n_A>j)=\bbP\left(\bigcap_{k=1}^{j}\te^{-k}(\Om\setminus A)\right)\leq 
\bbP\left(\bigcap_{k=1}^{[j/s]}\te^{-ks}(\Om\setminus A_r)\right)+[j/s]\beta_r.
$$
Now, by Lemma \ref{L phi} we have
$$
\bbP\left(\bigcap_{k=1}^{[j/s]}\te^{-ks}(\Om\setminus A_r)\right)\leq \left(1-\bbP(A_r)+\phi_U(r)\right)^{[j/s]}.
$$
Next, since  $\lim_{r\to\infty}\bbP(A_r)=\bbP(A)$ and $\limsup_{r\to\infty}\phi_U(r)<\bbP(A)$ 
we see that 
$$
\limsup_{r\to\infty}\left(1-\bbP(A_r)+\phi_U(r)\right)=\del<1.
$$
Thus, if we take $s$ of the form  $s=s_j=C^{-1}[j/\ln j]$, then for $j$ large enough we have 
$$
\bbP(n_A>j)\leq \del^{[j/s]}+[j/s]\beta_{[s/3]}.
$$
If we take $C$ so that $C|\ln\del|(1-p/2)>1$ we get that the series $\sum_j \del^{j(1-2/p)/s_j}$ converges. Now the proof of the lemma is complete  since the series $\sum_j([j/s_j]\beta_{[s_j/3]})^{1-2/p}$ converges by the assumptions of the lemma.
\end{proof}

\begin{proof}[Proof of Theorem \ref{CLT} under Assumption (M2)]
By Corollary \ref{Cor2} condition \eqref{Series} in Lemma \ref{LV} is valid. Now the proof of Theorem \ref{CLT} under (M2) follows  by combining Lemma \ref{LV} and 
\cite[Theorem 2.3]{Kifer 1998}.
\end{proof}

\subsubsection{Proof Theorem \ref{CLT} under Assumption (M3)}
We first need the following result.
\begin{lemma}\label{L psi}
Let $I_1,I_2,...,I_m$, $m\geq 2$ be finite subsets of $\bbN$ so that $I_i$ is to the left of $I_{i+1}$ and the gap between them is at least $L$ for some $L>0$.
Let $A_1,A_2,...,A_m$ be sets so that $A_i$ is measurable with respect to $\sig\{X_j:\,j\in I_i\}$. 
Then 
$$
\bbP\left(\bigcap_{i=1}^{m} A_i\right)\leq (1+\psi_U(L))^{m-1}\prod_{j=1}^{m}\bbP(A_i).
$$
Hence, if $\bbP(A_i)=p$ for all $i$ and some $p$ then 
$$
\bbP\left(\bigcap_{i=1}^{m} A_i\right)\leq p\left(p(1+\psi_U(L))\right)^{m-1}.
$$
\end{lemma}
\begin{proof}
The lemma follows directly by induction and the definition of $\psi_U$.
\end{proof}

 \begin{corollary}\label{Cor1}
Suppose that $\limsup_{k\to\infty}\psi_U(k)<\frac 1{1-\bbP(A)}-1$ and that $\sum_{j}(\ln j\beta_{j/(3C\ln j)})^{1-2/p}<\infty$ for some $C$ so that $C|\ln\del|(1-2/p)>1$, where 
$\del=\left(1+\limsup_{r\to\infty}\psi_U(r)\right)\left(1-\bbP(A)\right)<1$.
 Then the series on the left hand side of \eqref{Series} converges. 
\end{corollary}
\begin{proof}
As in the beginning of the proof of Corollary \ref{Cor3},
for all $s=3r$ we have
$$
\bbP(n_A>j)=\bbP\left(\bigcap_{k=1}^{j}\te^{-k}(\Om\setminus A)\right)\leq 
\bbP\left(\bigcap_{k=1}^{[j/s]}\te^{-ks}(\Om\setminus A_r)\right)+[j/s]\beta_r.
$$
Next, by applying Lemma \ref{L psi}, we have 
$$
\bbP\left(\bigcap_{k=1}^{[j/s]}\te^{-ks}(\Om\setminus A_r)\right)\leq 
 (1+\psi_U(r))^{[j/s]-1}(1-\bbP(A)+\beta_r)^{[j/s]}:=q_{s,j}
$$
where we have used that $|\bbP(A)-\bbP(A_r)|\leq \beta_r$. Now, let us take $s=s_j=C^{-1}[j/\ln j]$ for some $C>0$. Then when $j$ is large enough we see that 
$$
\left(1+\psi_U(r)\right)\left(1-\bbP(A)+\beta_r\right)\leq \del+\ve<1
$$
for some $\ve$ small enough, where $\del=\left(1+\limsup_{r\to\infty}\psi_U(r)\right)\left(1-\bbP(A)\right)<1$. Thus, if also $C|\ln\del|(1-2/p)>1$, by taking a  sufficiently small $\ve$ we get that  both series $\sum_j (q_{s_j,j})^{1-2/p}$ and $\sum_{j}([j/s_j]\beta_{[s_j/3]})^{1-2/p}$ converge, and the proof of the corollary is complete.
\end{proof}

\begin{proof}[Proof of Theorem \ref{CLT} under Assumption (M3)]
By the previous corollary condition \eqref{Series} in Lemma \ref{LV} is valid. The proof of Theorem \ref{CLT} in this case follows now by combining Lemma \ref{LV} and 
\cite[Theorem 2.3]{Kifer 1998}.
\end{proof}

\begin{remark}
The proofs of Corollaries \ref{Cor3}, \ref{Cor2} and \ref{Cor1} show that if  $A$ is measurable with respect to $\sig\{X_j, |j|\leq d\}$ for some $d$ then $\bbP(n_A>j)$ decays exponentially fast in $j$ under the other assumptions of the corollaries (since we can take $\beta_r=0$ if $r>d$).
\end{remark}

%Another example will be minimal systems (which assign positive mass to open sets). If $A$ contains an open set then $n_1$ is bounded. 
 
\subsection{A second approach to the CLT and LIL: a direct proof of Theorem \ref{CLT2}}\label{CLT 2 pf}
The idea in the proof of Theorem \ref{CLT2} is to verify the conditions of \cite[Theorem 2.3]{Kifer 1998} when $Q=\Om$, namely when there is no actual inducing involved. This requires us to verify the following three conditions:

\begin{equation}\label{Ver1 2}
\left\|c(\om)\right\|_{L^2(\bbP)}<\infty,\, c(\om)=\|\tilde u_\om\|
\end{equation}

\begin{equation}\label{Ver2 2}
\left\|\sum_{n=0}^\infty|\bbE_{\mu_\om}[\tilde u_\om\cdot\tilde u_{\te^n\om}\circ T_\om^n]|\right\|_{L^1(\bbP)}<\infty
\end{equation}
and
\begin{equation}\label{Ver3 2}
\left\|\sum_{n=0}^{\infty}\bbE_{\mu_\om}(|L_{\te^{-n}\om}^n\tilde u_{\te^{-n}\om}|)\right\|_{L^2(\bbP)}<\infty.
\end{equation}

Next, recall our assumption about the existence of a  sequence $\beta_r$  so that $\beta_r\to 0$ and for every $r$ there is a random variable $\rho_r(\om)$ which is measurable with respect to $\sig\{X_j: |j|\leq r\}$ so that 
\begin{equation}\label{beta r inf}
\|\rho-\rho_r\|_{L^\infty}\leq \beta_r.
\end{equation}

 The first condition \eqref{Ver1 2} is a part of the assumptions of Theorem \ref{CLT2}.
 In order to verify conditions \eqref{Ver2 2} and \eqref{Ver3 2}
we first need the following result.
 \begin{lemma}\label{psi Lemm 2}
 Let $I_1,...,I_d$ be intervals in the positive integers so that $I_j$ is to the left of $I_{j+1}$ and the distance between them is at least $L$. Let $Y_1,...,Y_d$ be nonnegative bounded random variables so that $Y_i$ is measurable with respect to $\sig\{X_k: k\in I_i\}$. Then 
 $$
 \bbE\left[\prod_{i=1}^{d}Y_i\right]\leq\left(1+\psi_U(L)\right)^{d-1}\prod_{i=1}^{d}\bbE[Y_i].
 $$
 \end{lemma}
 \begin{proof}
Once we prove the lemma for $d=2$ the general case will follow by induction. Let us assume that $d=2$. Next, we have 
$$
Y_i=\lim_{n\to\infty}Y_i(n)=\lim_{n\to\infty}\sum_k\bbI((k-1)2^{-n}<Y_i\leq k2^{-n})k2^{-n}
$$ 
and so with $\al_i(k,n)=\{(k-1)2^{-n}<Y_i\leq k2^{-n}\}$, by the monotone convergence theorem we have
$$
\bbE[Y_1Y_2]=\lim_{n\to\infty}\bbE[Y_1(n)Y_2(n)]=\lim_{n\to\infty}\sum_{k_1,k_2}(2^{-n}k_1)(2^{-n}k_2)\bbP(\al_1(k,n)\cap\al_2(k,n))$$$$\leq 
\lim_{n\to\infty}\sum_{k_1,k_2}(2^{-n}k_1)(2^{-n}k_2)(1+\psi_U(L))\bbP(\al_1(k,n))\bbP(\al_2(k,n))$$$$=
(1+\psi_U(L))\lim_{n\to\infty}\bbE[Y_1(n)]\bbE[Y_2(n)]=\left(1+\psi_U(L)\right)\bbE[Y_1]\bbE[Y_2]
$$
where in the above inequality we have used the definition of the upper mixing coefficients $\psi_U(\cdot)$.
 \end{proof}

Next, we need the following
\begin{lemma}\label{L2}
Suppose that 
$$
\lim_{s\to\infty}\Psi_{U}(s)<\infty
$$
and that with some $\del>0$ we have $\|u_\om\|, B_\om\in L^{3+\del}(\Om,\cF,\bbP)$.
Then conditions \eqref{Ver2 2} and \eqref{Ver3 2} are in force. 
\end{lemma}
\begin{proof}[Proof of Theorem \ref{CLT2}]
The proof of Theorem \ref{CLT2} is completed now by combining Lemma \ref{L2} with \cite[Theorem 2.3]{Kifer 1998} in the case $Q=\Om$.
\end{proof}

\begin{proof}[Proof of Lemma \ref{L2}]
Since $0<\rho(\cdot)<1$ we have $\lim_{q\to\infty}\rho^q=0$ and so 
by the monotone convergence theorem
$$
\lim_{q\to\infty}\bbE_{\bbP}[\rho^q]=0.
$$
Thus, since the limit superior of $\psi_U$ is finite, if $q$ is large enough then 
we have
$$
\limsup_{r\to\infty}\Psi_{U}(r)<\frac1{\bbE_{\bbP}[\rho^q]}-1.
$$
Let us take $q$ large enough so that its conjugate exponent $p$ satisfies $3p\leq 3+\del$, where $\del$ comes from the assumptions of the lemma (and Theorem \ref{CLT2}).

Next, to show that condition \eqref{Ver2 2} is in force, let us fix some $n\geq0$. We first note that by \eqref{DEC} we have
$$
|\bbE_{\mu_\om}[\tilde u_\om\cdot\tilde u_{\te^n\om}\circ T_\om^n]|\leq \|\tilde u_\om\|\|\tilde u_{\te^n\om}\|_{L^1(\mu_{\te^n\om})}B_{\te^n\om}\prod_{j=0}^{n-1}\rho(\te^j\om).
$$
Next,  by applying the generalized H\"older inequality with the exponents $q_1=q_2=q_3=3p$ and $q_4=q$ we get that
$$
\bbE_{\bbP}\left[\|\tilde u_\om\|\|\tilde u_{\te^n\om}\|_{L^1(\mu_{\te^n\om})}B_{\te^n\om}\prod_{j=0}^{n-1}
\rho(\te^j\om)\right]\leq \|c(\cdot)\|_{3p}^2\|B_{\om}\|_{3p}
\left(\bbE_{\bbP}\left[\prod_{j=0}^{n-1}\rho^q(\te^j\om)\right]\right)^{1/q}
$$
where we have used that $\|\tilde u_{\te^n\om}\|_{L^1(\mu_{\te^n\om})}\leq c(\om)$.
Now, for all $s$ of the form $s=3r$ have 
$$
\bbE_{\bbP}\left[\prod_{j=0}^{n-1}\rho^q(\te^j\om)\right]\leq \bbE_{\bbP}\left[\prod_{j=1}^{[(n-1)/s]}\rho^q(\te^{js}\om)\right]\leq 
\bbE_{\bbP}\left[\prod_{j=1}^{[(n-1)/s]}(\rho_r^q(\te^{js}\om)+C_q\beta_r)\right]
$$
$$
\leq
\left(1+\psi_U(r)\right)^{[(n-1)/s]-1}\prod_{j=1}^{[(n-1)/s]}(\bbE_{\bbP}[\rho_r^q]+C_q\beta_r)
$$
where in the last inequality we have used Lemma \ref{psi Lemm 2}, and $C_q$ is a constant that depends only on $q$. Since $\lim_{r\to\infty}\bbE_{\bbP}[\rho_r^q]=\bbE_{\bbP}[\rho^q]$, $\lim_{r\to\infty}\beta_r=0$ and $(1+\psi_U(r))\bbE_{\bbP}[\rho^{q}]<1$, by fixing a sufficiently large $s=s_0$ we conclude that
$$
\bbE_{\bbP}[|\bbE_{\mu_\om}[\tilde u_\om\cdot\tilde u_{\te^n\om}\circ T_\om^n]|]\leq  C(1-\ve)^n
$$
for some constants $\ve\in(0,1)$ and $C>0$ (which depend on $s_0$ and $q$), and thus Condition \eqref{Ver2 2} is in force.

Next, in order to verify condition \eqref{Ver3 2},  by Theorem \ref{RPF} we have
$$
|L_{\te^{-n}\om}^n\tilde u_{\te^{-n}\om}|\leq B_\om\|\tilde u_{\te^{-n}\om}\|\prod_{j=1}^{n}\rho(\te^{-j}\om)
$$
and so by the H\"older inequality,
$$
\left\|\bbE_{\mu_\om}(|L_{\te^{-n}\om}^n\tilde u_{\te^{-n}\om}|)\right\|_{L^2(\bbP)}\leq 
\|B_\om\|_{L^{2p}(\bbP)}\|\tilde u_\om\|_{L^{2p}(\bbP)}\left\|\prod_{j=1}^{n}\rho(\te^{-j}\om)\right\|_{L^{q}(\bbP)}.
$$
Due to stationarity we have
$$
\left\|\prod_{j=1}^{n}\rho(\te^{-j}\om)\right\|_{L^q(\bbP)}=\left\|\prod_{j=0}^{n-1}\rho(\te^{j}\om)\right\|_{L^q(\bbP)}=\left(\bbE_{\bbP}\left[\prod_{j=0}^{n-1}\rho^q(\te^j\om)\right]\right)^{1/q}
=O((1-\ve)^n)
$$
where the last estimates was obtained in the course of the proof of \eqref{Ver2 2}. This completes the proof of \eqref{Ver3 2}.
\end{proof}

\subsection{An almost sure invariance principle: proof of Theorem \ref{ASIP}}\label{ASIP pf}
 Let $\beta_r$ satisfy \eqref{beta r inf}. 
 
 \subsubsection{Key auxiliary result}
 Before proving Theorem \ref{ASIP} we need the following  result.
 
\begin{lemma}\label{MomLem}
Under the Assumptions of Theorem \ref{ASIP} we have the following.

(i) Let $R_n(\om)=\sum_{j=0}^{n-1}\rho(\te^j\om)\cdots \rho(\te^{n-1}\om)$.
Then for every $p\in\bbN$, 
\begin{equation}\label{Mom0}
\bbE_{\bbP}[R_n^p]\leq C_p
\end{equation}
for some constant $C_p$ which does not depend on $n$. Therefore, for every $\ve>0$ we have
$$
R_n(\om)=o(n^\ve),\, \bbP-\text{a.s.}
$$

(ii) For every pair of positive integers $(n,m)$ such that $m\leq n$ let 
$$
R_{m,n}(\om)=\sum_{k=m}^{n}\sum_{j=k}^{n}\rho(\te^k\om)\cdot \rho(\te^{k+1}\om)\cdots \rho(\te^{j}\om)=\sum_{m\leq k\leq j\leq n} \rho(\te^{k}\om)\cdots \rho(\te^{j}\om).
$$
Then for every $\ell\in\bbN$, $\ve>0$ and a positive integer $p$ we have 
\begin{equation}\label{Mom}
\bbE_\bbP\left[\sup_{(n,m):\,0\leq n-m\leq \ell}n^{-(1+\ve)}R_{m,n}^p\right]\leq C_{p,\ve}\ell^{1+p}
\end{equation}
for some constant $C_{p,\ve}>0$ which depends only on $p$ and $\ve$.
Therefore, $\bbP$-a.s. for every $\ve>0$, uniformly in $n$ and $m$ as $(n-m)\to\infty$ we have
$$
n^{-\ve}R_{m,n}(\om)=O\left((n-m)^{1+\ve}\right),\,\,\bbP\text{-a.s.}
$$
\end{lemma}
\begin{proof}
(i) First, the almost sure estimate $R_n(\om)=o(n^\ve)$ follows from \eqref{Mom0} and the Borel Cantelli Lemma. Indeed, by taking $p>\frac 1\ve$ and applying the Markov inequality we arrive at
$$
\bbP(R_n\geq n^\ve)=\bbP(R_n^p\geq n^{\ve p})\leq C_pn^{-p\ve}.
$$
In order to prove \eqref{Mom0}, let us take $s\in\bbN$ of the form $s=3r$. Then, since $0<\rho(\cdot)<1$, 
$$
\bbE_{\bbP}[R_n^p]=\sum_{0\leq j_1\leq j_2\leq...\leq j_p<n}\bbE_{\bbP}\left[\prod_{k=1}^{p}\prod_{u=j_{k}}^{n-1}\rho(\te^u\om)\right]
\leq \sum_{0\leq j_1\leq j_2\leq...\leq j_p<n}\bbE_{\bbP}\left[\prod_{u=j_1}^{n-1}\rho(\te^u\om)\right]
$$
$$
\leq  \sum_{0\leq j_1\leq j_2\leq...\leq j_p<n}\bbE_{\bbP}\left[\prod_{v=0}^{[(n-1-j_1)/s]}\rho(\te^{j_1+sv}\om)\right]
\leq  \sum_{0\leq j_1\leq j_2\leq...\leq j_p<n}\bbE_{\bbP}\left[\prod_{v=0}^{[(n-1-j_1)/s]}(\rho_r(\te^{j_1+sv}\om)+\beta_r)\right]
$$
$$
\leq \sum_{0\leq j_1\leq j_2\leq...\leq j_p<n}(1+\psi_U(r))^{[(n-j_1-1)/s]}\prod_{v=0}^{[(n-1-j_1)/s]}\bbE_{\bbP}\left[(\rho_r(\te^{j_1+sv}\om)+\beta_r)\right]
$$
$$
= \sum_{0\leq j_1\leq j_2\leq...\leq j_p<n}(1+\psi_U(r))^{[(n-j_1-1)/s]}a_r^{[(n-1-j_1)/s]+1}
$$
where  $a_r=\bbE[\rho_r]+\beta_r$ and in the last inequality we have used Lemma \ref{psi Lemm 2}. Taking $s$ large enough so that $a_r(1+\psi_U(r))=\del<1$ (using \eqref{LS}) and using that $n-1-j_1\geq n-1-j_i$ for $i=1,2,...,d$ we conclude that 
 $$
 \bbE_{\bbP}[R_n^p]\leq \sum_{0\leq j_1\leq j_2\leq...\leq j_p<n}b^{\sum_{i=1}^{d}(n-1-j_i)}=\left(\sum_{j=0}^{n-1}b^{n-1-j}\right)^p\leq C_p(b)=\frac{1}{(1-b)^p}
 $$ 
 where $b=b_{p,s}=\del^{1/sp}\in(0,1)$.
\vskip0.1cm

(ii) First, the almost sure estimate $R_{m,n}(\om)=O(n^{1/p+\ve}(n-m)^{1+\ve})$ follows from \eqref{Mom} and the Borel Cantelli Lemma. Indeed, for all $A>0$ we have 
$$
\bbP\left(\sup_{(n,m):\,0\leq n-m\leq \ell}n^{-(1+\ve)}R_{m,n}^p\geq A^p\right)=O(\ell^{p+1}A^{-p})
$$
and so for $A_\ell=\ell^{1+\frac{3}{p}}$  we have 
$$
\bbP\left(\sup_{(n,m):0\,\leq n-m\leq \ell}n^{-(1+\ve)/p}R_{m,n}\geq A_\ell\right)\leq C\ell^{-2}.
$$
Now, given $\ve>0$,  by taking $p$ large enough  we  conclude from the  Borel Cantelli Lemma that 
$$
\sup_{(n,m):\,0\leq m-n\leq \ell}n^{-(1+\ve)/p}R_{m,n}=O(\ell^{1+\ve}),\,\,\bbP\text{-a.s}
$$
 Thus, $\bbP$-a.s. there is a constant $C$ so that for a given $n$ and $m$ with $n-m$ large enough we have $R_{m,n}\leq C(n-m)^{1+\ve}n^{1/p+\ve/p}$. Finally, by taking $p$ large enough we can also insure that $(1+\ve)/p<\ve$.

Next, in order to prove \eqref{Mom}, we have 
\begin{equation}\label{Tg}
\bbE_\bbP\left[\sup_{(n,m):\,0\leq n-m\leq \ell}n^{-(1+\ve)}R_{m,n}^p\right]\leq \sum_{(n,m):\, 0\leq n-m\leq \ell}n^{-(1+\ve)}\bbE[R_{m,n}^p]=\sum_{n=1}^{\infty}n^{-(1+\ve)}\sum_{m=n-\ell}^{n}\bbE[R_{m,n}^p]
\end{equation}
$$
\leq\left(\sum_{n=1}^{\infty}n^{-(1+\ve)}\right)(\ell+1)\sup_{(n,m):\, n-\ell\leq m\leq n}\bbE[R_{m,n}^p]
\leq C_\ve \ell\sup_{(n,m):\, n-\ell\leq m\leq n}\bbE[R_{m,n}^p].
$$
Next, let us estimate $\bbE[R_{m,n}^p]$ 
for a fixed pair of positive integers $(n,m)$ so that $n-\ell\leq m\leq n$. We first write
$$
\bbE_{\bbP}[R_{m,n}^p]=\sum_{m\leq k_1,...,k_p\leq n}\,\sum_{k_i\leq j_i\leq n;\, 1\leq i\leq p}\bbE_\bbP\left[\prod_{i=1}^{p}\prod_{u=k_i}^{j_i}\rho(\te^{u}\om)\right].
$$
For a fixed choice of pairs $(k_i,j_i), i=1,2,...,p$ let $a=a(\{(k_i,j_i):\,1\leq i\leq p\})$ be an index so that $j_a-k_a$ is the largest among $j_i-k_i$. Since $0<\rho(\om)<1$, by disregarding the products $\prod_{u=k_i}^{j_i}\rho(\te^{u}\om)$ for $i\not=a$ we see that  
$$
\bbE_{\bbP}[R_{m,n}^p]\leq \sum_{m\leq k_1,...,k_p\leq n}\,\sum_{k_i\leq j_i\leq n;\, 1\leq i\leq p}\bbE_{\bbP}\left[\prod_{u=k_a}^{j_a}\rho(\te^{u}\om)\right].
$$
Next, since $0<\rho(\cdot)<1$, for all $s$ of the form $s=3r$, by \eqref{beta r inf} and Lemma \ref{psi Lemm 2} we have
$$
\bbE_\bbP\left[\prod_{u_a=k_a}^{j_a}\rho(\te^{u_a}\om)\right]\leq 
\bbE_\bbP\left[\prod_{u=0}^{[(j_a-k_a)/s]}\rho(\te^{k_a+su}\om)\right]
$$
$$
\leq
\bbE_\bbP\left[\prod_{u=0}^{[(j_a-k_a)/s]}(\rho_r(\te^{k_a+su}\om)+\beta_r)\right]
\leq
(1+\psi_U(r))^{[(j_a-k_a)/s]}\left(\bbE_\bbP[\rho_r]+\beta_r\right)^{[(j_a-k_a)/s]+1}.
$$
Since $\limsup_{r\to\infty}(1+\psi_U(r))<\frac1{\bbE[\rho]}$ (by \eqref{LS}),
by fixing some $s=s_0$ large enough we get that $(1+\psi_U(r))(\bbE_\bbP[\rho_r]+\beta_r)=\del<1$. Thus, since $j_a-k_a$ is the maximal difference, we have
$$
\bbE_\bbP\left[\prod_{u_a=k_a}^{j_a}\rho(\te^{u_a}\om)\right]\leq \del^{[(j_a-k_a)/s]}\leq \ve^{\sum_{i=1}^{p}(j_i-k_i)}
$$ 
where $\ve=\ve_{p,s}=\del^{\frac{1}{ps}}<1$. Hence for all $n,m$ so that $n-\ell\leq m\leq n$ we have
$$
\bbE[R_{m,n}^p]\leq
\sum_{m\leq k_1,...,k_p\leq n}\,\sum_{k_i\leq j_i\leq n;\, 1\leq i\leq p}\ve^{\sum_{i=1}^{p}(j_i-k_i)}
=\left(\sum_{i=1}^{p}\sum_{k_i=m}^{n}\sum_{j_i=k_i}^{j_i}\ve^{j_i-k_i}\right)^p=O((m-n)^p)=O(\ell^p)
$$
which together with \eqref{Tg} completes the proof of the maximal moment estimates in (ii).
\end{proof}

\subsubsection{A martingale co-boundary representation}
Let $\tilde u_\om=u_\om-\mu_\om(u_\om)$. Set
$$
G_{\om,n}=\sum_{j=0}^{n-1}L_{\te^j\om}^{n-j}(\tilde u_{\te^j\om})
$$
and
$$
M_{\om,n}=\tilde u_{\te^n\om}+G_{\om,n}-G_{\om,n+1}\circ T_{\te^n\om}.
$$
Then for every fixed $\om$ we have that $M_{\om,n}\circ T_\om^n$ is a reverse martingale difference with respect to the reverse filtration $\cT_\om^n=(T_\om^n)^{-1}(\cB_\om)$, where $\cB_\om$ is the Borel $\sig$-algebra on $\cE_\om$ (see \cite[Proposition 2]{Davor ASIP}).

\begin{lemma}\label{M norm}
If $\om\to C_\om$, $\om\to B_{\om}$, $\om\to N(\om)$ and $\om\to \|\tilde u_\om\|$ are in $L^p(\Om,\cF,\bbP)$ for some $p$ then for every $\ve>0$
for $\bbP$-a.e. $\om$ we have 
$$
\|M_{\om,n}^2\|=O(n^{8/p+\ve}).
$$
\end{lemma}

\begin{proof}
First by Theorem \ref{RPF}, 
\begin{equation}\label{G est}
\|G_{\om,n}\|\leq B_{\te^n\om}\sum_{j=0}^{n-1}\|\tilde u_{\te^j\om}\|\rho(\te^j\om)\cdots \rho(\te^{n-1}\om)\leq B_{\te^n\om}u_n(\om)R_n(\om)
\end{equation}
where $u_n(\om)=\sup_{j\leq n}\|\tilde u_{\te^j\om}\|$ and $R_n$ is defined in Lemma \ref{MomLem} (i).
Therefore by the definition  \eqref{N def} of $N(\cdot)$ we have
$$
\|G_{\om,n+1}\circ T_{\te^n\om}\|\leq \|G_{n+1,\om}\| N(\te^n\om)\leq 
B_{\te^{n+1}\om}N(\te^n\om)u_{n+1}(\om)R_{n+1}(\om).
$$
We thus conclude that 
$$
\|M_{\om,n}^2\|\leq 3\|M_{\om,n}\|^2\leq A\left(B_{\te^n\om}+B_{\te^{n+1}\om}\right)^2\left(1+N(\te^n\om)\right)^2u_{n+1}^2(\om)\left(R_n(\om)+R_{n+1}(\om)\right)^2
$$
where $A$ is an absolute constant.  Now the lemma follows by
Lemma \ref{MomLem} (i) together with the fact that for any random variable $Q(\om)$, if $Q(\om)\in L^p(\Om,\cF,\bbP)$ then  $|Q(\te^n\om)|=o(n^{1/p})$,\, $\bbP$-almost surely (as a consequence of the mean ergodic theorem).
\end{proof} 

Next, we need the following quadratic variation estimates.

\begin{lemma}\label{QV}
If $\om\to C_\om$, $\om\to U_{\om}$, $\om\to N(\om)$ and $\om\to \|\tilde u_\om\|$ are in $L^p(\Om,\cF,\bbP)$
for some $p$ then for every $\ve>0$, for $\bbP$-a.e. $\om$ we have
$$
\sum_{k=0}^{n-1}\bbE_{\mu_\om}[M_{\om,k}^2\circ T_\om^k|\cT_\om^{k+1}]=\sum_{k=0}^{n-1}\bbE_{\mu_\om}[M_{\om,k}^2\circ T_\om^k]+o(n^{1/2+9/p+\ve}\ln^{3/2+\ve}n),\, \mu_\om\text{ a.s.}
$$
\end{lemma}
\begin{proof}
Set 
$$
A_{k,\om}=\bbE_{\mu_\om}[M_{\om,k}^2\circ T_\om^k|\cT_\om^{k+1}],\, B_{k,\om}=\bbE_{\mu_\om}[A_{k,\om}]=\bbE_{\mu_\om}[M_{\om,k}^2\circ T_\om^k]=\mu_{\te^k\om}(M_{\om,k}^2)
$$
and $Y_{k,\om}=A_{k,\om}-B_{k,\om}$.
Then, by \cite[Lemma 9]{CIRM paper} in order to prove the lemma it is enough to show that for all $n>m$ and all $\ve>0$ we have 
\begin{equation}\label{Q}
\left\|\sum_{k=m}^{n}Y_{k,\om}\right\|_{L^2(\mu_\om)}^2\leq C(n-m)n^{2/q+\ve}
\end{equation}
where $C$ is a constant which  may depend on $\om$ and $\ve$.

In order to prove \eqref{Q}, 
we first have
$$
\left\|\sum_{k=m}^{n}Y_k\right\|_{L^2(\mu_\om)}^2=\text{Var}\left(\sum_{k=m}^{n}A_k\right)=
\left\|\sum_{k=m}^{n}A_k\right\|_{L^2(\mu_\om)}^2-\left(\sum_{k=m}^n B_k\right)^2
$$
where we abbreviate $A_{k,\om}=A_k$ and $B_{k,\om}=B_k$.
Thus, 
$$
\left\|\sum_{k=m}^{n}Y_k\right\|_{L^2(\mu_\om)}^2\leq 
\sum_{k=m}^{n}\mu_\om(A_{k}^2)+2\left(\sum_{m\leq i<j\leq n}\mu_\om(A_{i}A_{j})-\sum_{m\leq i<j\leq n}B_{i}B_{j}\right).
$$
Next, arguing as in \cite[Lemma 6]{Davor ASIP}  we have
$$
A_i=L_{\te^i\om}(M_i^2)\circ T_\om^{i+1}
$$
where we abbreviate $M_i=M_{\om,i}$. Hence, by also using that $(T_\om^{i+1})_*\mu_{\om}=\mu_{\te^{i+1}\om}$ and that $L_\om$ is the dual of $T_\om$ (w.r.t. $\mu_\om$) we see that
$$
\mu_\om(A_{i}A_{j})=\int L_{\te^i\om}(M_i^2)\cdot (L_{\te^j\om}(M_j^2)\circ T_{\te^{i+1}\om}^{j-i})d\mu_{\te^{i+1}\om}
=\int L_{\te^{i}\om}^{j-i+1}(M_i^2)\cdot L_{\te^j\om}(M_j^2)d\mu_{\te^{j+1}\om}.
$$
Now, by \eqref{Exp L}  we have 
$$
\left\|L_{\te^{i}\om}^{j-i+1}(M_i^2)-\mu_{\te^i\om}(M_i^2)\right\|\leq B_{\te^j\om}\|M_i^2\|\rho(\te^i\om)\cdots \rho(\te^{j}\om).
$$
Since 
$$
\int L_{\te^j\om}(M_j^2)d\mu_{\te^{j+1}\om}=\bbE[A_{j}]=B_{j}
$$
and $B_{\te^j\om}=o(j^{2/p})$ (as $B_\om\in L^{2p}$)
we conclude from the above estimates together with Lemma \ref{M norm} that when $j>i$,
$$
|\mu_\om(A_{i}A_{j})-B_iB_j|\leq  B_{\te^j\om}\|M_i^2\|\|M_j^2\|\rho(\te^i\om)\cdots \rho(\te^{j}\om)=O(n^{18/p+\ve})\rho(\te^i\om)\cdots \rho(\te^{j}\om)
$$
for every $\ve>0$.
Thus,
$$
\left|\sum_{m\leq i<j\leq n}\mu_\om(A_{i}A_{j})-\sum_{m\leq i<j\leq n}B_{i}B_{j}\right|\leq 
Cn^{18/p+\ve}\sum_{m\leq i<j\leq n}\rho(\te^i\om)\cdots \rho(\te^{j}\om)\leq Cn^{18/p+\ve}R_{n,m}(\om).
$$
Now the proof of the lemma is completed using Lemma \ref{MomLem} (ii).
\end{proof}

\subsubsection{Proof of Theorem \ref{ASIP}}
First, we have 
$$
S_n^\om \tilde u=\sum_{j=0}^{n-1}M_{\om,j}\circ T_\om^j+G_{\om,n}\circ T_\om^n-G_{\om,0}.
$$
Next, by \eqref{G est}, Lemma \ref{MomLem} (i) and the assumption that $B_\om\in L^{2p}(\Om,\cF,\bbP)$ and $\|\tilde u_\om\|\in L^p(\Om,\cF,\bbP)$ we see that for every $\ve>0$,
$$
\|G_{\om,n}\|=o(n^{3/p+\ve}), \,\text{a.s.}
$$
and so
\begin{equation}\label{Mapprox}
\left\|S_n^\om \tilde u-\sum_{j=0}^{n-1}M_{\om,j}\circ T_\om^j\right\|_{L^\infty(\mu_\om)}=O(n^{3/p+\ve}).
\end{equation}
In particular, with $\sig_n^2=\sig_{n,\om}^2=\bbE_{\mu_\om}[(S_n^\om\tilde u)^2]$ and $\ve$ is small enough we have
$$
\sig_n^2:=\left\|\sum_{j=0}^{n-1}M_{\om,j}\circ T_\om^j\right\|_{L^2(\mu_\om)}^2=\sum_{j=0}^{n-1}\bbE_{\mu_{\te^j\om}}[M_{\om,j}^2]=\sig_n^2+O(n^{1/2+3/p+\ve})\asymp \sig^2 n
$$
where we have used that $\sig_n^2/n\to \sig^2>0$ and that $p>3$. 

In order to complete the proof we apply \cite[Theorem 2.3]{CM} (taking into account \cite[Remark 2.4]{CM}) with the reverse martingale difference $(M_{\om,n}\circ T_\om^n)$ and the sequence $a_n=n^{1/2+9/p+\ve}\ln^{3/2+\ve}n$, 
 noticing  that $\bbE_{\mu_{\te^n\om}}[M_{\om,n}^2]=O(n^{8/p+\ve})$ (by Lemma \ref{M norm}),  and so when $p>8$ we have $\bbE_{\mu_{\te^n\om}}[M_{\om,n}^2]=O(\sig_n^{2s})$ for some $0<s<1$.
Taking into account Lemma \ref{QV}, the first additional condition (i) of \cite[Theorem 2.3]{CM}  holds true. In order to verify the second additional condition (ii) with $v=2$, for $\bbP$-a.a. $\om$ we  have 
$$
\sum_{n\geq 1}a_n^{-2}\bbE_{\te^n\om}[M_{\om,n}^4]\leq C_\om\sum_{n\geq 1}a_n^{-2}n^{^{16/p+\ve}}
$$
which is a convergent series since $a_n^{-2}n^{^{16/p+\ve}}\leq n^{-1-2/p+2\ve}=O(n^{-1-\del}), \del>0$ (assuming that $\ve$ is small enough). In the above estimate we used that $\|M_{\om,n}^4\|\leq 3\|M_{\om,n}^2\|^2$ together with Lemma \ref{M norm}. We conclude that $\bbP$-a.s. there is a coupling of the reverse martingale $(M_{\om,n}\circ T_\om^n)$ with a Gaussian independent sequence $(Z_n)$, so that the ASIP rates in Theorem \ref{ASIP} hold true with $\sum_{j=0}^{n-1}M_{\om,n}\circ T_\om^j$ instead of $S_n^\om u-\mu_\om(S_n^\om u)=S_n^\om\tilde u$.  
Finally, in order to pass from the ASIP for the reverse martingale $(M_{\om,n}\circ T_\om^n)$ to the ASIP for the random Birkhoff sums $S_n^\om u$ we use \eqref{Mapprox} and then the, so-called, Berkes-Philipp lemma (which allows us to further couple $u_{\te^j\om}\circ T_\om^j$ with the Gaussian sequence).
\qed

\subsection{Large deviations principle with quadratic rate function: proof  of Theorems \ref{MDP1} and  \ref{MDP2}}\label{LDP pf}
In the circumstances of both Theorems \ref{MDP1} and  \ref{MDP2}, by the Gartner-Ellis theorem (see \cite{DemZet}) in order to prove the appropriate moderate deviations principle it is enough to show that for all real $t$ we have 
$$
\lim_{n\to\infty}\frac 1{s_n}\ln \bbE[e^{ta_n S_n^\om \tilde u/n}]
=\frac12 t^2\sig^2
$$
where the sequence $a_n$ is described in Theorems \ref{MDP1} and  \ref{MDP2}, $s_n=a_n^2/n$ and $\sig^2=\lim_{n\to\infty}\frac 1n\text{Var}_{\mu_\om}(S_n^\om u)$ (which does not depend on $\om$ and is assumed to be positive). Henceforth we will assume that $\mu_\om(u_\om)=0$, which means that we replace $u_\om$ by $\tilde u_\om=u_\om-\mu_\om(u_\om)$.

\subsubsection{Auxiliary estimates}
We first need the following result.

\begin{lemma}\label{BB}
Let $(\bar\la_\om(z), \bar h_\om^{(z)}, \bar\nu_\om^{(z)})$ be the normalized RPF triplets from Theorem \ref{Complex RPF}.
There is a constant $r>0$ so that $\bbP$-a.s.  for every complex number $z$ with  $|z|\leq r$ we have
$$
\|\bar\nu_\om^{(z)}\|\leq M_\om K_\om, \,\|\bar h_\om^{(z)}\|\leq \frac{2\sqrt 2 U_\om K_\om}{|\al_\om(z)|}, \,\,U_\om=6B_{\om,1}^2K_\om
$$
and 
$$
|\bar\la_\om(z)|\leq 3e^{\|u_\om\|_\infty}(1+2H_\om)\|\cL_\om\textbf{1}\|_{\infty}\leq \bar D_\om
$$
(where $K_\om, M_\om$ and $\bar D_\om$ are defined in Sections \ref{Aux1} and \ref{Aux2}).
\end{lemma}

\begin{proof}
Using the upper and lower bounds in Theorem \ref{RPF} together with \eqref{la bounds} we see that 
$$
\|\bar\nu_\om^{(z)}\|\leq M_\om \|h_\om\|\leq M_\om K_\om\,\,\text{ and }\,\, \|\bar h_\om^{(z)}\|\leq \frac{U_\om \cdot(2\sqrt 2K_\om)}{|\al_\om(z)|}
$$
where we used that $v(1/h)\leq v(h)(\inf h)^{-2}$ for every positive function $h$ (and so $\|1/h_\om\|\leq U_\om$).
To bound $\la_\om(z)$, 
notice that $\la_\om(z)=\nu_{\te\om}^{(z)}(\cL_\om^{(z)}\textbf{1})$ which yields that 
$$
|\la_\om(z)|\leq M_{\te\om}\|\cL_\om^{(z)}\textbf{1}\|\leq M_{\te\om}\|e^{zu_\om}\|\|\cL_\om\|.
$$
Firstly, let us bound the norm $\|\cL_\om\|$. Let $g$ be a H\"older continuous function. Then $\|\cL_\om g\|_{\infty}\leq \|\cL_\om \textbf{1}\|_\infty \|g\|_\infty$ (this part does not require continuity of $g$, only boundedness). Secondly, let us estimate the H\"older constant of $\cL_\om g$. In the setup of Section \ref{Maps1} set $c_\om=\gamma_\om^{-1}$, while in the setup of Section \ref{Maps2} set $c_\om=l_\om$. Then for every two points $x,x'\in\cE_{\te\om}$ %so that $\rho(x,x')\leq \xi$ (where we set $\xi=1=\text{diam}(\cE_\om)$ in the setup of Section \ref{Maps2}) 
we have
$$
|\cL_\om g(x)-\cL_\om g(x')|\leq \sum_{i}e^{\phi_\om(y_i)}|g(y_i)-g(y_i')|+
 \sum_{i}|e^{\phi_\om(y_i)}-e^{\phi_\om(y_i')}||g(y_i')|
$$
$$
\leq v(g)c_\om^\al \rho^\al(x,x')\|\cL_\om\textbf{1}\|_{\infty}+2H_\om \rho^\al(x,x')\|g\|_{\infty}\|\cL_\om \textbf{1}\|_{\infty}\leq(c_\om^\al+2H_\om)\|\cL_\om\textbf{1}\|_{\infty}\|g\|\rho^\al(x,x').
$$
In the second inequality we have also used that 
$$
|e^{\phi_\om(y_i)}-e^{\phi_\om(y_i')}|\leq(e^{\phi_\om(x_i)}+e^{\phi_\om(y_i)})H_\om\rho^\al(x,y)
$$
which is obtained using the mean value theorem and the definition of $H_\om$ (in either \eqref{phi cond} or \eqref{H def}).
Here $y_i=y_{i,\om}(x)$ and $y_i'=y_{i,\om}(x')$ are the inverse images of $x$ and $x'$ under $T_\om$, respectively.
Combining the above estimates we see that
\begin{equation}\label{See}
\|\cL_\om\|\leq (1+H_\om+c_\om^\al)\|\cL_\om \textbf{1}\|_{\infty}=\tilde D_\om.
\end{equation}
Finally, using also that $\la_\om\geq e^{-\|\phi_\om\|_\infty}$ we conclude that when $|z|\leq 1$ then
$$
|\bar\la_\om(z)|=\frac{|\la_\om(z)|}{\la_\om}\leq M_{\te\om}e^{\|\phi_\om\|_\infty}\|e^{zu_\om}\|\tilde D_\om\leq M_{\te\om}e^{\|\phi_\om\|_\infty}e^{\|u_\om\|_\infty}(1+v(u_\om))\tilde D_\om\leq \bar D_\om.
$$
\end{proof}

\begin{corollary}\label{Cor BB}
There exist  constants $r_1,C_1>0$ so that $\bbP$-a.s. if $|z|\leq r_1$ then 
$$
|\bar\la_\om(z)-1|\leq C_1|z|\bar D_\om
$$
and for every $r_2\leq r_1$ if $|z|\leq r_2/2$ then with $\beta_\om=\inf_{|z|\leq r_2}|\al_\om(z)|$ we have
$$
\|\bar h_\om^{(z)}-\textbf{1}\|\leq 2\sqrt 2 U_\om K_\om|z|\beta_\om^{-1}
$$
where $U_\om$ was defined in Lemma \ref{BB}.
\end{corollary}
\begin{proof}
Since $z\to\bar\la_\om(z)$ and $z\to\bar h_\om^{(z)}$ are analytic, the corollary follows from Lemma \ref{BB} together with the Cauchy integral formula.
\end{proof}
\subsubsection{MDP via inducing: proof of Theorem \ref{MDP1}}
Let $A$ be the set from the assumptions of Theorem \ref{MDP1}. Then there is a constant $Q$ so that for every $\om\in A$ we have $\max(M_\om, K_\om, U_\om)\leq Q$ (noting that $U_\om\leq B_\om$). Let $n_A$ be the first visiting time to $A$. Then, using the upper bounds on $|\al_\om(z)|$ from Theorem  \ref{Complex RPF} and the Cauchy integral formula, we see that there is a constant $r_0>0$ so that if $|z|\leq r_0$ then
 for every $\om\in A$ we have
$$
|\al_\om(z)-1|\leq 2\sqrt 2 Q^2|z|<\frac12
$$
and so 
$$
\beta_\om=\min_{|z|\leq r_0}|\al_\om(z)|\geq\frac12.
$$ 
Now, let $n$ be so that $\te^n\om\in A$. Then if $|z|\leq r_0$ we have $|\al_{\te^n\om}(z)|\geq \frac12$. On other hand, since $K_\om$ and $M_\om$ are in $L^p(\Om,\cF,\bbP)$ we have $\max(K_{\te^n\om}, M_{\te^n\om})=o(n^{1/p})$ (a.s.). Using also that $0<\rho(\om)<1$ we see that $K_{\te^n\om}M_{\te^n\om}\rho_{\om,n}$ decays to $0$ exponentially fast. In particular, for every $n$ large enough he have 
$$
\beta_\om\geq 2\sqrt 2K_{\te^n\om}M_{\te^n\om}\rho_{\om,n}.
$$
Hence, by applying \eqref{Exponential convergence CMPLX} with $\te^n\om$ instead of $\om$ 
we see that if $\te^n\om\in A$ and $n$ is large enough then
\begin{equation}\label{Exponential convergence new1}
\Big\|\frac{L_{\om}^{z,n}g}{\bar\la_{\om,n}(z)}-\bar\nu_{\om}^{(z)}(g)\bar h^{(z)}_{\te^n\om}\Big\|\leq C(Q)M_{\om}\|g\|\rho_{\om,n}
\end{equation}
for every H\"older continuous function $g$, where $C(Q)$ is a constant that depends on $Q$, but not on $\om$ or $n$.

Next, notice that under the Assumptions of Theorem \ref{MDP1} we have that $\bar D_\om$ from Lemma \ref{BB} belongs to $L^{2p}$. Hence $\bar D_{\te^j\om}=o(|j|^{2/p})$ and by Corollary \eqref{Cor BB} we have
$$
|\bar\la_{\te^j\om}(z)-1|\leq C|j|^{2/p}.
$$
Thus there are uniformly bounded analytic branches (vanishing at the origin) of $\ln\bar \la_{\te^j\om}(z)$ for $j\leq n$ on any domain of the form $|z|=o(n^{-2/p})$. Let us denote these branches by $\Pi_{\te^j\om}(z)$.

Now, when $\te^n\om\in A$ then by Corollary \ref{Cor BB} when $|z|$ is small enough for $\bbP$-a.a. $\om$ we have 
\begin{equation}\label{h bar est}
\|\bar h_{\te^n\om}^{(z)}-\textbf{1}\|\leq \frac12.
\end{equation}
and so
$$
\frac 12\leq|\mu_\om(\bar h_\om^{(z)})|\leq \frac32
$$
Therefore we can also develop uniformly bounded branches of $\ln \mu_\om(\bar h_\om^{(z)})$ around the complex origin which vanishes at the origin.

Next, by using \eqref{Exponential convergence new1} we see that for $n$  large enough, if $\te^n\om\in A$ and $|z|=O(n^{-2/p})$ then
\begin{equation}\label{Char}
\bbE[e^{zS_n^\om \tilde u}]=\mu_{\te^n\om}(L_\om^{z,n} \textbf{1})=\bar\la_{\om,n}(z)\left(\mu_{\te^n\om}(\bar h_{\te^n\om}^{(z)})+O(\rho_{\om,n}z)\right).
\end{equation}
Since $\rho_{\om,n}$ decays exponentially fast to $0$ and $|\mu_{\te^n\om}(\bar h_{\te^n\om}^{(z)})-1|\leq \frac12 |z|$ (by \eqref{h bar est}) by taking the logarithms of both sides and using anlyticity (and the Cauchy integral formula) we see that when $\te^n\om\in A$ and $|z|=O(n^{-2/p})$ and $n$ is large enough we have
$$
\ln \bbE[e^{zS_n^\om \tilde u}]=\Pi_{\om,n}(z)+O(|z|)+O(\del^n)
$$
where $\Pi_{\om,n}(z)=\sum_{j=0}^{n-1}\Pi_{\te^j\om}(z)$ and $\del=\del_\om\in (0,1)$, and we have used that $\ln(1+w)=O(w)$ when $|w|$ is small enough. By taking the derivatives at $z=0$ and using the Cauchy integral formula on domains of the form $|z|=O(n^{-2/p})$ we see that 
$$
0=\bbE[S_n^\om \tilde u]=\Pi_{\om,n}'(0)+O(n^{2/p})
$$
and 
$$
\sig_{\om,n}^2=\bbE[(S_n^\om \tilde u)^2]=\Pi_{\om,n}''(0)+O(n^{4/p}).
$$
Moreover, since $|\Pi_{\om,n}(z)|=O(n)$, by using the Cauchy integral formula to estimate the error term in the second order Taylor expansion of $\Pi_{\om,n}(z)$ around $z=0$ we see that when $|z|=O(n^{-2/p})$ then
$$
\Pi_{\om,n}(z)=z\Pi_{\om,n}'(0)+\frac12 z^2\Pi_{\om,n}''(0)+O(|z|^3)n^{1+6/p}
$$
and so 
$$
\ln \bbE[e^{zS_n^\om \tilde u}]=O(n^{2/p})z+\frac12 z^2\sig_{\om,n}^2+O(n^{4/p})z^2+
O(|z|^3)n^{1+6/p}+
O(|z|)+O(\del^n).
$$
Let us now fix some $t\in\bbR$ and take $z=t_n=itb_n/n$, where $b_n$ satisfies $b_n\gg n^{4/p}$ and $\frac{b_n}{n}=o(n^{-6/p})$. Then, since $p>4$, 
$$
\frac{\ln \bbE[e^{t(b_n/n)S_n^\om \tilde u}]}{b_n^2/n}=o(1)+\frac12 t^2(\sig_{\om,n}^2/n),\,\, \sigma=\lim_{n\to\infty}\frac1n \sig_{\om,n}^2.
$$
Thus, 
\begin{equation}\label{By}
\lim_{n\to\infty, \te^n\om\in A}\frac{1}{b_n^2/n}\ln \bbE[e^{t_n S_n^\om \tilde u}]=\frac12 t^2\sig^2.
\end{equation}

The next step will be to use \eqref{By} to derive a similar result without the restriction $\te^n\om \in A$. Let $(a_n)$ be a sequence with the properties described in Theorem \ref{MDP1}.  Let us take some  $n$ so that $\te^n\om\not\in A$, and let $m=m_n=m_n(\om)$ be the largest time $m\leq n$ so that $\te^{m}\om\in A$. Then
$$
\left|\ln \bbE[e^{ta_n S_n^\om \tilde u/n}]-\ln \bbE[e^{ta_n S_{m_n}^\om \tilde u/n}]\right|\leq
|ta_n/n|\cdot\left\|\sum_{j=m_n}^{n-1}\tilde u_{\te^j\om}\circ T_{\te^{m_n}\om}^{n-m_n}\right\|_{\infty}.
$$
Now, if we set 
$$
\tilde \Psi_\om=\sum_{j=0}^{n_A(\om)-1}\|\tilde u_{\te^j\om}\|_{\infty}
$$
then 
$$
\left\|\sum_{j=m_n}^{n-1}\tilde u_{\te^j\om}\circ T_{\te^m_n\om}^{n-m_n}\right\|_{\infty}\leq \tilde \Psi_{\te^{m_n(\om)}\om}.
$$
Observe now that with $c(\om)=\|\tilde u_\om\|_\infty$ we have
$$
\tilde \Psi_\om=\sum_{j=0}^\infty c(\te^j\om)\bbI(n_A(\om)>j) 
$$
and so by the H\"older inequality, if $q$ denotes the conjugate exponent of $p$ then
$$
\|\tilde \Psi_\om\|_{L^p(\bbP)}\leq \|c(\cdot)\|_{L^p(\bbP)}\sum_{j=0}^{\infty}(\bbP(n_A>j))^{1/q}.
$$
Arguing as in  Section \ref{Tails} we see that under each one of the conditions (M1'), (M2') or (M3') 
we have $\sum_{j=0}^{\infty}(\bbP(n_A>j))^{1/q}<\infty$. We thus see that $\|\tilde \Psi_\om\|_{L^p(\bbP)}<\infty$ and so $\tilde \Psi_{\te^j\om}=o(j^{1/p})$ almost surely.
Thus, since $p>2$ and $a_n/\sqrt n\to\infty$  we see that 
$$
\frac{|ta_n/n|\left\|\sum_{j=m_n}^{n-1}\tilde u_{\te^j\om}\circ T_{\te^{m_n}\om}^{n-m_n}\right\|_{\infty}}{s_n}=O(a_n^{-1}n^{1/p})\to 0,\,\,s_n=a_n^2/n.
$$
Finally, since $m_n=n(1+o(1))$ by the assumptions on the sequence $(a_n)$ in Theorem \ref{MDP1} we have $a_n=a_{m_n}(1+o(1))$ and $s_n=a_n^2/n=s_{m_n}(1+o(1))$. Therefore   by \eqref{By} we have
$$
\lim_{n\to\infty}\frac 1{s_n}\ln \bbE[e^{ta_n S_n^\om \tilde u/n}]=
\lim_{n\to\infty, \te^n\in A}\frac 1{s_n}\ln \bbE[e^{t_n S_n^\om \tilde u}]=\frac12 t^2\sig^2
$$
and the proof of Theorem \ref{MDP1} is complete. 
\qed
%need: assume that  $m_n=n(1+o(1))$ we have $a_n=a_{m_n}(1+o(1))$ and $s_n=a_n^2/n=s_{m_n}(1+o(1))$ and that $p>4$ and the right mixing assumptions  in rgards to A and that $a_n\gg n^{1/2}, n^{4/p}$ and $a_n\ll n, n^{2-6/p}$....Then get LDP...note that we also need to work under the additioal conditions of the complex RPF theorem
\subsubsection{A direct approach to the MDP: proof of Theorem \ref{MDP2}}
Recall that when $|z|\leq r_0$ (for some constant $r_0$) then $|\al_\om(z)|\leq 2\sqrt 2 K_\om M_\om$.
Now, using the Cauchy integral formula, when $|z|\leq r_0/2$ we have
$$
|\al_\om(z)-1|\leq CK_\om M_\om|z|
$$
where $C=C(r_0)$ is some constant.
Since $K_\om$ and $M_\om$ are in $L^p(\Om,\cF,\bbP)$ we have $K_{\te^j\om} M_{\te^j\om}=o(j^{2/p})$ and so when $|z|=O(n^{-2/p})$, then for every $n$ large enough
\begin{equation}\label{al est0}
|\al_{\te^n\om}(z)-1|\leq \del_n\to 0 \,\text{ as }\,n\to\infty.
\end{equation}
 Thus, for such $z$'s when  $n$ is large enough so that $n^{2/p}\rho_{\om,n}<1/4$ we can apply \eqref{Exponential convergence CMPLX} with $\te^n\om$ instead of $\om$ and $|z|=O(n^{-2/p})$ and get that 
\begin{equation}\label{Exponential convergence new2}
\Big\|\frac{L_{\om}^{z,n}g}{\bar\la_{\om,n}(z)}-\bar\nu_{\om}^{(z)}(g)\bar h^{(z)}_{\te^n\om}\Big\|\leq CM_{\om}\|g\|\del_\om^n
\end{equation}
for some $\del_\om\in(0,1)$, where we have used that $K_{\te^n\om}, M_{\te^n\om}$ and $U_{\te^n\om}$ grow at most polynomially fast  and $\rho_{\om,n}$ decays to $0$ exponentially fast in $n$. 
 
 Next, by applying the Cauchy integral formula on a domain of the form $\{|z|=O(n^{-2/p})\}$ and using Lemma \ref{BB} to bound the derivative of $z\to\bar h_{\te^n\om}^{(z)}$ on such domains (taking into account \eqref{al est0} and that $U_{\te^n\om} K_{\te^n\om}=o(n^{2/p})$) we see that  when $|z|=O(n^{-2/p})$ then 
\begin{equation}\label{hhh0}
\|\bar h_{\te^n\om}^{(z)}-\textbf{1}\|=|z|O(n^{4/p}).
\end{equation}
Thus,  we can develop a branch of $\ln \mu_{\te^n\om}(h_{\te^n\om}^{(z)})$ on a domain of the form $|z|=O(n^{-2/p})$ so that 
\begin{equation}\label{hhh}
\ln \mu_{\te^n\om}(h_{\te^n\om}^{(z)})=1+|z|O(n^{4/p}).
\end{equation}

Similarly, by the Assumptions of Theorem \ref{MDP2} we have that $\bar D_\om$ defined in Lemma \ref{BB} belongs to $L^{2p}(\Om,\cF,\bbP)$. Hence $\bar D_{\te^j\om}=o(|j|^{2/p})$ and by using  Corollary \eqref{Cor BB} we see that
$$
|\bar\la_{\te^j\om}(z)-1|=o(j^{2/p})|z|.
$$
Thus there are uniformly bounded branches (vanishing at the origin) of $\bar \la_{\te^j\om}(z)$ for $j\leq n$ on any domain of the form  $|z|=O(n^{-2/p})$. Let us denote these branches by $\Pi_{\te^j\om}(z)$.

Next, by \eqref{Exponential convergence new2} we have 
$$
\bbE[e^{zS_n^\om \tilde u}]=\mu_{\te^n\om}(L_\om^n \textbf{1})=\bar\la_{\om,n}(z)\left(\mu_{\te^n\om}(\bar h_{\te^n\om}^{(z)})+O(\del_\om^n z)\right).
$$
Using the above estimates, by taking the logarithm of both sides and using anlyticity (and the Cauchy integral formula) we see that when $|z|=O(n^{2/p})$ and $n$ is large enough then
$$
\ln \bbE[e^{zS_n^\om \tilde u}]=\Pi_{\om,n}(z)+O(|z|n^{4/p})+O(\tilde \del_\om^n)
$$
where $\Pi_{\om,n}(z)=\sum_{j=0}^{n-1}\Pi_{\te^j\om}(z)$, $\tilde \del_\om\in(0,1)$ and we have used that $\ln(1+w)=O(w)$ when $|w|$ is small enough. By taking the derivatives at $z=0$ and using the Cauchy integral formula we see that 
\begin{equation}\label{Pi 1}
0=\bbE[S_n^\om \tilde u]=\Pi_{\om,n}'(0)+O(n^{6/p})
\end{equation}
and 
\begin{equation}\label{Pi 2}
\sig_{\om,n}^2=\bbE[(S_n^\om \tilde u)^2]=\Pi_{\om,n}''(0)+O(n^{8/p}).
\end{equation}
Moreover, since $|\Pi_{\om,n}(z)|=O(n)$, by using the Cauchy integral formula to estimate the error term in the second order Taylor expansion of $\Pi_{\om,n}(z)$ around $z=0$ we see that when $|z|=O(n^{-2/p})$ then
\begin{equation}\label{Tay2}
\Pi_{\om,n}(z)=z\Pi_{\om,n}'(0)+\frac12 z^2\Pi_{\om,n}''(0)+|z|^3O(n^{1+8/p})
\end{equation}
and so 
$$
\ln \bbE[e^{zS_n^\om \tilde u}]=O(n^{6/p})z+\frac12 z^2\sig_{\om,n}^2+O(n^{8/p})z^2+
|z|^3O(n^{1+8/p})+
O(|z|n^{4/p})+O(\tilde\del_\om^n).
$$
Finally, let us fix some $t\in\bbR$ and take $z=t_n=ta_n/n$. Then, since $p>8$, $a_n\gg n^{6/p}$ and $\frac{a_n}{n}=o(n^{-8/p})$ we have
$$
\frac{\ln \bbE[e^{t(a_n/n)S_n^\om \tilde u}]}{a_n^2/n}=o(1)+\frac12 t^2(\sig_{\om,n}^2/n).
$$
Thus, 
$$
\lim_{n\to\infty}\frac{1}{a_n^2/n}\ln \bbE[e^{t(a_n/n) S_n^\om \tilde u}]=\frac12 t^2\sig^2
$$
and the proof of Theorem \ref{MDP2} is complete.
\qed

\subsection{Berry-Esseen type estimates: proof of Theorem \ref{BE}}\label{BE pf}
In this section we will prove Theorem \ref{BE} (ii), and the proof of Theorem \ref{BE} (i) is similar (we will provide a few details after completing the proof of the second part).
\begin{lemma}\label{ll}%p>8
Let $\Pi_{\om,n}$ be as defined in the proof of Theorem \ref{MDP2}.

(i) We have 
$$
\left|\frac{\Pi_{\om,n}''(0)}{\sig_{\om,n}^2}-1\right|=O(n^{8/p-1})=o(1).
$$

(ii) On any domain of the form  $|t/\sig_{\om,n}|=O(n^{-2/p})$ we have
$$
|\la_{\om,n}(it/\sig_{\om,n})|=|e^{\Pi_{\om,n}(it/\sig_{\om,n})}|\leq e^{-ct^2/2}
$$
where $c\in(0,\frac12)$ is some constant.
\end{lemma}
\begin{proof}
The first part follows from \eqref{Pi 2}, and the second part follows from the first and \eqref{Tay2} together with the fact that $\Pi_{\om,n}'(0)\in\bbR$ (since $\Pi_{\om,n}(t)\in\bbR$ when $t$ is real) and recalling that $\sig_{\om,n}^2$ grows linearly fast in $n$.
\end{proof}

\begin{proof}[Proof of Theorem \ref{BE} (ii)]
Suppose $\mu_\om(u_\om)=0$.
Let $d_n=n^{\frac12-2/p}$. Then by the Esseen inequality (see \cite{IL} or a generalized version \cite[\S XVI.3]{Feller}) there is an absolute constant $C$ so that
\begin{equation}\label{Ess}
\sup_{t\in\bbR}\left|\mu_\om(S_n^\om u\leq t\sig_{\om,n})-\Phi(t)\right|\leq \frac{C}{d_n}+\int_{-d_n}^{d_n}\frac{\left|\mu_\om(e^{it S_n^\om u/\sig_{\om,n}})-e^{-t^2/2}\right|}{|t|}dt.
\end{equation}
In order to bound the integral on the right hand side, first  by \eqref{Tay2}, \eqref{Pi 1}, \eqref{Pi 2} and Lemma \ref{ll} (i), for every $t\in[-d_n,d_n]$ we have 
$$
\Pi_{\om,n}(it/\sig_{\om,n})=-t^2/2+O(|t|n^{6/p-1/2})+O(t^2n^{8/p-1}) +O(n^{8/p-1/2}|t|^3)
$$
where we have used that $\sig_{\om,n}^2$ grows linearly fast in $n$, which, in particular, insures that $z=it/\sig_{\om,n}=O(n^{-2/p})$. 
Using also Lemma \ref{ll} (ii) and the mean value theorem we get that 
$$
\left|e^{\Pi_{\om,n}(it/\sig_{\om,n})}-e^{-t^2/2}\right|\leq e^{-ct^2/2}\left(O(|t|n^{6/p-1/2})+O(t^2n^{8/p-1}) +O(n^{8/p-1/2}|t|^3\right).
$$
Using now \eqref{Char}, \eqref{hhh0} and Lemma \ref{ll} (ii)  we see that 
\begin{equation}\label{Char dif}
\left|\mu_\om(e^{it S_n^\om u/\sig_{\om,n}})-e^{-t^2/2}\right|\leq C_\om|t|e^{-ct^2}\left(n^{6/p-1/2}+|t|n^{8/p-1}+t^2 n^{8/p-1/2}+n^{4/p-1/2}\right)
\end{equation}
for some constant $C_\om$ which depends on $\om$ but not on $t$ or $n$. The proof of Theorem \ref{BE} (ii) is completed now by combining \eqref{Ess} with \eqref{Char dif}.
\end{proof}
The proof of Theorem \ref{BE} (i) proceeds similarly for $n$'s so that $\te^n\om\in A$, and in order to pass to general indexes $n$ we use that $\tilde \Psi_{\te^j\om}=o(n^{1/p})$ together with \cite[Lemma 3.3]{HK BE} (applied with $a=\infty$).

\subsection{A moderate local limit theorem: proof of Theorem \ref{LLT}}\label{LLT pf}
As in the proof of Theorem \ref{BE}, let us assume that $\mu_\om(u_\om)=0$. 
By using a density argument (see \cite[Section VI.4]{HH}) it is enough to obtain \eqref{llt} for a function $g\in L^1(\bbR)$ whose Fourier transform has a compact support. Note that such a function $g$ satisfies the inversion formula.
 Let $g$ be a  function  with these properties and let $L>0$  be so that $\hat g(x)=0$ if $|x|>L$. Then, by the inversion formula for $g$, for all $y\in\bbR$ we have 
$$
g(y)=\int_{-\infty}^\infty \hat g(x)e^{iyx}dx=\int_{-L}^{L}\hat g(x)e^{iyx}dx.
$$ 
Taking some $v\in\bbR$, setting $y=S_n^\om u/a_n-v$ and then integrating with respect to $\mu_\om$ we see that
$$
\bbE_{\mu_\om}[g(S_n^\om u/a_n-v)]=\bbE_{\mu_\om}\left[\int_{-L}^{L}\hat g(x)e^{ixS_n^\om u/a_n}e^{-ivx}dx\right]=\int_{-L}^{L}\hat g(x)e^{-ivx}\mu_\om(e^{it S_n^\om u/a_n})dx=
$$
$$
\frac{a_n}{\sig_{\om,n}}\int_{-L\sig_{\om,n}/a_n}^{L\sig_{\om,n}/a_n}\hat g(a_n t/\sig_{\om,n})e^{-iv a_nt/\sig_{\om,n}}\bbE[e^{it S_n^\om u/\sig_{\om,n}}]dt
$$
where in the last equality we used the change of variables $x=\frac{a_n}{\sig_{\om,n}}t$. Here $(a_n)$ is the sequence specified in Theorem \ref{LLT}.
Now, since $a_n n^{-2/p}\to\infty$, the estimate \eqref{Char dif} is valid on the domain $\{|t|\leq L\sig_{\om,n}/a_n\}$. Therefore, uniformly in $v\in\bbR$, we have
$$
\frac{\sig_{\om,n}}{a_n}\bbE_{\mu_\om}[g(S_n^\om u/a_n-v)]-\int_{-L\sig_{\om,n}/a_n}^{L\sig_{\om,n}/a_n}\hat g(a_n t/\sig_{\om,n})e^{-iv a_nt/\sig_{\om,n}}e^{-t^2/2}dt=o(1).
$$
Next, set $\ka_n=\ka_{\om,n}=\sig_{\om,n}/a_n$. Then, in order to complete the proof of Theorem \ref{LLT} we need to show that, uniformly in $v\in\bbR$, we have
$$
\int_{-L\sig_{\om,n}/a_n}^{L\sig_{\om,n}/a_n}\hat g(a_n t/\sig_{\om,n})e^{-iv a_nt/\sig_{\om,n}}e^{-t^2/2}dt-\frac{1}{\sqrt{2\pi}}e^{-\frac{v^2}{2\ka_n^2}}=o(1).
$$
To prove that let us take an arbitrary small $\ve>0$ and  fix $T$ large enough so that
\begin{equation}\label{g}
\|g\|_{L^1(\bbR)}\int_{|t|>T}e^{-t^2/2}dt<\ve/3.
\end{equation}
Then, using that $\sup|\hat g|\leq \|g\|_{L^1(\bbR)}$ we see that for every $n$ large enough and all $v\in\bbR$ we have
$$
\left|\int_{-L\sig_{\om,n}/a_n}^{L\sig_{\om,n}/a_n}\hat g(a_n t/\sig_{\om,n})e^{-iv a_nt/\sig_{\om,n}}e^{-t^2/2}dt
-\int_{-T}^{T}\hat g(a_n t/\sig_{\om,n})e^{-iv a_nt/\sig_{\om,n}}e^{-t^2/2}dt\right|<\ve/3
$$
where we have used that $\sig_{\om,n}/a_n\to\infty$. Next, since $\lim_{x\to0}\hat g(x)=\hat g(0)=\int g(y)dy$  we see that for every $n$ large enough we have
$$
\sup_{v\in\bbR}\left|\int_{-T}^{T}\hat g(a_n t/\sig_{\om,n})e^{-iv a_nt/\sig_{\om,n}}e^{-t^2/2}dt-
\int_{-T}^{T}\hat g(0)e^{-iv a_nt/\sig_{\om,n}}e^{-t^2/2}dt\right|<\ve/3.
$$
Now, using again \eqref{g} we see that
$$
\sup_{v\in\bbR}\left|\int_{-T}^{T}\hat g(0)e^{-iv a_nt/\sig_{\om,n}}e^{-t^2/2}dt-\int_{-\infty}^{\infty}\hat g(0)e^{-iv a_nt/\sig_{\om,n}}e^{-t^2/2}dt\right|<\ve/3.
$$
We conclude from the above estimates that, for every $n$ large enough uniformly in $v$ we have 
$$
\left|\frac{\sig_{\om,n}}{a_n}\bbE_{\mu_\om}[g(S_n^\om u/a_n-v)]-\int_{-\infty}^{\infty}\hat g(0)e^{-iv a_nt/\sig_{\om,n}}e^{-t^2/2}dt\right|<\ve.
$$
Finally, by the inversion formula
$$
\int_{-\infty}^{\infty}\hat g(0)e^{-iv a_nt/\sig_{\om,n}}e^{-t^2/2}dt=\hat g(0)\frac{1}{\sqrt{2\pi}}e^{-\frac{v^2}{2\ka_n^2}}
$$
and the proof of Theorem \ref{LLT} is complete.

\section{Proof of the real RPF theorem (Theorem \ref{RPF})}\label{RPF pf}
\subsection{Effective rates for properly expanding maps: proof of Theorem \ref{RPF} (i)-(iii) in the setup of Section \ref{Maps1}}
For the sake of completeness, in this section we will also consider the setup described in Remark \ref{Rem SFT}, where for the sake of simplicity we focus on the case $n_0=1$   (to consider the case $n_0>1$ we essentially need to replace $T_\om$ with $T_\om^{n_0}$ and $\phi_\om$ with $S_{n_0}^\om\phi$). The setup from Section \ref{Maps1} will be referred to as ``the case $\xi=1$" (as we can pair the inverse images of any two points), while the setup from Remark \ref{Rem SFT} will be  referred to as ``the case $\xi<1$".
\subsubsection{The cones}
For each $a>0$ let us consider the real Birkhoff cone
$$
\cC_{\om,a}=\{g\in \cH_{\om,\al}: g\geq0,\, g(x)\leq e^{a \rho(x,x')^\al}g(x')\,\,\forall\,\,x,x'\in\cE_\om\,\text{ with  }\rho(x,x')\leq \xi\}.
$$
Set also $\cC_{\om}=\cC_{\om,\gamma_\om^\al}$.
\begin{lemma}\label{Cont1}
We have
$$
\cL_\om\cC_{\om}\subset\cC_{\te\om, H_\om+1}\subset\cC_{\te\om}.
$$ 
\end{lemma}
\begin{proof}
First, by \eqref{H cond} we have $H_\om+1\leq\gamma_{\te\om}^\al$ and so the second inclusion holds true.
To prove the first inclusion, 
let $g\in \cC_{\om}$ and let $x,x'\in\cE_{\te\om}$ be so that $\rho(x,x')\leq\xi$. Then, with $y_i=y_{i,\om}(x)$ and $y_i'=y_{i,\om}(x')$ as in \eqref{Pair1.0}, we have 
$$
\cL_\om g(x)=\sum_{i}e^{\phi_\om(y_i)}g(y_i)\leq \sum_{i}e^{\phi_\om(y_i')+\rho_\om^\al(x,x')H_\om}e^{\gamma_\om^\al\rho_\om^\al(y_i,y_i')}g(y_i')
$$
$$
\leq e^{(H_\om+\gamma_\om^\al\gamma_\om^{-\al})\rho_{\te\om}^{\al}(x,x')}\sum_{i}e^{\phi_\om(y_i')}g(y_i')=e^{(H_\om+1)\rho_{\te\om}(x,x')}\cL_\om g(x').
$$
\end{proof}
Next,
\begin{lemma}\label{BalLemma}
For all $g\in\cC_{\om}$ and every $x,x'\in\cE_{\te\om}$ we have
\begin{equation}\label{Bal}
\cL_\om g(x)\leq B_{\om,0}\cL_\om g(x')
\end{equation}
where when $\xi<\text{diam}(\cE_\om)=1$ 
$$
B_{\om,0}=e^{H_\om\xi^\al+\gamma_\om^\al\xi^\al} \deg(T_\om).
$$
while when $\xi=1$ we have 
$$
B_{\om,0}=e^{\gamma_{\te\om}^{\al}}.
$$
\end{lemma}

\begin{proof}
Suppose first that $\xi<1$. Then
$$
\cL_\om g(x)\leq \deg(T_\om)\max_{y\in T_\om^{-1}\{x\}}e^{\phi_\om(y)}g(y)=e^{\phi_\om(y_0)}g(y_0)
$$
for some $y_0$. On the other hand, 
let $y'\in T_{\om}^{-1}\{x'\}$ be so that $\rho(y_0,y')\leq \xi$ (existence of such $y'$ follows from our assumptions on the map $T_\om$). Then, since $g\in\cC_{\om}$, we have
$$
e^{\phi_\om(y_0)}g(y_0)\leq e^{H_\om+\gamma_\om^\al\xi^\al}e^{\phi_\om(y')}g(y')
$$
where we have also used \eqref{phi cond}.
On the other hand, 
$$
e^{\phi_\om(y')}g(y')\leq \cL_\om g(x')
$$
which together with the previous estimates yields the desired result in the case $\xi<1$.

When $\xi=1$ then $\cL_\om g\in\cC_{\te\om}$ and so (since $\xi=1$), 
$$
\cL_\om g(x)\leq e^{\gamma_\om^\al}\cL g(x')
$$
for all $x,x'$.
 \end{proof}

 \begin{corollary}\label{Cor diam}
The 
 the projective diameter of $\cL_{\om}\cC_{\om}$ inside $\cC_{\te\om}$ does not exceed $D(\om)$ (which was defined in Section \ref{Aux1}).
 \end{corollary}
 \begin{proof}
The statement follows from Lemma \ref{Cont1} and \ref{BalLemma}, and it appears in various forms in several places, and we refer to  \cite[Lemma 5.7.1]{HK} or \cite{Kifer thermo}.
 \end{proof}
 
\subsubsection{Reconstruction of $\nu_\om$ using dual cones}\label{nu sec 1}
Let $\cC_{\om}=\cC_{\om,\gamma_\om^\al}$ and let $\cC_\om^*$ be the dual cone which is given by 
$$
\cC_\om^*=\left\{\nu\in\cH_\om^*:\nu(g)\geq 0,\,\forall g\in\cC_\om\right\}.
$$
Let $\cL_\om^*:\cH_{\te\om}^*\to\cH_{\om}^*$ be the dual operator. Then by \cite[Lemma A.2.6]{HK} 
the projective diameter of $\cL_\om^* \cC_{\te\om}^*$ inside $\cC_\om^*$ equals the the projective  $\cL_{\om}\cC_{\om,}$ inside $\cC_{\te\om}$ (which by Corollary \ref{Cor diam} does not exceed $D(\om)$). Note that \cite[Lemma A.2.6]{HK} is technically about complex cones, but the arguments needed in the case of real cones are essentially the same\footnote{Noticing also that the closure of the cone $\tilde\cC_\om=\left\{\nu\in\cH_\om^*:\nu(g)> 0,\,\forall g\in\cC_\om\setminus\{0\}\right\}$ coincides with $\cC_\om^*$ 
(because there is a linear functional which is strictly positive on $\cC_\om\setminus\{0\}$).}.

We need now the following result.

\begin{lemma}\label{Dual aper}
For every $\mu\in \cC_\om^*$ and all $h\in\cH$ we have
$$
|\mu(h)|\leq 2\|h\|\mu(\textbf{1}).
$$
\end{lemma}

\begin{proof}
First, let us show that a closed ball of radius $1/2$ around $\textbf{1}$ is contained in $\cC_\om$.
Indeed, let $h=1+f$ where $\|f\|\leq \frac12$. Then $h$ belongs to $\cC_\om$ if and only if for all $x$ and $x'$ so that $\rho(x,x')\leq \xi$ we have 
$$
h(x)\leq e^{\gamma_\om^\al\rho(x,x')^\al}h(x'),\,\,\,\rho(x,x')^\al=\left(\rho(x,x')\right)^\al
$$ 
which can also be written as 
\begin{equation}\label{Obt}
f(x)-f(x')\leq (e^{\gamma_\om^\al\rho(x,x')^\al}-1)(1+f(x')).
\end{equation}
Now, since $e^t-1\geq t$ for all $t\geq0$ and $1+f(x')\geq 1-\|f\|_\infty$ we have
$$
 (e^{\gamma_\om^\al\rho(x,x')^\al}-1)(1+f(x'))\geq \gamma_\om^\al\rho(x,x')^\al (1-\|f\|_\infty)\geq 
 \frac12\rho(x,x')^\al
$$
where we have used that $\gamma_\om^\al\geq1$ and $\|f\|_{\infty}\leq\|f\|\leq \frac12$.
On the other hand, since $v(f)\leq \|f\|\leq\frac12$ we have
$$
f(x)-f(x')\leq \rho(x,x')^\al v(f)\leq \frac12\rho(x,x')^\al.
$$
Combing the last two estimates we obtain \eqref{Obt}.

Next, let $\mu\in\cC_\om^*$, and let $h\in\cH_\om$ be so that $\|h\|\leq 1$. Then $1\pm\frac12h\in\cC_\om$ and so  
$$
\mu(1\pm\frac12h)\geq 0,
$$
that is 
$$
|\mu(h)|\leq 2\mu(\textbf{1}).
$$
\end{proof}

Next, let $\rho(\om)=\tanh (D(\om)/4)$ (as was defined in Section \ref{Aux1}). Let $\mu\in\cC_{\te^n\om}^*$ and $\nu\in\cC_{\te^m\om}^*$ for some $m\geq n$. Then by  the projective contraction properties of linear maps (see \cite{Birk} and \cite[Theorem 1.1]{Liv}) the projective distance between $(\cL_\om^n)^*\mu$ and $(\cL_\om^m)^*\nu=(\cL_\om^n)^*(\cL_{\te^n\om}^{m-n})^*\nu$ does not exceed $\rho_{\om,n}=\prod_{j=0}^{n-1}\rho(\te^j\om)$. Hence by\footnote{These results are formulated for complex cones, but a real complex cone is embedded in its canonical complexification (together with the corresponding projective metrics)} \cite[Theorem A.2.3]{HK} and \cite[Lemma 5.2]{Rugh},
$$
\left\|\frac{(\cL_\om^n)^*\mu}{\mu(\cL_\om^n\textbf{1})}-\frac{(\cL_\om^m)^*\nu}{\nu(\cL_\om^m\textbf{1})}\right\|\leq \sqrt 2\rho_{\om,n}.
$$
Notice that $\rho_{\om,n}$ converges exponentially fast to $0$ for $\bbP$-a.a. $\om$ (indeed $\rho(\cdot)<1$ and $\te$ is ergodic).
Thus, for any sequence $\mu_n$ so that $\mu\in\cC_{\te^n\om}^*$ the limit
$$
\nu_\om=\lim_{n\to\infty}\frac{(\cL_\om^n)^*\mu}{\mu(\cL_\om^n\textbf{1})}
$$
exists, belongs to $\cC_\om^*$ and it does not depend on the choice of the sequence (hence $\om\to\nu_\om$ is measurable). Moreover, by fixing $n$ and letting $m\to\infty$ we have 
\begin{equation}\label{AAb}
\left\|\frac{(\cL_\om^n)^*\mu}{\mu(\cL_\om^n\textbf{1})}-\nu_\om\right\|\leq \sqrt 2\rho_{\om,n}.
\end{equation}
Note that $\nu_\om(\textbf{1})=1$.
Furthermpre, by  plugging in $(\cL_{\te\om}^n)^*\mu$ inside $\cL_\om^*$ and using \eqref{AAb} with $\te\om$ instead of $\om$  we see that there is a number $\la_\om$ so that $\cL_\om^*\nu_{\te\om}=\la_\om\nu_\om$. Plugging in $g=\textbf{1}$ we also see that $\la_\om=\nu_{\te\om}(\cL_\om \textbf{1})$. Finally, since $\cH_\om$ is dense in $C(\cE_\om)$ and $\nu_\om$ is positive we get that $\nu_\om$ can be extended to a probability measure on $\cE_\om$.

\subsubsection{Reconstruction of $h_\om$ with effective rates}\label{h sec}
We first need the following result.

\begin{lemma}\label{Lemma 5.5}
In the case $\xi<1$ for every $i$ we have
$$
\nu_\om(B(x_i,\xi))\geq \frac{e^{-2\|\phi_\om\|_\infty}}{\deg (T_\om)}:=b_\om
$$
where the points $x_{i}=x_{i,\om}$ are described in Section \ref{Maps1}.
\end{lemma}
\begin{proof}
First, recall our (covering) assumption that  for all $i$ we have $T_\om(B_\om(x_i,\xi))=\cE_{\te\om}$. Hence, for every $x\in\cE_{\te\om}$ we have
$$
\left(\cL_\om(\textbf{1}_{B_\om(x_i,\xi)})\right)(x)\geq e^{-\|\phi_\om\|_\infty}.
$$
Next, since  $\la_\om=\nu_{\te\om}(\cL_\om \textbf{1})$ we have
$$
\la_\om\leq \|\cL_{\om}\textbf{1}\|_\infty\leq \deg(T_\om)e^{\|\phi_\om\|_\infty}
$$
and so 
$$
\nu_\om(B(x_i,\xi))=\nu_{\te\om}(\la_\om^{-1}\cL_\om(\textbf{1}_{B_\om(x_i,\xi)}))\geq \la_\om^{-1}e^{-\|\phi_\om\|_{\infty}}\geq \frac{e^{-2\|\phi_\om\|_\infty}}{\deg (T_\om)}.
$$
\end{proof}

 \begin{lemma}\label{Aper}
 For every $g\in\cC_{\om}$ we have 
 $$
 \|g\|\leq K_\om\nu_\om(g)
 $$
 where when $\xi<1$,
 $$
 K_\om=e^{2\|\phi_\om\|_\infty+2\xi^\al\gamma_\om^\al}\deg(T_\om)(1+\gamma_\om^\al)
 $$
 while when $\xi=1$ we have $K_\om=(1+\gamma_\om^\al)e^{\gamma_\om^\al}$.
 
 \end{lemma}
 
 \begin{proof}
 Fix some $g\in\cC_{\om}$.
Suppose first that $\xi<1$. Let $x\in\cE_\om$ and  let $i$ be so that $\rho(x,x_i)\leq \xi$, $x_i=x_{i,\om}$.
Then since $g\in\cC_\om$ we have
 $$
 g(x)\leq e^{\xi^\al \gamma_\om^\al}g(x_i)\leq e^{2\xi^\al \gamma_\om^\al}\inf\{g(y): d(y,x_i)\leq \xi\}\leq \frac{e^{2\xi^\al \gamma_\om^\al}}{\nu_\om(B(x_i,\xi))}\int_{B_\om(x_i,\xi)}g(z)dz$$$$\leq e^{2\xi^\al \gamma_\om^\al}b_\om^{-1}\nu_\om(g).
 $$
 where in the last inequality we have used Lemma \ref{Lemma 5.5}.
 By taking the supremum over all possible choices of $x$ we see that
 $$
\sup g= \|g\|_\infty\leq  e^{2\xi^\al \gamma_\om^\al}b_\om^{-1}\nu_\om(g).
 $$
 When $\xi=1$ we have 
 $$
\sup g\leq e^{\gamma_\om^\al}\inf g\leq e^{\gamma_\om^\al}\nu_\om(g). 
 $$
 
 Finally, in both cases, if $g(x)>g(x')$ and $\rho(x,x')\leq\xi$ then 
 $$
 g(x)-g(x')=g(x)(1-g(x)/g(x'))\leq \|g\|_\infty(1-e^{-\gamma_\om^\al\rho^\al(x,x')})\leq 
  \|g\|_\infty \gamma_\om^\al\rho^\al(x,x').
 $$
 Thus,
 $$
 v(g)=v_{\al,\xi}(g)\leq \|g\|_\infty\gamma_\om^\al
 $$
 and so 
 $$
 \|g\|=v(g)+\|g\|_\infty\leq \left(1+\gamma_\om^\al\right)\|g\|_\infty.
 $$
 Now the lemma follows from the above upper bounds on $\|g\|_\infty$.
 \end{proof}
 
 Next, arguing as in \cite[Theorem 5.3.1 (iii)]{HK} we can prove the following result.
 
 \begin{lemma}\label{generating}
 For every $f\in\cH_\om$ there are $f_1,f_2\in\cC_{\om}$ and negative constants $c_1,c_2$
   so that $f=f_1-c_1-(f_2-c_2)$ and 
 $$
\|f_1\|+\|f_2\|+|c_1|+|c_2|\leq r_\om\|f\|
 $$
 where $r_\om=4\left(1+\frac{2}{\gamma_\om^\al}\right)\leq 8$.
 \end{lemma}
 
 Next, by applying \cite[Theorem A.2.3]{HK} and taking into account Corollary \ref{Cor diam} and Lemma \ref{Aper} we see that for every $g\in\cC_{\te^{-n}\om}$ and $f\in\cC_{\te^{-m}\om}$ with $m\geq n$ we have
$$
\left\|\frac{\cL_{\te^{-n}\om}^n g}{\nu_{\te^{-n}\om}(\cL_{\te^{-n}\om}^n g)}-\frac{\cL_{\te^{-m}\om}^n f}{\nu_{\te^{-m}\om}(\cL_{\te^{-m}\om}^nf)}\right\|\leq \frac 12K_\om\rho_{\te^{-n}\om,n}.
$$
Notice that 
$$
\nu_{\te^{-n}\om}(\cL_{\te^{-n}\om}^n g)=\la_{\te^{-n}\om,n}\nu_{\te^{-n}\om}(g), \,\nu_{\te^{-m}\om}(\cL_{\te^{-m}\om}^m f)=\la_{\te^{-m}\om,m}\nu_{\te^{-m}\om}(f)
$$
where $\la_{\om,n}=\prod_{j=0}^{n-1}\la_{\te^j\om}$.
We thus see that for every sequence $(g_n)$ so that $g_n\in\cC_{\te^{-n}\om}$ the limit 
\begin{equation}\label{h lim}
h_\om=\lim_{n\to\infty}\frac{\cL_{\te^{-n}\om}^n g_n}{\nu_{\te^{-n}\om}(g_n)\la_{\te^{-n}\om,n}}
\end{equation}
exists, it does not depend on the sequence $(g_n)$,  it belongs to  $\cC_\om$ and $\nu_\om(h_\om)=1$ (and so by Lemma \ref{Aper} we have $\|h_\om\|\leq K_\om$). Moreover,
by taking $g_n=\textbf{1}$ for every $n$ and applying $\cL_\om$ to both sides of \eqref{h lim} we see that  $\cL_\om h_\om=\la_\om h_{\te\om}$. Furthermore, by fixing $n$, taking $f=f_m\in\cC_{\te^{-m}\om}$ and letting $m\to\infty$ we see that for every $g\in\cC_{\te^{-n}\om}$ we have
\begin{equation}\label{uu}
\left\|\frac{\cL_{\te^{-n}\om}^n g}{\la_{\te^{-n}\om,n}}-\nu_{\te^{-n}\om}(g)h_\om\right\|\leq \frac12 K_\om\nu_{\te^{-n}\om}(g)\rho_{\te^{-n}\om,n}.
\end{equation}
In addition, since $h_\om\in \cC_\om$  by \eqref{Bal}, for all $x,x'\in\cE_{\te\om}$  we have 
$$
h_{\te\om}(x)=(\la_\om)^{-1}\cL_{\om}h_{\om}(x)\leq \la_\om^{-1}B_{\om,0}\cL_{\om}h_{\om}(x')=B_{\om,0}h_{\te\om}(x').
$$
Since $\nu_\om(h_\om)=1$ we conclude that $\min h_{\te\om}\geq B_{\om,0}^{-1}$. 

Finally, by Lemma \ref{generating}, for every $g\in\cH_{\te^{-n}\om}$, there are $g_1,g_2$ in $\cC_{\te^{-n}\om}$ and nonnegative constants $c_1,c_2$ so that $g=g_1-c_1-(g_2-c_2)$ and 
$$
\|g_1\|+\|g_2\|+|c_1|+|c_2|\leq 8\|g\|.
$$
Thus, by applying \eqref{uu}  with $g=g_1, g=g_2$ and the constant functions $g=-c_1$ and $g=-c_2$ we see that
$$
\left\|\frac{\cL_{\te^{-n}\om}^n g}{\la_{\te^{-n}\om,n}}-\nu_{\te^{-n}\om}(g)h_\om\right\|\leq 4 K_\om\|g\|\rho_{\te^{-n}\om,n}.
$$

\subsection{Effective rates for  partially expanding maps: proof of Theorem \ref{RPF} (i)-(iii) in the setup of Section \ref{Maps2}}
\subsubsection{The cones}
Set $\ka_\om=\frac{1}{s_\om}$ and 
consider the real cone 
\[
\cC_{\om}=\cC_{\om,\ka_\om}=\{g\in\cH_\om:\,g>0\,\text{ and }\,v(g)\leq\ka_\om\inf g\}.
\]

\begin{lemma}\label{Inc1}
We have
\begin{equation}\label{Inclusion}
 \cL_\om\cC_{\om}\subset\cC_{\te\om,\zeta_\om\ka_{\te\om}}\subset\cC_{\te\om}
\end{equation}
where\footnote{The fact that $\zeta_\om<1$ follows from the condition on $H_\om$ in \eqref{Bound ve}.} $\zeta_\om=\frac{s_\om\ka_\om+(1+\ka_\om)e^{\ve_\om}H_\om}{\ka_{\te\om}}=s_{\te\om}\left(1+(1+s_\om^{-1})e^{\ve_\om}H_\om\right)<1$. 
\end{lemma}

\begin{proof}[Proof of Lemma \ref{Cont1}] 
The proof of \eqref{Inclusion} proceeds similarly to the proof of  \cite[Theorem 5.1]{castro}. 
Let $g\in\cC_{\om}=\cC_{\om,\ka_\om}$.
Fix some $\om$ and two points $x,y$ in $\cE_{\te\om}$ and denote by $(x_i)$ and $(y_i)$ their inverse images under $T_\om$, respectively. Then 
\begin{eqnarray*}
\frac{|\cL_\om g(x)-\cL^{(j)}g(y)|}{\inf\cL_\om g}\leq
\frac{|\cL_\om g(x)-\cL_\om g(y)|}{d_\om e^{\inf\phi_\om}\inf g}\leq
d_\om^{-1}\sum_{i=1}^{d_\om}e^{\phi_\om(x_i)-\inf\phi_\om}|g(x_i)-g(y_i)|(\inf g)^{-1}\\+d_\om^{-1}\sum_{i=1}^{d_\om}|(g(y_i)/\inf g)e^{-\inf\phi_\om}|e^{\phi_\om(x_i)}-e^{\phi_\om(y_i)}|:=I_1+I_2.
\end{eqnarray*}
Next, since $g\in\cC_\om$ for each $i$ we have $|g(x_i)-g(y_i)|\leq v(g)\rho^\al(x_i,y_i)\leq\ka_\om\inf g\cdot\rho^\al(x_i,y_i)$. 
Moreover, we have 
 $\phi_\om(x_i)-\inf\phi_\om\leq\sup\phi_\om-\inf\phi_\om=\ve_\om$.
 Combining these estimates and taking into account \eqref{Pair1}, \eqref{Pair2} and \eqref{Pair3} we get that
$$
I_1\leq \ka_\om\rho^\al(x,y)e^{\ve_\om}d_\om^{-1}(L_\om^\al q_\om+(d_\om-q_\om)\sig_\om^{-\al})=\rho^\al(x,y)s_\om\ka_\om
$$
where we recall that $s_\om$ was defined in \eqref{Bound ve}.

In order to bound $I_2$, we first observe that  $\sup g\leq\inf g+v(g)\leq (1+\ka_\om)\inf g$ and that  by the definition \ref{H def} of the local H\"older constant $H_\om$ and the mean value theorem we see that
\[
|e^{\phi_\om(x_i)}-e^{\phi_\om(y_i)}|\leq e^{\max(\phi_\om(x_i),\phi_\om(y_i))}|\phi_\om(x_i)-\phi_\om(y_i)|
\leq e^{\inf\phi_\om+\ve_\om} H_\om\rho^\al(x,y).
\]
Using these estimates we obtain that 
\[
I_2\leq \rho^{\alpha}(x,y)(1+\ka_\om)e^{\ve_\om}H_\om.
\]
We conclude that 
\[
v(\cL_\om g)\leq \left(s_\om\ka_\om+(1+\ka_\om)e^{\ve_\om}H_\om\right)\inf \cL_\om g=
\zeta_\om\ka_{\te\om}\inf \cL_\om g
\]
and therefore
\begin{equation}\label{Real inv}
\cL_\om\cC_{\om,\ka_\om}\subset\cC_{\te\om,\zeta_\om\ka_{\te\om}}\subset\cC_{\te\om,\ka_{\te\om}}=\cC_{\te\om}.
\end{equation}
\end{proof}

Next,

 \begin{corollary}\label{Cor diam 1}
The projective diameter of $\cL_{\om}\cC_{\om}$ inside $\cC_{\te\om}$ does not exceed 
 $$
D(\om):=2\ln\left(\frac{1+\zeta_\om}{1-\zeta_\om}\right)+2\ln\left(1+\zeta_\om\ka_\om\right).
 $$
 \end{corollary}
 \begin{proof}
See \cite[Section 4]{castro} or \cite[Section 5]{Varandas} (recalling our assumption that $\text{diam}(\cE_\om)\leq 1$).
 \end{proof}
 
\subsubsection{Reconstruction of $\nu_\om$ using dual cones}

 Let $\cC_\om^*$ be the dual cone  of $\cC_\om$.
Let $\cL_\om^*:\cH_{\te\om}^*\to\cH_{\om}^*$ be the dual operator. Then, as explained in Section \ref{nu sec 1}, the projective diameter of $\cL_\om^* \cC_{\te\om}^*$ inside $\cC_\om^*$ equals the the projective  $\cL_{\om}\cC_{\om}$ inside $\cC_{\te\om}$ (which does not exceed $D(\om)$). 

Next, we need the following result.

\begin{lemma}\label{Dual aper1}
For every $\mu\in \cC_\om^*$ and all $h\in\cH$ we have
$$
|\mu(h)|\leq k_\om\|h\|\mu(\textbf{1})
$$
where $k_\om=\frac{1+\ka_\om}{\ka_\om}=1+s_\om\leq 2$ (and $\ka_\om=\frac1{s_\om}$).
\end{lemma}

\begin{proof}
As in the proof of Lemma \ref{Dual aper}, it is enough to show that a closed ball or radius $\frac{\ka_\om}{1+\ka_\om}$ around $\textbf{1}$ is contained in $\cC_\om$.
Indeed, let $h=1+f$ where $\|f\|<1$. Then $\inf (1+f)\geq 1-\|f\|_\infty\geq 1-\|f\|$ and $v(1+f)=v(f)\leq \|f\|$. Hence $h\in\cC_\om$ if 
$$
\|f\|\leq \ka_\om(1-\|f\|).
$$ 
\end{proof}

Finally, let $\rho(\om)=\tanh (D(\om)/4)$. Let $\mu\in\cC_{\te^n\om}^*$ and $\nu\in\cC_{\te^m\om}^*$ for some $m\geq n$. Then the projective distance between $(\cL_\om^n)^*\mu$ and $(\cL_\om^m)^*\nu=(\cL_\om^n)^*(\cL_{\te^n\om}^{m-n})^*\nu$ does not exceed $\rho_{\om,n}=\prod_{j=0}^{n-1}\rho(\te^j\om)$. 
Thus, as in Section \ref{nu sec 1} we conclude that
$$
\left\|\frac{(\cL_\om^n)^*\mu}{\mu(\cL_\om^n\textbf{1})}-\frac{(\cL_\om^m)^*\nu}{\nu(\cL_\om^m\textbf{1})}\right\|\leq \sqrt 2k_\om\rho_{\om,n}.
$$
Thus, for any sequence $\mu_n$ so that $\mu\in\cC_{\te^n\om}^*$ the limit
$$
\nu_\om=\lim_{n\to\infty}\frac{(\cL_\om^n)^*\mu}{\mu(\cL_\om^n\textbf{1})}
$$
exists, belongs to $\cC_\om^*$ and it does not depend on the choice of the sequence. Moreover, by fixing $n$ and letting $m\to\infty$ we have 
$$
\left\|\frac{(\cL_\om^n)^*\mu}{\mu(\cL_\om^n\textbf{1})}-\nu_\om\right\|\leq \sqrt 2k_{\om}\rho_{\om,n}.
$$
Note that $\nu_\om(\textbf{1})=1$.
Moreover, as in Section \ref{nu sec 1} there is a number $\la_\om$ so that $\cL_\om^*\nu_{\te\om}=\la_\om\nu_\om$, where $\la_\om=\nu_{\te\om}(\cL_\om \textbf{1})$. Furthermore, $\nu_\om$ is a probability measure.

\subsubsection{Reconstruction of $h_\om$ with effective rates}
We first need the following result.

 \begin{lemma}
 For every $g\in\cC_{\om}=\cC_{\om,\ka_\om}$  we have 
 $$
 \|g\|\leq K_\om\inf g\leq K_\om\nu_\om(g)
 $$
 where 
 $$
 K_\om=1+2\ka_\om.
 $$
 \end{lemma}
 
 \begin{proof}
 First, since $g\in\cC_\om$ we have 
 $$
 \|g\|=\sup g+v(g)\leq \sup g+\ka_\om\inf g.
 $$
 Second, in order to estimate $\sup g$, using that $v(g)\leq \ka_\om\inf g$ we see that for every $x\in\cE_\om$ we have $|g(x)-\inf g|\leq v(g)\text{diam}(\cE_\om)^\al\leq \ka_\om\text{diam}(\cE_\om)^\al\inf g$. Taking into account that  $\text{diam}(\cE_\om)\leq 1$ we conclude that $\sup g\leq (1+\ka_\om)\inf g$, which completes the proof of the lemma.
 \end{proof}

The next result we need is the following.
 
 \begin{lemma}\label{generating1}
 For every $g\in\cH_\om$ there is a constant $c(g)>0$ and a function $g_1\in\cC_\om$
   so that $g=g_1-c(g)$ and 
 $$
\|g_1\|+c(g)\leq 3\|g\|.
 $$
 \end{lemma}
 
 \begin{proof}
 Let  $c(g)=v(g)/\ka_\om+\sup|g|\leq \|g\|$.   
  Then $g_1=g+c(g)\in\cC_{\om}$ and so $g=g_1-c_g$ and  since $\ka_\om\geq1$ we have
 \[
\|g_1\|+c(g)=\|g+c(g)\|+c(g)\leq \|g\|+2c(g)\leq 3\|g\|.
 \]
 \end{proof}
 
By repeating the arguments in Section \ref{h sec} we see that for every $m,n\in\bbN$ such that $m\geq n$ and all $g\in\cC_{\te^{-n}\om}$ and $f\in\cC_{\te^{-m}\om}$ we have
$$
\left\|\frac{\cL_{\te^{-n}\om}^n g}{\nu_{\te^{-n}\om}(\cL_{\te^{-n}\om}^n g)}-\frac{\cL_{\te^{-m}\om}^n g}{\nu_{\te^{-m}\om}(\cL_{\te^{-m}\om}^nf)}\right\|\leq \frac 12K_\om\rho_{\te^{-n}\om,n}.
$$
Moreover, for any sequence $(g_n)$ so that $g_n\in\cC_{\te^{-n}\om}$ the limit 
$$
h_\om=\lim_{n\to\infty}\frac{\cL_{\te^{-n}\om}^n g_n}{\nu_{\te^{-n}\om}(g_n)\la_{\te^{-n}\om,n}}
$$
exists, it does not depend on $(g_n)$, it belongs to  $\cC_\om$ and $\nu_\om(h_\om)=1$. Furthermore, $\cL_\om h_\om=\la_\om h_{\te\om}$ and
$$
\left\|\frac{\cL_{\te^{-n}\om}^n g}{\la_{\te^{-n}\om,n}}-\nu_{\te^{-n}\om}(g)h_\om\right\|\leq \frac12 K_\om\nu_{\te^{-n}\om}(g)\rho_{\te^{-n}\om,n}.
$$
In addition, since $h_\om\in \cC_\om$ we have $\sup h_\om\leq B_{\om,1}\inf h_\om$ where 
$B_{\om,1}=1+\ka_\om$ and we have used that $\text{diam}(\cE_\om)\leq1$.
Since $\nu_\om(h_\om)=1$ we conclude that $\min h_{\om}\geq B_{\om,1}^{-1}$. 

Finally, arguing as in Section \ref{h sec}, by using Lemma \ref{generating1} instead of Lemma \ref{generating} we see that
$$
\left\|\frac{\cL_{\te^{-n}\om}^n g}{\la_{\te^{-n}\om,n}}-\nu_{\te^{-n}\om}(g)h_\om\right\|\leq 2 K_\om\|g\|\rho_{\te^{-n}\om,n}.
$$

\subsection{Decay of correlations and the normalized transfer operators: proof of Theorem \ref{RPF} (iv)-(v)}\label{SecDec}
Let the operator $L_\om$ be defined by
$$
L_\om g=\frac{\cL_\om(g h_\om)}{h_{\te\om}\la_\om}.
$$
Then, using that $\|fg\|\leq 3\|f\|\|g\|$ for every two H\"older continuous functions, we see that 
$$
\left\|L_\om^{n} g-\mu_\om(g)\right\|=\left\|\frac{\cL_\om^{n} (gh_\om)}{\la_{\om,n}h_{\te^n\om}}-\nu_\om(gh_\om)\right\|=\left\|\left(\frac 1{h_{\te^n\om}}\right)\left(\cL_\om^n(g h_\om)-\nu_{\om}(g h_{\om})h_{\te^n\om}\right)\right\|
$$
$$
\leq 3\left\|\frac1 {h_{\te^n\om}}\right\|
\left \|\frac{\cL_\om^{n} (gh_\om)}{\la_{\om,n}}-\nu_\om (gh_\om)h_{\te^n\om}\right\|.
$$
Now, since $h_\om\geq B_{\om,1}^{-1}$ we have $\|1/h_\om\|\leq v(h_\om)B_{\om,1}^{2}+B_{\om,1}\leq 2B_{\om,1}^2K_\om$. Thus, by \eqref{RPF ExpC} (and recalling the formula of $B_\om$ in Section \ref{Aux2}),
$$
\|L_\om^n g-\mu_\om(g)\|\leq B_{\te^n\om}\rho_{\om,n}.
$$
Now the decay of correlations \eqref{DEC} follows from the equality
$$
\text{Cov}_{\mu_\om}(g, f\cdot T_\om^n)=\int f\cdot (L_\om^n g-\mu_\om(g))d\mu_{\te^n\om}.
$$

\section{Random complex RPF theorems with effective rates: proof of Theorem \ref{Complex RPF}}\label{RPF CPMLX pf}

As opposed to the previous sections in this section we will begin with the setup of Section \ref{Maps2}. The reason is that in the setup of Section \ref{Maps1} the appropriate projective estimates needed to prove Theorem \ref{Complex RPF} are similar to \cite[Ch. 4-5]{HK} (with some modifications).
Henceforth, for the sake of convenience, we will always assume that $\mu_\om(u_\om)=0$, namely we will replace $u_\om$ with $\tilde u_\om=u_\om-\mu_\om(u_\om)$.

%We note that, as in \cite{Annealed}, 
%when the $\{T_\om\}$ are "sequentially non-singular" the measures $\mu_\om$ are absolutely continuous, then we have the following

%\begin{proposition}\label{NonSinThm}
%Let $\textbf{m}_j,\,j\in\bbZ$ be a family of probability measures on $\cE_j$, which assign positive mass to open sets, so that for each $j$
%we have $(T_\om)_*\textbf{m}_j\ll\textbf{m}_{j+1}$ and that $e^{-f_j}=\frac{d(T_\om)_*\textbf{m}_j}{d\textbf{m}_{j+1}}$. Then for any $j$ we have
%$\la_\om(0)=1$ and $\nu_\om^{(0)}=\textbf{m}_j$.
%\end{proposition}

\subsection{Complex cones contractions for random maps with expansion and contraction}
The proof of Theorem \ref{Complex RPF} relies on the theory of the canonical complexification of real Birkhoff cones. We will give a reminder of the appropriate  results concerning this theory  in the body of the proof of Theorem \ref{Complex cones Thm} below, and the readers are referred to \cite[Appendix A]{HK} for a summary of the main definitions and results concerning contraction properties of real and complex cones (the properties of the complex cones is essentially a summary of the appropriate results in \cite{Rugh, Dub1, Dub2}).

Let $\cC_{\om,\bbC}$ be the canonical complexification of the cone $\cC_\om$ (see \cite[Appendix A]{HK}), 
and let $\cC_{\om,\bbC}^{*}:=\{\nu\in\cH_\om^*:\,\nu(c)\not=0\,\,\,\forall\nu\in\cC_{\om,\bbC}\setminus\{0\}\}$ 
be its complex dual cone.

%Formulate a similar theorem for the first type of maps, and refer to the book for most of the details in the proof...
\begin{theorem}\label{Complex cones Thm}

(i) The cones $\cC_{\om,\bbC}$ and their duals $\cC_{\om,\bbC}^{*}$ have bounded aperture:  for all $g\in\cC_{\om,\bbC}$ and $\nu\in\cC_{\om,\bbC}^*$ and every point $x_\om\in\cE_\om$ we have
\begin{equation}\label{Aper cmplx}
\|g\|\leq Q_\om|g(x_\om)|\,\,\text{ and }\,\,\|\nu\|\leq M_\om|\nu(\textbf{1})|
\end{equation}
where $Q_\om=2\sqrt 2(1+2\ka_\om)=2\sqrt 2 K_\om$ and $M_\om=\frac{1}{6\ka_\om}$, $\ka_\om=s_\om^{-1}$

(ii) The cones $\cC_{\om,\bbC}$ are linearly convex, namely  for every $g\not\in\cC_{\om,\bbC}$
there exists $\mu\in\cC_{\om,\bbC}^*$ such that $\mu(g)=0$.

(iii) The cones $\cC_{\om,\bbC}$ are reproducing:
 for any complex-value function $g\in\cH_\om$ there is are constant $c_1(g), c_2(g)>0$ and functions $g_1,g_2\in\cC_\om\subset \cC_{\om,\bbC}$ so that   $g=g_1-c_1(g)+i(g_2-c_2(g))$ and 
 $$
\|g_1\|+c_1(g)+\|g_2\|+c_2(g)\leq 6\|g\|.
 $$

(iv) Let
$$
c_0(\om)=32\ka_{\te\om}^{-1}(1+2\ka_\om)e^{\|u_\om\|_\infty+2\|\phi_\om\|_\infty} \|u_\om\|(1+H_\om)(1-\zeta_\om)^{-1}
$$
and for all complex $z$ so that $|z|\leq1$ set  
$$
\del_\om(z)=2|z|c_o(\om)\left(1+\cosh(D(\om)/2)\right).
$$
Then, if $\del_\om(z)\leq 1-e^{-D(\om)}$ we have that
\[
\cL_\om^{(z)}\cC_{\om,\bbC}\subset\cC_{\te\om,\bbC}
\]
and the Hilbert diameter of the image with respect to the complex projective metric corresponding to the cone $\cC_{\te\om,\bbC}$ does not exceed $7D(\om)$.
\end{theorem}

\begin{proof}[Proof of Theorem \ref{Complex cones Thm}]
(i)
First (see \cite[Appendix A]{HK} and \cite[Section 5]{Rugh}) we have
\begin{equation}\label{Complexification}
\cC_{\om,\bbC}=\{g\in \cH_\om:\,\Re\big(\overline{\mu(g)}\nu(g)\big)
\geq0\,\,\,\,\forall\mu,\nu\in\cC_{\om}^*\}.
\end{equation}
%where $\cC_{j,\bbR}^*=\{\mu\in\cH_j^*:\,\mu(c)\geq 0\,\,\,\forall\,c\in\cC_{j,\bbR}\}$.

We begin with showing that the complex cones $\cC_{\om,\bbC}$ and their duals have bounded aperture. First, for every point $a\in \cE_\om$ and $g\in\cC_{\om}$ we have
\[
\|g\|=\sup g+v(g)\leq \inf g+2v(g)\leq (1+2\ka_\om)\inf g\leq (1+2\ka_\om)g(a)
\]
where we have used that $g(x)-g(y)\leq (\text{diam}(\cE_\om))^\alpha v(g)\leq v(g)$ for every real-valued function on $\cE_\om$. By applying  \cite[Lemma 5.2]{Rugh} we conclude that for every  $g\in\cC_{\om, \bbC}$ we have
\[
\|g\|\leq 2\sqrt 2(1+2\ka_\om)g(a).
\]
 Next, in order to show that the cone $\cC_{\om,\bbC}$ has bounded aperture we will apply  \cite[Lemma A.2.7]{HK} which states that 
$$
\|\nu\|\leq M_\om\nu(\textbf{1}),\,\,\forall\,\nu\in\cC_{\om,\bbC}^*
$$
if the complex cone $\cC_{\om,\bbC}$ contains the ball of radius $1/M_\om$ around the constant function $\textbf{1}$. The first step in showing that such a ball exists is the following representation of the real cone:
\[
\cC_{\om}=\{g\in\cH_\om: s(g)\geq 0,\,\forall\,\,\, s\in \Gamma_\om\}
\]
where $\Gamma_\om\subset \cH_\om^*$ is the class of linear functionals $s$ which either have the form $s(g)=s_a(g)=g(a)$ for some $a\in\cE_\om$ or have the form 
\[
s=s_{x,y,t,\ka_\om}(g)=\ka_\om g(t)-\frac{g(x)-g(y)}{\rho^\al(x,y)}
\]
for some $x,y,t\in\cE_\om$ so that $x\not=y$.
Then (see \cite[Appendix A]{HK}), since $\Gamma_\om$ generates the dual cone $\cC_{\om}^*$, the cannonical complexification of $\cC_\om$ can be written in the following form: 
 \begin{equation}\label{complexification}
\cC_{\om,\bbC}=\{x\in\cH_{\om,\bbC}:\,\Re\big(\overline{\mu(x)}\nu(x)\big)\geq0,\,\,\,\,\,\forall\mu,
\nu\in \Gamma_\om\}.
\end{equation}
Using (\ref{complexification}), it is enough to show that  for all $g$ of the form $g=\textbf{1}+h$ with $\|h\|\leq \bar\ve_\om:=\frac1{M_\om}$, and every $s_1,s_2\in \Gamma_\om$ we have
$$
\Re(s_{1}(g)\cdot\overline{s_2(g)})\geq 0.
$$
Notice that $s_i(g)=1+s_i(h)$ and so 
$$
\Re(s_{1}(g)\cdot\overline{s_2(g)})\geq 1-|s_1(h)|-|s_2(h)|-|s_1(h)s_2(h)|.
$$
Now, there are four cases. When $s_i(g)=g(a_i)$ for some $a_i\in\cE_\om$ then 
$$
\Re(s_{1}(g)\cdot\overline{s_2(g)})\geq 1-2\|h\|_\infty-\|h\|_\infty^2>0
$$
since $\|h\|\leq\bar\ve_\om<\frac13$. Let us suppose next that $s_1(g)=g(a)$ and $s_2(g)=\ka_\om g(t)-\frac{g(x)-g(y)}{\rho^\al(x,y)}$. Then 
$$
\Re(s_{1}(g)\cdot\overline{s_2(g)})\geq 1-\ka_\om\|h\|_\infty-v(h)-\|h\|_\infty^2\ka_\om-\|h\|_\infty v(h)\geq 1-2\|h\|(\ka_\om+1)
$$
where we have used that $\|h\|\leq1$. Notice that the above right hand side is nonnegative if $\|h\|\leq \frac 1{M_\om}=\frac1{6\ka_\om}$ since $\ka_\om\geq1$. A similar inequality holds true when $s_2(g)=g(a)$ and $s_1(g)=\ka_\om g(t)-\frac{g(x)-g(y)}{\rho^\al(x,y)}$.
Let us assume now that $s_i(g)=\ka_\om g(t_i)-\frac{g(x_i)-g(y_i)}{\rho^\al(x_i,y_i)}$ for appropriate choices of $t_i,x_i,y_i$, $i=1,2$. Then 
$$
\Re(s_{1}(g)\cdot\overline{s_2(g)})\geq 1-2(\ka_\om\|h\|_\infty+v(h))-\left(\ka_\om\|h\|_\infty+v(h)\right)^2$$$$\geq 1-2\ka_\om\|h\|-(\ka_\om\|h\|)-2\ka_\om\|h\|-\|h\|\geq 1-6\ka_\om\|h\|
$$
where in the last two inequalities we have used that $\ka_\om\geq1$ and $\|h\|\leq \ka_\om^{-1}\leq 1$.
The above left hand side is nonnegative since $\|h\|\leq \bar\ve_\om=M_\om^{-1}=\frac{1}{6\ka_\om}$.

\vskip0.1cm

(ii) By \cite[Lemma 4.1]{Dub2}, in order to show that the cone $\cC_{\om,\bbC}$ is linearly convex it is enough to show that there there is a continues linear functional $\ell$ which is strictly positive on $\cC_\om\setminus\{0\}$. Clearly we can take $\ell(g)=g(a)$ for an arbitrary point $a$ in $\cE_\om$.
\vskip0.1cm

 (iii) This is a direct consequence of Lemma \ref{generating1} applied with the real and imaginary parts of $g$.
 \vskip0.1cm
 
(iv) Recall first that  by Lemma \ref{Cont1}, for all $\om$,
\begin{equation}\label{Inclusion0}
 \cL_\om\cC_{\om}\subset\cC_{\te\om,\zeta_\om\ka_{\te\om}}.
\end{equation}

We will next prove that  for every $s\in\Gamma_{\te\om}$, $g\in\cC_{\om}$ (the real cone) and a complex number $z$ so that $|z|\leq 1$ we have 
\begin{equation}\label{Comparison Key}
\left|s(\cL_\om^{(z)}g)-s(\cL_\om g)\right|\leq c_0(\om)|z| s(\cL_\om g).
\end{equation}
After this is established we can apply \cite[Theorem A.2.4]{HK}  and obtain item (iv).

Let us first consider the case when $s(f)=f(a)$ for some $a\in\cE_{\te\om}$. Then 
$$
\left|s(\cL_\om^{(z)}g)-s(\cL_\om g)\right|=\left|\cL_\om\big(g(e^{zu_\om}-1)\big)(a)\right|\leq \|e^{zu_\om}-1\|_\infty \cL_\om g(a)= \|e^{zu_\om}-1\|_\infty s(\cL_\om g)
$$
$$
\leq |z|\|u_\om\|_\infty e^{\|u_\om\|_\infty}s(\cL_\om g)\leq c_0(\om)|z|s(\cL_\om g).
$$
Next, let us consider the case when $s=s_{x,y,t,\ka_{\te\om}}$.
We first need the following simple result/observation: let $A$ and $A'$ be complex numbers, $B$ and $B'$ be real numbers, and let $\ve_1>0$ and $\zeta\in(0,1)$ be so that
\begin{itemize}
\item
$B>B'$
\item
$|A-B|\leq\ve_1B$
\item
$|A'-B'|\leq\ve_1 B$
\item
$|B'/B|\leq\zeta$.
\end{itemize}
Then 
\[
\left|\frac{A-A'}{B-B'}-1\right|\leq 2\ve_1(1-\zeta)^{-1}.
\]
The proof of this result is elementary, just write
\[
\left|\frac{A-A'}{B-B'}-1\right|\leq\left|\frac{A-B}{B-B'}\right|+
\left|\frac{A'-B'}{B-B'}\right|\leq \frac{2B\ve_1}{B-B'}=\frac{2\ve_1}{1-B'/B}.
\]

Next, fix some nonzero $g\in\cC_{\om}$ and $(x,y,t)\in\Del_{\te\om}$. Then, in order to obtain \eqref{Comparison Key} when  $s=s_{x,y,t,\ka_\om}$ we need to show that the conditions of the above result hold true with $A=\ka_{\te\om}\cL_\om^{(z)} g(t)$, 
\begin{equation*}
 B=\ka_{\te\om}\cL_\om g(t),\,\,\, A'=\frac{\cL_\om^{(z)} g(x)-\cL_\om^{(z)} g(y)}{\rho^\al(x,y)},
\,\,\,\,B'=
\frac{\cL_\om g(x)-\cL_\om g(y)}{\rho^\al(x,y)}
\end{equation*}
and $\zeta=\zeta_\om$ and $\ve_1=16\ka_{\te\om}^{-1}(1+2\ka_\om)(1+H_\om)e^{\|u_\om\|_\infty+2\|\phi_\om\|_\infty}\|u_\om\||z|$.

We begin by noting that $B>B'$ since the function $\cL_\om g$ is a nonzero member of the cone $\cC_{\om,\bbR,\zeta_\om\ka_{\te\om}}$. In fact, this already implies that
\[
|B'/B|\leq \zeta_\om\inf\cL_\om g/B\leq\zeta_\om<1.
\]
Next, notice that when $|z|\leq 1$ we have
\begin{eqnarray*}
|A-B|=\ka_{\te\om}|\cL_\om^{(z)} g(t)-\cL_\om g(t)|=\ka_{\te\om}|\cL_\om(g(e^{zu_\om}-1))(t)|\\\leq 
\ka_{\te\om}\|e^{zu_\om}-1\|_\infty\cL_\om g(t)\leq |z|e^{\|u_\om\|_\infty}\|u_\om\|_\infty B.
\end{eqnarray*}

Finally, let us estimate the difference $|A'-B'|$. For each $a,b\in\cE_\om$ we define
\[
\Del_{a,b}(z)=e^{\phi_\om(a)}(e^{zu_\om(a)}-1)g(a)-e^{\phi_\om(b)}(e^{zu_\om(b)}-1)g(b).
\] 
Denote again by $x_i$ and $y_i$ the preimages of $x$ and $y$ under $T_\om$, respectively, where $1\leq i\leq d_\om$. Then 
\[
\rho^\alpha(x,y)(A'-B')=\sum_{i=1}^{d_\om}\Del_{x_i,y_i}(z).
\]
Next, by using the mean value theorem we see that
\[
|\Del_{x_i,y_i}(z)|=|\Del_{x_i,y_i}(z)-\Del_{x_i,y_i}(0)|\leq|z|\sup_{|q|\leq |z|}|\Del'_{x_i,y_i}(q)|.
\]
In order to estimate the above derivative, first note that
\[
|e^{\phi_\om(x_i)}-e^{\phi_\om(y_i)}|\leq (e^{\phi_\om(x_i)}+e^{\phi_\om(y_i)})H_\om\rho^\al(x,y).
\]
Therefore, for every complex $q$ so that $|q|\leq1$ and all $1\leq i\leq d_\om$,
\[
|\Del'_{x_i,y_i}(q)|\leq 8(1+H_\om)e^{\|u_\om\|_\infty}\|u_\om\|(e^{\phi(x_i)}+e^{\phi(y_i)})\|g\|\rho^\al(x,y).
\]
 We conclude that   for every $z\in\bbC$ with $|z|\leq1$,
$$
|A'-B'|\leq 8|z|(1+H_\om)e^{\|u_\om\|_\infty}\|u_\om\|\|g\|(\cL_\om\textbf{1}(x)+\cL_\om\textbf{1}(y)).
$$
Now, since $g\in\cC_\om$ we have $\|g\|\leq r_\om \inf g$, where $r_\om=(1+2\ka_\om)$ and
$$
\inf g\leq d_\om^{-1}e^{\|\phi_\om\|_\infty}\cL_\om g(t)=\ka_{\te\om}^{-1}d_\om^{-1}e^{\|\phi_\om\|_\infty}B.
$$
Using also that   $\cL_\om \textbf{1}\leq d_\om e^{\|\phi_\om\|_\infty}$
we see that
$$
|A'-B'|\leq 16\ka_{\te\om}^{-1}(1+2\ka_\om)(1+H_\om)e^{\|u_\om\|_\infty+2\|\phi_\om\|_\infty}\|u_\om\||z|B.
$$
We thus conclude that we can take 
$$
\zeta=\zeta_\om\,\,\text{ and }\,\ve_1=16\ka_{\te\om}^{-1}(1+2\ka_\om)(1+H_\om)e^{\|u_\om\|_\infty+2\|\phi_\om\|_\infty}\|u_\om\||z|
$$
in the above general result, which completes the proof of \eqref{Comparison Key} in the case $s=s_{x,y,t,\ka_{\te\om}}$ since $c_\om=2\ve_1(1-\zeta)^{-1}$.
\end{proof}

\subsubsection{A complex RPF theorem with effective rates}\label{Cmplx pf}
\begin{proof}[Proof of Theorem \ref{Complex RPF} in the setup of Section \ref{Maps2}]
Set 
$$
\del(\om)=2c_0(\om)\left(1+\cosh(D(\om)/2)\right)=2E_\om.
$$
Then $\del_\om(z)=|z|\del(\om)$ and the assumptions in Theorem \ref{Complex RPF} insure that
$$
A:=\text{esssup}\left(\del(\om)(1-e^{-D(\om)})^{-1}\right)<\infty.
$$
Hence the condition $\del_\om(z)\leq 1-e^{-D(\om)}$ holds true when $|z|\leq 1/A$. Relying on Theorem \ref{Complex cones Thm}, proceeding as in the 
 of the proof of the real RPF theorem (Theorem \ref{RPF}),  we see that there is a constant $r_0$ so that $\bbP$-a.s. for every complex number $z$ so that $|z|\leq r_0$ there is a triplet consisting of a nonzero complex random variable $\la_\om(z)$, a random function $\hat h_\om^{(z)}\in\cC_{\om,\bbC}$ and a random linear functional $\nu_\om^{(z)}\in\cC^*_{\om,\bbC}$ so that $\nu_\om^{(z)}(\textbf{1})=\nu_\om(\hat h_\om^{(z)})=1$, 
$$
(\cL_{\om}^{(z)})^{*}\nu_{\te\om}^{(z)}=\la_\om(z)\nu_\om^{(z)}
$$
and for every $\mu\in\cC_\om$
$$
\left\|\frac{(\cL_\om^{z,n})^*\mu}{\mu(\cL_\om^{z,n}\textbf{1})}-\nu_\om^{(z)}\right\|\leq M_\om\tilde\rho_{\om,n}
$$
where $\tilde\rho(\om)=\tanh(7D(\om)/4)$. Moreover,  for every $g\in\cC_{\te^{-n}\om,\bbC}$ we have
\begin{equation}\label{Exp Temp}
\left\|\frac{\cL_{\te^{-n}\om}^{z,n} g}{\nu_\om(\cL_{\te^{-n}\om}^{z,n} g)}-\hat h_\om^{(z)}\right\|\leq \sqrt 2 K_\om\tilde \rho_{\te^{-n}\om,n}.
\end{equation}
Since $\nu_\om^{(z)}$ and $\hat h_\om^{(z)}$ are uniform limits (in $z$) of analytic in $z$ measurable functions they are analytic in $z$ and measurable in $\om$. 
Similarly, also $\la_\om(z)$ is analytic in $z$. Since $\nu_\om^{(z)}(\textbf{1})=1$ we conclude from \eqref{Aper cmplx} that $\|\nu_\om^{(z)}\|\leq M_\om$.  Moreover, since $\nu_\om(\hat h_\om^{(z)})=1$ we conclude from \eqref{Aper cmplx} that $\|\hat h_\om^{(z)}\|\leq 2\sqrt 2K_\om$.
 It is also clear that $\la_\om(0)=\la_\om$, $\nu_\om^{(0)}=\nu_\om$ and $\hat h_\om^{(0)}=h_\om$. To correct the fact that $\nu_\om^{(z)}(\hat h_\om^{(z)})$ might not equal $1$ (notice that it does not vanish since $\nu_\om^{(z)}$ belongs to the dual cone) let us define 
$$
h_\om^{(z)}=\frac{\hat h_\om^{(z)}}{\al_\om(z)}
$$
where $\al_\om(z)=\nu_\om^{(z)}(\hat h_\om^{(z)})$. Notice that $\al_\om(0)=1$, $\al_\om(z)$ is analytic in $z$ and 
$$
|\al_\om(z)|\leq \|\nu_\om^{(z)}\|\|\hat h_\om^{(z)}\|\leq 2\sqrt 2 M_\om K_\omega.
$$

 Let us now obtain \eqref{Exponential convergence}, which in particular will yield that  $\cL_\om^{(z)}(h_\om^{(z)})=\la_\om(z)h_{\te\om}^{(z)}$. Let us first prove a version of \eqref{Exponential convergence} for functions in the cone $\cC_\om$. 
First, for every complex number $z$ such that $|z|\leq r_0$
and $q\in \cC_{\te^{-n}\om}$ we have
\[
\la_{\te^{-n}\om,n}(z)\nu_{\te^{-n}\om}^{(z)}(q)=
\nu_\om^{(z)}(\cL_{\te^{-n}\om}^{z,n}q).
\]
Next, set
$$
b_n(q,z)=b_n(\om,q,z)=
\frac{\nu_\om(\cL_{\te^{-n}\om}^{z,n}q)}{\nu_\om^{(z)}(\cL_{\te^{-n}\om}^{z,n}q)}.
$$
Then, 
$$
\left\|\frac{\cL_{\te^{-n}\om}^{z,n}q}{\la_{\te^{-n}\om,n}(z)\nu_{\te^{-n}\om}^{(z)}(q)}
-h_\om^{(z)}\right\|=
\left\|b_n(q,z)\cdot\frac{\cL_{\te^{-n}\om}^{z,n}q}{\nu_\om(\cL_{\te^{-n}\om}^{z,n}q)}-h_\om^{(z)}\right\|
$$
$$
\leq\left\|b_n(q,z)\left(\frac{\cL_{\te^{-n}\om}^{z,n}q}{\nu_\om(\cL_{\te^{-n}\om}^{z,n}q)}-\hat h_\om^{(z)}\right)
\right\|+
\left\|\left(b_n(q,z)-\frac1{\al_\om(z)}\right) \hat h_\om^{(z)}\right\|:=I_1+I_2
$$
where we have used the above definition of $h_\om^{(z)}$. Now, by \eqref{Exp Temp} we have
$$
\left|(b_n(q,z))^{-1}-\al_\om(z)\right|\\=
\left|\nu_\om^{(z)}\left(\frac{\cL_{\te^{-n}\om}^{z,n}q}{\nu_\om(\cL_{\te^{-n}\om}^{z,n}q)}\right)-\nu_\om^{(z)}(\hat h_\om^{(z)})\right|
$$
$$
\leq\|\nu_\om^{(z)}\|
\left\|\frac{\cL_{\te^{-n}\om}^{z,n}q}{\nu_\om(\cL_{\te^{-n}\om}^{z,n}q)}-\hat h_\om^{(z)}\right\|\leq
\sqrt 2 M_\om K_\om \tilde\rho_{\te^{-n}\om,n}.
$$
Hence, if $n$ satisfies that $\sqrt 2 M_\om K_\om \tilde\rho_{\te^{-n}\om,n}<\frac12|\al_\om(z)|$ (which, since $z\to \al_\om(z)$ is analytic and non-vanishing, is true $\bbP$-a.s. for every $n$ large enough uniformly in $z$) then 
\begin{equation}\label{bnq}
\left|b_n(q,z)-\frac1{\al_\om(z)}\right|\leq \frac{2\sqrt 2 M_\om K_\om \tilde\rho_{\te^{-n}\om,n}}{|\al_\om(z)|^2}\leq |\al_\om(z)|^{-1}.
\end{equation}
Combing this with the upper bound $\|\hat h_\om^{(z)}\|\leq 2\sqrt 2K_\om$ we conclude that for such $n$'s we have 
$$
I_2\leq \frac{8M_\om K_\om^2 \tilde\rho_{\te^{-n}\om,n}}{|\al_\om(z)|^2}.
$$
Using now \eqref{Exp Temp} together with \eqref{bnq}
and that  $|\al_\om(z)|\leq  2\sqrt 2 M_\om K_\omega$ we see that for $n$ satisfying the above properties we have 
$$
I_1\leq |b_n(q,z)|\sqrt 2 K_\om\rho_{\te^{-n}\om,n}\leq2|\al_\om(z)|^{-1} \left(\sqrt 2 K_\om\tilde\rho_{\te^{-n}\om,n}\right).
$$
By combining the above estimates on $I_1$ and $I_2$ we see that
$$
\left\|\frac{\cL_{\te^{-n}\om}^{z,n}q}{\la_{\te^{-n}\om,n}(z)\nu_{\te^{-n}\om}^{(z)}(q)}
-h_\om^{(z)}\right\|\leq\left(2\sqrt 2 K_\om|\al_\om(z)|^{-1}+8M_\om K_\om^2|\al_\om(z)|^{-2}\right)\tilde\rho_{\te^{-n}\om,n}:=R(\om,n,z).
$$
Finally, using that $\|\nu_{\te^{-n}\om}^{(z)}\|\leq M_{\te^{-n}\om}$ we conclude that for every $g\in\cC_{\te^{-n\om},\bbC}$ we have 
$$
\left\|\frac{\cL_{\te^{-n}\om}^{z,n}q}{\la_{\te^{-n}\om,n}(z)}
-h_\om^{(z)}\right\|\leq M_{\te^{-n}\om} R(\om,n,z).
$$
Now the proof of \eqref{Exponential convergence} is completed by using  the reproducing property stated in Theorem \ref{Complex cones Thm} (iii).

\end{proof}

\subsection{Complex cones for properly expanding maps}
In this section we will briefly explain how to prove Theorem \ref{Complex RPF} in the setup of Section \ref{Maps1}. Since this is completed similarly to the proof in the setup of Section \ref{Maps2} (using ideas from \cite[Ch.5]{HK}) we will formulate the results concerning complex cones without their proofs. 

We suppose here that \eqref{phi cond} holds true and that $u_\om$ satisfies $v(u_\om)\leq H_\om$ with some $H_\om$ so that
$$
\gamma_{\om}^{-\al}v(u_\om)+H_\om\leq \gamma_{\te\om}^\al-1.
$$
Then, by replacing $\phi_\om$ with $\phi_{\om,t}=\phi_\om+tu_\om$ all the results concerning real cones hold true for $\phi_{\om,t}$ when $t\in[-1,1]$, with $Z_\om=\gamma_{\om}^{-\al} v(u_\om)+H_\om$ instead of $H_\om$ and with $\|\phi_\om\|_\infty+\|u_\om\|_\infty$ instead of $\|\phi_\om\|_\infty$.

We will need the following result, whose proof proceeds essentially as in \cite[Ch. 5]{HK}.

\begin{theorem}\label{Complex cones Thm0}

(i) The cones $\cC_{\om,\bbC}$ and their duals $\cC_{\om,\bbC}^{*}$ have bounded aperture:  for all $g\in\cC_{\om,\bbC}$ and $\nu\in\cC_{\om,\bbC}^*$ we have
\[
\|g\|\leq 2\sqrt 2K_\om |\nu_\om(g)|\,\,\text{ and }\,\,\|\nu\|\leq M_\om|\nu(\textbf{1})|
\]
where  $M_\om$ and $K_\om$ were defined in Section \ref{Aux1}.

(ii) The cones $\cC_{\om,\bbC}$ are linearly convex, namely  for every  $g\not\in\cC_{\om,\bbC}$
there exists $\mu\in\cC_{\om,\bbC}^*$ such that $\mu(g)=0$.

(iii) The cones $\cC_{\om,\bbC}$ are reproducing:
 for every complex-valued function $g\in\cH_\om$ there  are constants $c_1(g), c_2(g)>0$ and functions $g_1,g_2\in\cC_\om\subset \cC_{\om,\bbC}$ so that  $g=g_1-c_1(g)+i(g_2-c_2(g))$ and 
 $$
\|g_1\|+c_1(g)+\|g_2\|+c_2(g)\leq 2r_\om\|g\|
 $$
 where $r_\om=8\left(1+\frac{2}{\gamma_\om^\al}\right)\leq 24$.

(iv) Let $c_0(\om)$  be defined as in Section \ref{Aux1}
 and let 
 $$
\tilde D(\om)=\gamma_{\te\om}^\al+2\ln\left(\frac{1+\tilde q(\om)}{1-\tilde q(\om)}\right)
 $$
 where $\tilde q(\om)=\frac{\tilde H_\om+1}{\gamma_{\te\om}^\al}$,
Then, if $\del_\om(z)\leq 1-e^{-\tilde D(\om)}$ we have that
\[
\cL_\om^{(z)}\cC_{\om,\bbC}\subset\cC_{\te\om,\bbC}
\]
and the Hilbert diameter of the image with respect to the complex projective metric corresponding to the cone $\cC_{\te\om,\bbC}$ does not exceed $7D(\om)$, with $ D(\om)$ defined in Section \ref{Aux1}.
\end{theorem}
Relying on Theorem \ref{Complex cones Thm0} the proof of Theorem \ref{Complex RPF} in the setup of Section \ref{Maps1} proceeds exactly as in Section \ref{Cmplx pf}.

\begin{remark}
Notice that $\tilde D(\om)\geq1$ and so $1-e^{-\tilde D(\om)}\geq 1-e^{-1}$. Hence, if 
$$
E_\om=c_0(\om)\left(1+\cosh(\tilde D(\om)/2)\right)
$$
is a bounded random variable then there is a constant $r_0>0$ so that the condition $\del_\om(z)\leq 1-e^{-\tilde D(\om)}$ holds true when $|z|\leq r_0$. Notice also that $\cosh(\tilde D(\om)/2)\leq e^{\tilde D(\om)/2}$.
\end{remark}

\subsection{The ``normalized"  complex operators}
In this section we will prove \eqref{Exponential convergence CMPLX} relying on \eqref{Exponential convergence}. Let us consider the operators $L_\om^{(z)}$ given by
$$
L_\om^{(z)} g=\frac{\cL_\om^{(z)}(g h_\om)}{h_{\te\om}\la_\om}.
$$
Then 
$$
\left\|\frac{L_\om^{z,n} g}{\bar\la_{\om,n}(z)}-\mu^{(z)}_\om(g)\bar h_\om^{(z)}\right\|=\left\|\frac{\cL_\om^{z,n} (gh_\om)}{\la_{\om,n}(z)h_{\te^n\om}}-\nu_\om^{(z)}(gh_\om)\frac{h_{\te^n\om}^{(z)}}{h_{\te^n\om}}\right\|
$$
$$
\leq 3\left\|\frac1 {h_{\te^n\om}}\right\|
\left \|\frac{\cL_\om^{z,n} (gh_\om)}{\la_{\om,n}(z)}-\nu_\om^{(z)}(gh_\om)h_{\te^n\om}^{(z)}\right\|.
$$
Now, as in Section \ref{SecDec} we have $3\|1/h_\om\|\leq U_\om$ 
and thus \eqref{Exponential convergence CMPLX}  follows from \eqref{Exponential convergence}.

%\subsection{Other applications}
%analyticity of the pressure? what else?
%are there non-trivial things we can prove (say, Haussdorf dim etc. ) which will have a better form because of the effective rates.

%\section{BE and LLT}
%Two cases, uniformly random...
%for LLT need L_\om\leq1 or the properly expanding case.

%\section{Relations with large deviations Davor}

\end{document}